\documentclass[12pt]{article}                
\usepackage{graphicx, xcolor}
\usepackage{psfrag}
\usepackage{amsmath, amssymb, amsthm}
\usepackage{mathptmx}
\usepackage{esvect, bm, bbm}
\usepackage{hyperref}
\usepackage{cleveref}
\usepackage{enumitem}

\hypersetup{
	colorlinks=true,
	linkcolor={blue!80!black},
	citecolor={green!50!black},
	urlcolor={red!50!black}
}

\newcommand{\R}{\mathbb{R}}
\newcommand{\Rp}{\mathbb{R}_{\geq 0}}
\newcommand{\Rpp}{\mathbb{R}_{>0}}
\newcommand{\1}{\mathbbm{1}}

\renewcommand{\epsilon}{\varepsilon}

\newcommand{\A}{\mathcal{A}}
\newcommand{\D}{\mathcal{D}}

\newcommand{\tG}{\widetilde{G}}

\newcommand{\HH}{\mathcal{H}}

\newcommand{\m}{\overline{m}}
\newcommand{\M}{\mathcal{M}}

\newcommand{\p}{\varphi}

\renewcommand{\P}{\mathbb{P}}

\newcommand{\td}{\widetilde}
\renewcommand{\bar}{\overline}

\newcommand{\abs}[1]{\lvert#1\rvert}
\newcommand{\norm}[1]{\lVert#1\rVert}
\newcommand{\Norm}[1]{{\vert\kern-0.1ex\vert\kern-0.1ex\vert#1\vert\kern-0.1ex\vert\kern-0.1ex\vert}}
\newcommand{\dNorm}[2]{{\vert\kern-0.1ex\vert\kern-0.1ex\vert#2\vert\kern-0.1ex\vert\kern-0.1ex\vert}_{#1,*}}

\newcommand{\dec}{\searrow}

\newcommand{\ud}{\, \mathrm{d}}

\DeclareMathOperator{\E}{\mathbb{E}}

\DeclareMathOperator{\Par}{\mathcal{P}}

\DeclareMathOperator{\BigO}{\operatorname{\mathcal{O}}}

\DeclareMathOperator{\Ber}{\operatorname{Ber}}
\DeclareMathOperator{\Poi}{\operatorname{Poi}}
\DeclareMathOperator{\Exp}{\operatorname{Exp}}

\newcommand{\step}[1]{{\\[6pt]\emph{Step #1.}}}

\newtheorem{theorem}{Theorem}[section]
\newtheorem{lemma}[theorem]{Lemma}
\newtheorem{proposition}[theorem]{Proposition}
\newtheorem{corollary}[theorem]{Corollary}

\newtheorem*{acknowledgements}{Acknowledgements}

\theoremstyle{remark}
\newtheorem{remark}[theorem]{Remark}

\numberwithin{equation}{section}

\begin{document}

\title{\vspace{-1cm}Critical point representation of the mutual information in the sparse stochastic block model
}
\author{Tomas Dominguez\thanks{\textsc{\tiny Department of Mathematics, University of Toronto, tomas.dominguezchiozza@mail.utoronto.ca}} \and  Jean-Christophe Mourrat\thanks{\textsc{\tiny Department of Mathematics, ENS Lyon and CNRS, jean-christophe.mourrat@ens-lyon.fr}}}

\date{}
\maketitle
\vspace*{-0.7cm}
\begin{abstract}
We consider the problem of recovering the community structure in the stochastic block model. We aim to describe the mutual information between the observed network and the actual community structure as the number of nodes diverges while the average degree of a given node remains bounded. Our main contribution is a representation of the limit of this quantity, assuming it exists, as an explicit functional evaluated at a critical point of that functional. While we mostly focus on the two-community setting for clarity, we expect our method to be robust to model generalizations. We also present an example involving four communities where we show the invalidity of a plausible candidate variational formula for this limit.
\end{abstract}

\section{Introduction}
\label{sec:intro}

\subsection{Informal summary of the main results}

We consider a problem of community detection in the context of the stochastic block model. We focus on the case with two communities, where the community structure is encoded by a vector 
\begin{equation}
\sigma^*:=(\sigma^*_1,\ldots,\sigma^*_N)\in \Sigma_N:=\{-1,+1\}^N.
\end{equation}
We understand that individuals $i$ and $j$ belong to the same community if and only if $\sigma^*_i = \sigma^*_j$. The labels $(\sigma^*_i)_{i \leq N}$ are sampled independently from a Bernoulli distribution~$P^*$ with probability of success $p\in (0,1)$ and expectation $\m$,
\begin{equation}
p:=P^*\{1\}=\P\{\sigma^*_i=1\} \quad \text{and} \quad \m:=\E \sigma^*_i=2p-1.
\end{equation}
The assignment vector $\sigma^*$ is thus distributed according to the product law 
\begin{equation}
\sigma^*\sim P_N^*:=(P^*)^{\otimes N}.
\end{equation}
We fix parameters $c> 0$ and $\Delta \in (-c,c) \setminus \{0\}$. The case when $\Delta = 0$ would be trivial and we prefer to exclude it. Given the community structure $\sigma^*$, we sample a random undirected graph $\mathbf{G}_N:=(G_{ij})_{i,j\leq N}$ with vertex set $\{1,\ldots,N\}$ 
by stipulating that an edge between node $i$ and node $j$ is present with conditional probability
\begin{equation}\label{e.intro.SSBM.edge.probs}
\P\{G_{ij}=1 \mid \sigma^*\}:=\frac{c+\Delta \sigma_i^*\sigma_j^*}{N},
\end{equation}
and that this sampling is performed independently for each pair $i \neq j \in \{1,\ldots, N\}$. In words, we create a link between two vertices with probability $(c+\Delta)/N$ if those two vertices belong to the same community; and we create this link with probability $(c-\Delta)/N$ otherwise. In the case when $\Delta > 0$, it is more likely for an edge to be present between nodes belonging to the same community, and the model is called \emph{assortative}; otherwise it is called \emph{disassortative}. The inference task that we consider is that of recovering the community structure $\sigma^*$ given the observation of the graph $\mathbf{G}_N$. We focus in particular on characterizing the mutual information between the graph $\mathbf{G}_N$ and the community structure $\sigma^*$, in the limit of large $N$. 

We denote by $\M_+$ the space of non-negative measures over $[-1,1]$ with finite total mass, and by $\M_p$ the set of probability measures with mean $\m = 2p-1$, 
\begin{equation}
\M_p:=\bigg\{\nu\in \Pr[-1,1]\mid \int_{-1}^1 x\ud \nu(x)=\m \bigg\}.
\end{equation} 
Our main results will be stated in terms of a fixed point of the functional $\Gamma_{t,\mu} : \M_p \to \M_p$, which is defined for each $t \ge 0$ and $\mu \in \M_+$ 
according to the following procedure. Fixing $\nu \in \M_p$, we sample $\epsilon^*$ according to $P^*$, and conditionally on $\epsilon^*$, we let $\Pi_{t,\mu}(\nu)$ be a Poisson point process with intensity measure ${(c+\Delta \epsilon^* x)} \ud (\mu+t\nu)(x)$; we then set
\begin{equation}\label{e.intro.SBM.FP.operator}
\Gamma_{t,\mu}(\nu):=\text{Law}\Bigg(\frac{\int_{\Sigma_1} \epsilon e^{-\epsilon \m}\prod_{x\in \Pi_{t,\mu}(\nu)} (c+\Delta \epsilon x)\ud P^*(\epsilon)}{\int_{\Sigma_1} e^{-\epsilon \m}\prod_{x\in \Pi_{t,\mu}(\nu)} (c+\Delta \epsilon x)\ud P^*(\epsilon)}\Bigg).
\end{equation}
In other words, the quantity between the large parentheses in \eqref{e.intro.SBM.FP.operator} is a random variable, as it depends on the realization of $\epsilon^*$ and $\Pi_{t,\mu}(\nu)$, and we define $\Gamma_{t,\mu}(\nu)$ to be the law of this random variable.

Let $(\sigma^\ell)_{\ell \ge 1} = ((\sigma^\ell_i)_{i \le N})_{\ell \ge 1}$ be independent random variables that are sampled according to the conditional law of $\sigma^*$ given the observation of $\mathbf{G}_N$. 
Our first main result states that, up to a perturbation of the inference problem that does not change the asymptotic behavior of the mutual information, and up to the extraction of a subsequence in $N$, the asymptotic law of $(\sigma^\ell_i)_{i \le N, \ell \ge 1}$ can be sampled as follows. For some measure $\nu^*$ that is a fixed point of the operator $\Gamma_{1,0}$:
\begin{enumerate}[label = \textbf{S\arabic*}]
\item generate a sequence $(x_i)_{i\geq 1}$ of i.i.d.\@ random variables with law $\nu^*$;\label{gen.spin.array.1}
\item conditionally on $(x_i)_{i \ge 1}$, generate $\pm 1$-valued random variables $\sigma_i^\ell$ with mean $x_i$, independently over $i,\ell\geq 1$.\label{gen.spin.array.2}
\end{enumerate}

Our second main result states that, assuming that the limit mutual information exists, it can be expressed in terms of an explicit functional $\Par_{1,0}: \M_+ \to \R$ that is evaluated at a fixed point of the map $\Gamma_{1,0}$. 
The functional $\Par_{1,0}$ is closely related to the map $\Gamma_{1,0}$ since, at least formally, a measure is a critical point of $\Par_{1,0}$ if and only if it is a fixed point of the map $\Gamma_{1,0}$.

While we think of our results as being rather precise, since they substantially narrow down the space of possible asymptotic behaviors for the mutual information and for the conditional law of $\sigma^*$ given the observation of $\mathbf{G}_N$, we also wish to stress that our results do not provide a complete characterization of the limits of these objects. Indeed, this is due to the fact that the map~$\Gamma_{1,0}$ may admit several fixed points in general. We will show that the existence of a unique fixed point can be guaranteed in a regime of small signal-to-noise ratio, but we do not expect this uniqueness property to extend to arbitrary choices of parameters. 

A complete characterization of the limit of the mutual information has already been achieved in a number of cases. The disassortative case was settled in \cite{Abbe_disassortative, coja2018information} when $p = 1/2$, and in \cite{dominguez2024mutual} for arbitrary $p$. The technique in \cite{dominguez2024mutual} is based on partial differential equations (see also \cite{TD_JC_book} for a presentation of the approach), and the key point is that in the disassortative case, the non-linearity appearing in this equation is \emph{convex}. We expect this complete characterization to extend to settings with more than two communities, as long as the corresponding non-linearity remains convex. This would be analogous to \cite[Theorem~1.1]{chen2023free} in the context of spin glasses, a result which itself builds upon \cite{barcon, gue03, pan.aom, PanSKB, pan14,  pan.multi, pan.potts, pan.vec, Tpaper}. Moreover, whenever this convexity property holds, the limit mutual information can be expressed as a variational formula. From the point of view of partial differential equations, this relates to the validity of a Hopf-Lax representation of the solution, see \cite[Theorem~1.3]{TD_JC_HJ}. 

The assortative case presents a greater challenge. It was successfully resolved when $p = 1/2$ in \cite{yu2023ising}, building upon \cite{Abbe_assortative, kanade2016global, mossel2016belief, mossel2016local}. The approach in \cite{yu2023ising} relies upon the study of the fixed points of an operator that relates to the setting of broadcasting on trees, and the key step of the proof is to establish the uniqueness of a non-trivial fixed point. This uniqueness property does not generalize to models with more than two communities \cite{gu2023uniqueness}. Despite similarities, we could not identify a precise correspondence between the operator $\Gamma_{1,0}$ whose fixed points enter into our analysis and the operators appearing in \cite{yu2023ising}. Part of the difficulty is that the first step of the approach in these earlier works is to decompose the mutual information into an integral involving a family of new statistical inference problems with an additional survey mechanism. When one can guarantee that there is a unique non-trivial fixed point to each problem in this family, there is a tight relationship between the true community-detection problem and the problems involving broadcasting on trees. On the other hand, when uniqueness fails, this relationship seems less rigid to us, and we do not immediately see how one could derive results comparable to those obtained in the present paper in the context explored in these earlier works.

We expect the method presented here, primarily based on cavity calculations, to be robust to model variations such as situations involving more than two communities, and we also expect the lower bound on the mutual information obtained in \cite{dominguez2024mutual} to generalize similarly. The results we present here and those in \cite{dominguez2024mutual} are analogous to those obtained for spin glasses in \cite{chen2023free} and in \cite{mourrat2021nonconvex, mourrat2023free} respectively. We have preferred to stick to a two-community model here to facilitate the presentation, but as an illustration of the robustness of the approach, we will at least indicate as we proceed the few small adaptations that need to be implemented in order to cover the case when then measure $P^*$ is arbitrary with support in $[-1,1]$ in place of $\{-1,1\}$. 

In order to make further progress and obtain a full characterization of the limit mutual information in general, one must first ask what a good candidate for the limit could be. One proposal based on partial differential equations was put forward in \cite{dominguez2024mutual}. It would of course be convenient to identify a more explicit formulation of the candidate limit. For dense versions of the problem we consider, where the average degree of the graph diverges to infinity, a sort of central limit theorem takes place at the level of each node, and one can relate the problem to a simpler setup such as that of rank-one matrix estimation in the presence of additive Gaussian noise, as shown in \cite{deshpande2017asymptotic, lelarge2019fundamental} (see also \cite[Section~4.5]{TD_JC_book}). For the latter class of problems, a saddle-point variational formula was shown to be valid in a very wide range of settings in \cite{chen2022statistical} (see also \cite[Chapter~4]{TD_JC_book}), building upon \cite{barbier2016, barbier2019adaptive, barbier2017layered, chen2022hamilton, chen2021limiting, chen2022hamilton2, kadmon2018statistical, lelarge2019fundamental, lesieur2017statistical,  luneau2021mutual, luneau2020high, mayya2019mutual, miolane2017fundamental,  mourrat2020hamilton, mourrat2021hamilton, reeves2020information, reeves2019geometry}. From the point of view of partial differential equations, this corresponds to the possibility to write a Hopf formula based on the convexity of the initial condition (as opposed to the convexity of the non-linearity for the Hopf-Lax formula, cf.\ \cite[Theorems~3.8 and 3.13]{TD_JC_book}). However, the convexity of the relevant ``initial condition'' in our setting is invalid \cite{kireeva2023breakdown}, and as explained around \cite[(1.49)]{dominguez2024mutual}, this implies that this tentative saddle-point Hopf formula is invalid in our setting with sparse graph connectivity. An alternative candidate limit for the mutual information could be considered by extending the Hopf-Lax variational formula known to be valid in the disassortative case. While we could not rule this out in the two-community case, we show here that this candidate variational formula is invalid in general.
To do so, we focus on a version of the two-community problem in which the graph additionally possesses a bipartite structure. The model can be thought of as involving four communities, and it is inspired by the bipartite spin-glass model investigated in \cite{mourrat2021nonconvex} (see in particular \cite[Section~6]{mourrat2021nonconvex}). In this context, we show that a naive generalization of the formula valid in the standard disassortative two-community case would lead to a prediction that is in contradiction with the main results of this paper.

As was already said, the main focus of this paper is to obtain precise characterizations of the mutual information and conditional law of the community structure given the observation, in the limit of large system size. In the two-community case with $p = 1/2$, it was shown in \cite{Massoulie, Mossel_reconstruction, Mossel_detection} that there exists an efficient algorithm that allows one to recover non-trivial information about the community structure whenever $\Delta^2 > c$, and that it is information-theoretically impossible to do so otherwise. To the best of our knowledge, whether or not there exists an efficient algorithm that recovers as much of the signal as is information-theoretically possible is not known. In certain settings involving more than two communities, one expects the existence of a gap between the regime for which there exists an efficient algorithm for recovering a non-trivial fraction of the community structure, and the regime for which it is information-theoretically possible to do so. The results of \cite{Abbe_more_communities} vindicate a significant part of this picture. A precise characterization of the information-theoretic limit is not known however. These open questions are directly related to the main purposes of the present paper. Much of the recent work on the community detection problem has been triggered by \cite{DKMZ}, in which the results reviewed in this paragraph and much more had been anticipated using non-rigorous methods.

\subsection{Precise statement of the main results}

We now turn to a more precise description of our main results. As in \cite{dominguez2024mutual}, our starting point is to embed the community-detection problem under consideration into a richer family parametrized by a pair $(t,\mu) \in \Rp \times \M_+$. The conjecture in \cite{dominguez2024mutual} is that, up to a simple affine transformation, the limit $f(t,\mu)$ of the re-scaled mutual information between the signal and the observation solves an explicit Hamilton-Jacobi equation. As is well-known in finite dimensions, first-order Hamilton-Jacobi equations can be solved for a short time using the method of characteristics (see for instance \cite[Exercise~3.10 and solution]{TD_JC_book}). This method allows us to uncover the value of the solution to the equation along each characteristic line separately. At larger times though, difficulties occur because characteristic lines may intersect, and will typically not agree on what the value of the solution ought to be. Our main result can be interpreted as saying that, while this might indeed occur, there is always one characteristic line that prescribes the correct value for $f(t,\mu)$. In the language of \cite[Section~3.5]{TD_JC_book}, one may say that the graph of $f$ always belongs to the wavefront of the equation. 

As already stated, the central object of interest to us is the mutual information between the community structure $\sigma^*$ and the observed random graph~$\mathbf{G}_N$, that is,
\begin{equation}\label{e.intro.SBM.MI}
I(\mathbf{G}_N; \sigma^*):=\E\log\frac{\P(\mathbf{G}_N,\sigma^*)}{\P(\mathbf{G}_N)\P(\sigma^*)}=\E\int_{\R^N} \log\bigg(\frac{\ud P_{\sigma^*\mid \mathbf{G}_N}}{\ud P_N^*}(\sigma)\bigg)\ud P_{\sigma^* \mid \mathbf{G}_N}(\sigma),
\end{equation}
where $P_{\sigma^*\mid \mathbf{G}_N}$ denotes the conditional law of $\sigma^*$ given $\mathbf{G}_N$. 
The likelihood of the model is given by
\begin{equation}\label{e.intro.SBM.likelihood}
\P\big\{\mathbf{G}_N=(G_{ij})\mid \sigma^*=\sigma\big\}=\prod_{i<j}\Big(\frac{c+\Delta \sigma_i\sigma_j}{N}\Big)^{G_{ij}}\Big(1-\frac{c+\Delta \sigma_i\sigma_j}{N}\Big)^{1-G_{ij}},
\end{equation}
and Bayes' formula implies that the posterior of the model is the Gibbs measure
\begin{equation}\label{e.intro.SBM.posterior}
\P\big\{\sigma^*=\sigma \mid \mathbf{G}_N=(G_{ij})\big\}=\frac{\exp\big(H_N^\circ(\sigma)\big)P_N^*(\sigma)}{\int_{\Sigma_N}\exp\big(H_N^\circ(\tau)\big)\ud P_N^*(\tau)}
\end{equation}
associated with the Hamiltonian
\begin{equation}\label{e.intro.SBM.H0}
H_N^\circ(\sigma):=\sum_{i<j}\log\bigg[\big(c+\Delta \sigma_i\sigma_j\big)^{G_{ij}}\Big(1-\frac{c+\Delta \sigma_i\sigma_j}{N}\Big)^{1-G_{ij}}\bigg].
\end{equation}
Up to an error vanishing with $N$ and a simple additive constant, the normalized mutual information~\eqref{e.intro.SBM.MI} coincides with the free energy
\begin{equation}\label{e.intro.SBM.FE0}
\overline{F}_N^\circ:=\frac{1}{N}\E\log \int_{\Sigma_N}\exp H_N^\circ(\sigma)\ud P_N^*(\sigma).
\end{equation}
Indeed, one can observe that
\begin{equation}
I(\mathbf{G}_N;\sigma^*)=\binom{N}{2}\E\log(c+\Delta \sigma_1^*\sigma_2^*)^{G_{12}}\Big(1-\frac{c+\Delta \sigma_1^*\sigma_2^*}{N}\Big)^{1-G_{12}}-N\overline{F}_N^\circ;
\end{equation}
averaging with respect to the randomness of $G_{12}$ and Taylor-expanding the logarithm yields that
\begin{equation}\label{eqn: SBM MI and free energy 0}
\frac{1}{N}I(\mathbf{G}_N;\sigma^*)=\frac{1}{2}\E\big(c+\Delta \sigma_1^*\sigma_2^*\big)\log\big(c+\Delta \sigma_1^*\sigma_2^*\big)-\frac{c}{2}-\frac{\Delta \m^2}{2}-\overline{F}_N^\circ+\BigO\big(N^{-1}\big).
\end{equation}
In place of the mutual information, the main protagonist of this paper is in fact the free energy~$\bar F_N^\circ$ and its extension, which we now proceed to introduce. 

We begin by describing the enrichment of the free energy involving the parameter $t \ge 0$. For each $t\geq 0$, we introduce a random variable $\smash{\Pi_t\sim \Poi t\tbinom{N}{2}}$ as well as an independent family of i.i.d.\@ random matrices $\smash{(G^k)_{k \geq 1}}$ each having conditionally independent entries $\smash{(G_{i,j}^k)_{i,j\leq N}}$ taking values in $\{0,1\}$ with conditional distribution
\begin{equation}\label{e.intro.SBM.G.distribution}
\P\big\{G_{i,j}^k=1 \mid \sigma^*\big\}:=\frac{c+\Delta \sigma_{i}^*\sigma_{j}^*}{N}.
\end{equation}
Given a collection of random indices $\smash{(i_k,j_k)_{k\geq 1}}$ sampled uniformly at random from $\smash{\{1,\ldots,N\}^2}$, independently of the other random variables, we define the $t$-dependent Hamiltonian $H_N^t$ on $\Sigma_N$ by
\begin{equation}\label{e.intro.SBM.H.t}
H_N^t(\sigma):=\sum_{k\leq \Pi_t}\log \bigg[\big(c+\Delta \sigma_{i_k}\sigma_{j_k}\big)^{G^k_{i_k,j_k}}\Big(1-\frac{c+\Delta \sigma_{i_k}\sigma_{j_k}}{N}\Big)^{1-G^k_{i_k,j_k}}\bigg],
\end{equation}
and denote by 
\begin{equation}
\bar F_N(t):= \frac{1}{N} \E \log \int_{\Sigma_N}\exp H_N^t(\sigma)\ud P_N^*(\sigma)
\end{equation}
its corresponding free energy. Notice that this is the Hamiltonian associated with the task of inferring the signal $\sigma^*$ from the data 
\begin{equation}
\D_N^t:=\big(\Pi_t, (i_k,j_k)_{k\leq \Pi_t},(G^k_{i_k,j_k})_{k\leq \Pi_t}\big).
\end{equation}
By this we mean that the law of $\sigma^*$ given the observation of $\D_N^t$ is described by a formula analogous to that in \eqref{e.intro.SBM.posterior}, with $H_N^\circ$ replaced by $H_N^t$ there. 
By \cite[Appendix~A]{dominguez2024mutual}, the difference between the free energies $\bar F_N^\circ$ and $\bar F_N(1)$ converges to zero as $N$ tends to infinity. The main advantage of this construction is that it allows one to study derivatives of the free energy with respect to the ``time'' parameter $t$. It is also convenient to study derivatives of the free energy with respect to an additional parameter denoted by $\mu \in \M_+$. 
Given $\mu \in \M_+$, we let $(\Lambda_i(N\mu))_{i\leq N}$ be independent Poisson point processes with intensity measure $N\mu$, and we let the $\mu$-dependent Hamiltonian $H_N^\mu$ on $\Sigma_N$ be given by
\begin{equation}\label{e.intro.SBM.H.mu}
H_N^\mu(\sigma):=\sum_{i\leq N}\sum_{x\in \Lambda_i(N\mu)}\log \bigg[\big(c+\Delta \sigma_{i}x\big)^{G^x_{i}}\Big(1-\frac{c+\Delta \sigma_{i}x}{N}\Big)^{1-G^x_{i}}\bigg],
\end{equation}
where $(G_i^x)_{i\leq N}$ are conditionally independent random variables taking values in $\{0,1\}$ with conditional distribution
\begin{equation}
\P\{G_i^x=1\mid \sigma^*, x\}:=\frac{c+\Delta \sigma_i^* x}{N}.
\end{equation}
For a review of Poisson point processes, we refer the reader to \cite[Chapter 5]{TD_JC_book}. The function $H_N^\mu$ is the Hamiltonian associated with the task of inferring the signal $\sigma^*$ from the data
\begin{equation}
{\D}_N^{\mu}:=\big(\Lambda_i(N\mu), (G_{i}^x)_{x\in \Lambda_i(N\mu)}\big)_{i\leq N}.
\end{equation}
Finally, for each $(t,\mu)\in \Rp\times \M_+$, we introduce the enriched Hamiltonian
\begin{equation}\label{e.intro.SBM.en.H}
H_N^{t,\mu}(\sigma):=H_N^t(\sigma)+H_N^\mu(\sigma)
\end{equation}
as well as its associated free energy
\begin{equation}\label{e.intro.SBM.en.FE}
\bar F_N(t,\mu):=\frac{1}{N}\E\log \int_{\Sigma_N}\exp H_N^{t,\mu}(\sigma)\ud P_N^*(\sigma).
\end{equation}
The function $H_N^{t,\mu}$ is the Hamiltonian associated with inferring the signal $\sigma^*$ from the data
\begin{equation}\label{e.intro.SBM.en.data}
{\D}_N^{t,\mu}:=(\D_N^t, \D_N^{\mu}),
\end{equation}
where the randomness in these two datasets is taken to be independent conditionally on $\sigma^*$. We recall that by \cite[Appendix~A]{dominguez2024mutual}, the difference between the free energy $\bar F_N^\circ$ in \eqref{e.intro.SBM.FE0} and $\bar F_N(1,0)$ tends to zero as $N$ tends to infinity. The problem of finding the asymptotic value of the mutual information \eqref{e.intro.SBM.MI} is therefore a specific instance of the task of determining the limit of the enriched free energy \eqref{e.intro.SBM.en.FE}.

To state the main results of this paper, we introduce additional notation. We define the function $g:[-1,1]\to \R$ by
\begin{equation}\label{e.intro.SBM.g}
g(z):=(c+\Delta z)\big(\log(c+\Delta z)-1\big)=(c+\Delta z)\log(c)+c\sum_{n\geq 2}\frac{(-\Delta/c)^n}{n(n-1)}z^n-c,
\end{equation}
and for each measure $\mu \in \M_+$, we let $G_\mu:[-1,1]\to \R$ be the function obtained by marginalizing the kernel $(x,y)\mapsto g(xy)$,
\begin{equation}\label{e.intro.SBM.G.mu}
G_\mu(x):=\int_{-1}^1 g(xy)\ud \mu(y).
\end{equation}
Given $(t,\mu)\in \Rp\times \M_+$, we introduce the functional $\Par_{t,\mu}:\M_+\to \R$ defined by
\begin{equation}\label{e.intro.HJ.functional}
\Par_{t,\mu}(\nu):=\psi(\mu + t \nu)-\frac{t}{2} \int_{-1}^1 G_\nu(y)\ud \nu(y),
\end{equation}
where $\psi:\M_+\to \R$ denotes the function
\begin{align}\label{e.SBM.IC}
\psi(\mu):=-\mu[-1,&1] c+p\E \log \int_{\Sigma_1} \exp\Big(-\Delta \sigma \int_{-1}^1 x\ud \mu(x)\Big)\prod_{x\in \Pi_+(\mu)}(c+\Delta \sigma x)\ud P^*(\sigma)\notag\\
&+(1-p)\E \log \int_{\Sigma_1} \exp\Big(-\Delta \sigma \int_{-1}^1 x\ud \mu(x)\Big)\prod_{x\in \Pi_-(\mu)}(c+\Delta \sigma x)\ud P^*(\sigma)
\end{align}
for Poisson point processes $\Pi_{\pm}(\mu)$ with respective intensity measures $(c\pm \Delta x)\ud \mu(x)$ on $[-1,1]$. The function $\psi$ is the initial condition of the partial differential equation that was identified in \cite{dominguez2024mutual}, since it arises as the large-$N$ limit of $\bar F_N(0,\mu)$.
Our first main result states that the limit of the free energy, assuming it exists, is given by the functional \eqref{e.intro.HJ.functional} evaluated at one of its critical points. As mentioned in the previous subsection, it will be possible to recast the critical-point condition as a fixed-point equation for the operator $\Gamma_{t,\mu}$ defined in \eqref{e.intro.SBM.FP.operator}, so that our first main result reads as follows.

\begin{theorem}\label{t.intro.main}
Suppose that the sequence of enriched free energies $(\bar F_N)_{N\geq 1}$ converges pointwise to some limit $f:\Rp\times \M_+\to \R$. For every $(t,\mu)\in \Rp\times \M_+$, the map $\nu\mapsto \Gamma_{t,\mu}(\nu)$ admits a fixed point $\nu^* \in \M_p$ with the property that
\begin{equation}\label{e.intro.main}
f(t,\mu)=\Par_{t,\mu}(\nu^*).
\end{equation}
\end{theorem}

Our second result concerns the asymptotic behavior of the conditional law of the community structure $\sigma^*$ given the observation. We use the bracket notation $\langle \cdot \rangle$ to denote the expectation with respect to the conditional law of $\sigma^*$ given the observation, with $\sigma$ being the canonical random variable under $\langle \cdot \rangle$. Explicitly, this means that for every function $h : \Sigma_N \to \R$, we have
\begin{equation}
\label{e.first.def.gibbs}
\langle h(\sigma) \rangle := \frac{\int_{\Sigma_N} h(\sigma) \exp H_N^{t,\mu}(\sigma) \, \ud P_N(\sigma)}{\int_{\Sigma_N} \exp H_N^{t,\mu}(\sigma) \, \ud P_N(\sigma)}.
\end{equation}
We stress that the expectation $\langle \cdot \rangle$ is itself random, since it depends on the observation. We denote by $(\sigma^\ell)_{\ell \ge 1} = ((\sigma^\ell_i)_{i \le N})_{\ell \ge 1}$ independent copies, often called replicas, of the random variable $\sigma$ under $\langle \cdot \rangle$. We also introduce the set $\mathrm{Reg}(\M_+)$ of measures in $\M_+$ that admit a smooth and strictly positive density with respect to the Lebesgue measure on $[-1,1]$,
\begin{equation}\label{e.intro.Reg.M+}
\mathrm{Reg}(\M_+):=\big\{\mu\in \M_+\mid \mu \text{ admits a smooth and strictly positive density on } [-1,1]\big\}.
\end{equation}
Our next result states that for most choices of $(t,\mu)\in \Rp\times \M_+$, we can identify a fixed point of the map $\nu \mapsto \Gamma_{t,\mu}(\nu)$ for which the equality \eqref{e.intro.main} holds, and for which we can also identify the limit law of $(\sigma^\ell)_{\ell \ge 1}$ after a small perturbation of the inference problem. More precisely, we will show that this can be done at any point $(t,\mu)\in \Rp\times \mathrm{Reg}(\M_+)$ at which the limit free energy is Gateaux differentiable; we refer to Section \ref{sec:regularity} for a definition of Gateaux differentiability and a precise discussion of why the limit free energy is Gateaux differentiable at most points. The convergence in law of $(\sigma^\ell_i)_{i \le N, \ell \ge 1}$ is in the sense of finite-dimensional distributions. This result requires that we slightly modify the inference problem by adding a small amount of side information; the exact quantity we study then is defined in \eqref{e.cavity.pert.mod.H}-\eqref{e.cavity.mod.pert.FE}. 

\begin{theorem}\label{t.intro.main.limit.array}
Suppose that the sequence of enriched free energies $(\bar F_N)_{N\geq 1}$ converges pointwise to some limit $f:\Rp\times \M_+\to \R$. If $(t,\mu)\in \Rp\times \mathrm{Reg}(\M_+)$ is such that $f(t,\cdot)$ is Gateaux differentiable at $\mu$, then we can identify a probability measure $\nu^* \in \M_p$ such that the Gateaux derivative density of $f$ at $(t,\mu)$ is $G_{\nu^*}$. Moreover, the probability measure $\nu^*$ is a fixed point of the map $\Gamma_{t,\mu}$, and we have
\begin{equation}
f(t,\mu)=\Par_{t,\mu}(\nu^*).
\end{equation}
Finally, there exists a sequence $\smash{\big(N_k,\lambda^{N_k}\big)_{k\geq 1}}$ such that $\smash{(N_k)_{k\geq 1}}$ increases to infinity, the spin array $(\sigma^\ell_{i})_{i \le N_k, \ell \ge 1}$ converges in law under $\E \langle \cdot \rangle$ with the modified Hamiltonian \eqref{e.cavity.pert.mod.H} with perturbation parameters $\smash{(\lambda^{N_k})_{k\geq 1}}$, and the limit law is generated by $\nu^*$ according to \eqref{gen.spin.array.1}-\eqref{gen.spin.array.2}. The perturbation parameters $\lambda^{N_k}$ introduced in the modified Hamiltonian \eqref{e.cavity.pert.mod.H} are sufficiently small that the difference between the free energy \eqref{e.intro.SBM.en.FE} and the free energy with the perturbed Hamiltonian \eqref{e.cavity.pert.mod.H} tends to zero as $N$ tends to infinity.
\end{theorem}

One limitation of Theorems \ref{t.intro.main} and \ref{t.intro.main.limit.array} is that they assume the existence of the limit free energy, which we do not know a priori. To a large extent, this is a side effect of the fact that we do not know how to single out the correct critical point in general, as was discussed in the previous subsection. Correspondingly, this assumption can be removed whenever we can guarantee the uniqueness of a fixed point to $\Gamma_{t,\mu}$. As shown in Lemma \ref{l.limit.GD.measure.is.generator}, the second part of Theorem \ref{t.intro.main.limit.array} in fact does not require that we assume that the free energy converges. Indeed, to show this result we only need to ensure that the free energy converges along a subsequence, and this can always be obtained using the Arzelà-Ascoli theorem. 

Another limitation of Theorem~\ref{t.intro.main.limit.array} is that we need to add a small amount of side information to the inference problem, as described in \eqref{e.cavity.pert.mod.H}-\eqref{e.cavity.mod.pert.FE}, in order to identify the limit law of the posterior distribution. The perturbation is small in the sense that the amount of additional information per unknown tends to zero in the limit of large system size. Alternatively, we could consider a broader family of inference problems with side information, and then identify the limit law of the posterior distribution for most choices of the parameters defining the inference problem; see \cite[Theorem~1.4]{chen2023free} for a result in this spirit. We point out however that this perturbation is necessary for Theorem~\ref{t.intro.main.limit.array} to be valid. Indeed, even in the simpler setting of recovery of matrix tensor products, adding a small amount of side information is in some cases necessary in order to resolve certain symmetries of the model, see for instance \cite[Example~2]{reeves2020information}. While it would be interesting to make further progress on this point, we are not aware of precise results even in this simpler context.

The next proposition certifies the uniqueness of a fixed point in a regime of small signal-to-noise ratio.

\begin{proposition}\label{p.intro.main.small.t}
There exists $C<+\infty$ such that for all $t<C^{-1}$ and $\mu\in \M_+$, the map $\nu\mapsto \Gamma_{t,\mu}(\nu)$ admits a unique fixed point $\nu^*$ in $\M_p$. Moreover, the sequence of enriched free energies $\smash{(\bar F_N(t,\mu))_{N\geq 1}}$ converges to $\Par_{t,\mu}(\nu^*)$,
\begin{equation}\label{e.intro.lim.FE.t.small}
\lim_{N\to +\infty}\bar F_N(t,\mu)=\Par_{t,\mu}(\nu^*).
\end{equation}
\end{proposition}

For completeness, we also explain how to recover the variational formula already obtained in~\cite{Abbe_disassortative, coja2018information} (for $p = 1/2$) and in~\cite{dominguez2024mutual} for the limit free energy in the disassortative case.

\begin{proposition}\label{p.intro.main.disassortative}
In the disassortative case, that is when $\Delta\leq 0$, we have for every $(t,\mu)\in \Rp\times \M_+$ that
\begin{equation}\label{e.intro.main.disassortative}
\lim_{N \to +\infty} \bar F_N(t,\mu)= \sup_{\nu\in \M_p}\Par_{t,\mu}(\nu).
\end{equation}
\end{proposition}

As explained in the previous subsection, we will also show the invalidity of a candidate variational formula for the limit free energy. Since this involves a variant of the model with an additional bipartite structure, we prefer to postpone a precise description of the setting and result to Section~\ref{sec:CP_sel}.



\subsection{Organization of the paper}

In Section \ref{sec:regularity}, we recall the expressions for the derivatives of the free energy obtained in \cite{dominguez2024mutual}. We then leverage these to establish the uniform Lipschitz continuity of the free energy, the almost everywhere differentiability of the limit free energy, and the asymptotic behavior of the derivative of the free energy. 
Section \ref{sec:concentration} is a brief overview of the main multioverlap concentration result in~\cite{barbier2022strong} which implies that the spin array sampled from the asymptotic Gibbs measure is generated according to \eqref{gen.spin.array.1}-\eqref{gen.spin.array.2} for some measure $\nu^*$ in $\M_p$. In Section \ref{sec:cavity}, the cavity computations leading to the functional $\Par_{t,\mu}$ in \eqref{e.intro.HJ.functional} are performed. In Section \ref{sec:CP_rep}, the cavity representation is combined with the multioverlap concentration result in Section \ref{sec:concentration} and the regularity properties of the free energy obtained in Section \ref{sec:regularity} to prove Theorems \ref{t.intro.main} and \ref{t.intro.main.limit.array}. Section \ref{sec:special.cases} is devoted to the proofs of Propositions~\ref{p.intro.main.small.t} and \ref{p.intro.main.disassortative}. Finally, in Section \ref{sec:CP_sel}, it is argued that, unlike in the disassortative setting, one should not in general expect the limit free energy to be given by evaluating the functional $\Par_{t,\mu}$ at the measure $\nu \in \M_p$ which maximizes its value.

\begin{acknowledgements}
We would like to warmly thank Dmitry Panchenko and Jean Barbier for sharing their notes \cite{PanNote} on the free energy in the disassortative sparse stochastic block model with us, which helped us with many of the computations in Section \ref{sec:cavity}. We also warmly thank Hong-Bin Chen for pointing out an error in an earlier version of the proof of Proposition~\ref{p.intro.main.small.t}.
\end{acknowledgements}

\section{Regularity properties of the free energy}
\label{sec:regularity}

In this section we discuss three regularity properties of the enriched free energy \eqref{e.intro.SBM.en.FE}. In Section~\ref{subsec:regularity.Lipschitz}, we establish a Lipschitz continuity bound for the enriched free energy which implies in particular that the free energy restricted to the space of probability measures is Lipschitz continuous with respect to the Wasserstein distance. In Section \ref{subsec:regularity.differentiability}, we prove that a Lipschitz continuous function on the space of cumulative distribution functions on $[-1,1]$, viewed as a subset of $L^2[-1,1]$, is Gateaux differentiable on a dense subset of its domain. Finally, in Section \ref{subsec:regularity.convexity}, we leverage a convexity property of the free energy to prove that the sequence of derivatives of the free energy converges to the derivative of the limit free energy at any point of differentiability of said limit. We now state the three main results of this section. Since the ideas developed in these proofs will not reappear later, the reader may consider skipping these proofs on first reading.

The Lipschitz continuity result in Section \ref{subsec:regularity.Lipschitz} involves the Wasserstein distance. The \emph{Wasserstein distance} between two probability measures $\mu,\nu\in \Pr[-1,1]$ is defined by
\begin{align}\label{e.Wasserstein.def}
W(\mu,\nu)&:=\sup\bigg\{\Big\lvert \int_{-1}^1 h(x)\ud \mu(x)-\int_{-1}^1 h(x)\ud \nu(x)\Big\rvert \mid \norm{h}_{\text{Lip}}\leq 1\bigg\},
\end{align}
where $\norm{\cdot}_{\mathrm{Lip}}$ denotes the Lipschitz semi-norm
\begin{equation}
\norm{h}_{\mathrm{Lip}}:=\sup_{x\neq x'\in [-1,1]}\frac{\abs{h(x)-h(x')}}{\abs{x-x'}}
\end{equation}
on the space of Lipschitz continuous functions $h:[-1,1]\to \R$. The Kantorovich-Rubinstein theorem \cite[Theorem 4.15]{PanL} ensures that the Wasserstein distance admits the dual representation
\begin{equation}\label{e.Wasserstein.def.inf}
W(\mu,\nu)=\inf\big\{\E\abs{X-Y} \mid X\sim \mu \text{ and } Y\sim \nu\big\}.
\end{equation}
In the one-dimensional setting, as shown in \cite[Theorem 1.5.1 and Corollary 1.5.3]{panaretos2020invitation}, the Wasserstein distance admits an explicit representation in terms of the cumulative distribution function or the quantile function. Indeed, if $F_\mu:[-1,1]\to [0,1]$ and $F_\mu^{-1}: [0,1]\to [-1,1]$ denote the cumulative distribution function and the quantile transform of a probability measure $\mu\in \Pr[-1,1]$, 
\begin{equation}
F_\mu(x):=\mu[-1,x] \quad \text{and} \quad F_\mu^{-1}(u):=\inf\{x\in \R \mid F_\mu(x)\geq u\},
\end{equation}
then for all probability measures $\mu,\nu\in \Pr[-1,1]$,
\begin{equation}\label{e.Wasserstein.1D}
W(\mu,\nu)=\int_0^1 \big\lvert F_\mu^{-1}(u)-F_{\nu}^{-1}(u)\big\rvert \ud u=\int_{-1}^1 \big\lvert F_\mu(x)-F_\nu(x)\big\rvert \ud x.
\end{equation}
In this notation, the main Lipschitz continuity result for the enriched free energy reads as follows.

\begin{proposition}\label{p.regularity.Lipschitz}
There exists $C<+\infty$ such that for any $(t,\mu), (t,\mu')\in \Rp\times \M_+$,
\begin{equation}
\big\lvert \bar F_N(t,\mu)-\bar F_N(t',\mu')\big \rvert\leq C\big(\abs{t-t'}+\big\lvert \mu[-1,1]-\mu'[-1,1]\big\rvert + \mu[-1,1]\abs{F_{\bar \mu}-F_{\bar \mu'}}_{L^2}\big),
\end{equation}
where the probability measures $\bar \mu$, $\bar \mu' \in \Pr[-1,1]$ are such that 
\begin{equation}  
\label{e.def.barmu}
\mu = \mu[-1,1] \, \bar \mu \qquad \text { and } \qquad \mu' = \mu'[-1,1] \, \bar \mu'.
\end{equation}
\end{proposition}

In Sections \ref{subsec:regularity.differentiability}-\ref{subsec:regularity.convexity} and throughout the paper, the notion of Gateaux differentiability that we take may differ from the most classical one, so we proceed to define it precisely. As in \cite{chen2023free}, we fix a topological vector space $X$ and a subset $U$ of $X$. We denote by $X^*$ the continuous dual of $X$, and we write $\langle \cdot, \cdot \rangle$ for the duality pairing. We will not necessarily take $U$ to be an open set, so, given a point $q\in U$, the set
\begin{equation}\label{e.SBM.admissible.def}
\mathrm{Adm}(U,q):=\big\{x\in X\mid  \text{ there exists } r>0 \text{ such that for all } t\in [0,r] \text{ we have } q+tx\in U\big\}
\end{equation}
of admissible directions along which a small line segment starting at $q$ is contained in $U$ will play an important part in the definition. Indeed, we say that a function $h:U\to \R$ is \emph{Gateaux differentiable} at $q\in U$ if the following two conditions hold.
\begin{enumerate}[label = \textbf{GD\arabic*}]
\item For every $x\in \mathrm{Adm}(U,q)$, the limit
\begin{equation}
Dh(q;x):=\lim_{\epsilon \dec 0}\frac{h(q+\epsilon x)-h(q)}{\epsilon}    
\end{equation}
exists.\label{e.SBM.GD.1}
\item There is a unique $y\in X^*$ such that, for every $x\in \mathrm{Adm}(U,q)$, we have $Dh(q;x)=\langle y,x\rangle$.\label{e.SBM.GD.2}
\end{enumerate}
In this case, the dual vector $y$ is called the \emph{Gateaux derivative} of $h$ at $q$.

We will mainly be concerned with the case $X=\Rp\times \M_+$ endowed with the product of the Euclidean topology and the topology of weak convergence, which is induced by the Wasserstein distance \eqref{e.Wasserstein.def}. In this setting, given a function $f:\Rp\times \M_+\to \R$, a time $t\geq 0$, and measures $\mu\in \M_+$ and $\nu \in \mathrm{Adm}(\M_+,\mu)$, we denote by $D_\mu f(t,\mu;\nu)$ the Gateaux derivative of the function $f(t,\cdot)$ at the measure $\mu$ in the direction $\nu$,
\begin{equation}\label{e.SBM.GD.M+}
D_\mu f(t,\mu;\nu):=\lim_{\epsilon \to 0}\frac{f(t,\mu+\epsilon \nu)-f(t,\mu)}{\epsilon}.
\end{equation}
We say that the Gateaux derivative of $f(t,\cdot)$ admits a density at the measure $\mu \in \M_+$ if there exists a continuous function $x\mapsto D_\mu f(t,\mu,x)$ defined on the interval $[-1,1]$ such that for all $\nu \in \mathrm{Adm}(\M_+,\mu)$, 
\begin{equation}\label{e.SBM.GD.M+.density}
D_\mu f(t,\mu;\nu)=\int_{-1}^1 D_\mu f(t,\mu,x)\ud \nu(x).
\end{equation}
We often abuse notation and identify the density $D_\mu f(t,\mu,\cdot)$ with the Gateaux derivative $D_\mu f(t,\mu)$. Implicitly identifying the continuous dual of $\M_+$ with the space of continuous functions on $[-1,1]$, this definition of Gateaux differentiability coincides with that in \eqref{e.SBM.GD.1}-\eqref{e.SBM.GD.2} but differs slightly from that in \cite{dominguez2024mutual} where the density is only required to be measurable and bounded as opposed to continuous. However, all Gateaux derivative densities appearing in \cite{dominguez2024mutual} are of the finite-volume free energy and are continuous, so this difference is insignificant.

For technical reasons, it will be convenient for the dense subset of $\Rp\times \M_+$ on which we establish the Gateaux differentiability of the free energy to be composed of well-behaved measures. Recall the definition in \eqref{e.intro.Reg.M+} of the set $
\mathrm{Reg}(\M_+)$ of measures in $\M_+$ that admit a smooth and strictly positive density with respect to Lebesgue measure on $[-1,1]$. A key property of this set is that for any $\mu,\nu\in \mathrm{Reg}(\M_+)$ and $t \in \R$ with $\abs{t}$ small enough, we have $\mu+t \nu \in \M_+$. In particular, the set of admissible directions at $\mu \in \mathrm{Reg}(\M_+)$ contains $\mathrm{Reg}(\M_+)$. In this notation, the main Gateaux differentiability results for the enriched free energy proved in Sections \ref{subsec:regularity.differentiability} and \ref{subsec:regularity.convexity} read as follows.

\begin{proposition}\label{p.regularity.differentiability}
Suppose that the sequence of enriched free energies $(\bar F_N)_{N\geq 1}$ converges pointwise to some limit $f:\Rp \times \M_+\to \R$ along a subsequence $(N_k)_{k\geq 1}$. The subsequential limit $f$ is Gateaux differentiable jointly in its two variables on a subset of $\Rpp\times \mathrm{Reg}(\M_+)$ that is dense in $\Rp\times \M_+$.
\end{proposition}

\begin{proposition}\label{p.regularity.cont.of.derivative}
Suppose that the sequence of enriched free energies $(\bar F_N)_{N\geq 1}$ converges pointwise to some limit $f:\Rp\times \M_+\to \R$ along a subsequence $(N_k)_{k\geq 1}$. For each $t\geq 0$, if $f(t,\cdot)$ is Gateaux differentiable at $\mu \in \mathrm{Reg}(\M_+)$, then $D_\mu \bar F_{N_k}(t,\mu,\cdot)$ converges weakly to $D_\mu f(t,\mu, \cdot)$. For each $\mu \in \M_+$, if $f(\cdot,\mu)$ is differentiable at $t > 0$, then $\partial_t \bar F_{N_k}(t,\mu)$ converges to $\partial_t f(t,\mu)$. 
\end{proposition}

In proving Propositions \ref{p.regularity.Lipschitz} and \ref{p.regularity.differentiability}, it will be convenient to remember from \cite[Lemmas 2.1 and 2.3]{dominguez2024mutual} that the enriched free energy \eqref{e.intro.SBM.en.FE} is Gateaux differentiable at every $(t,\mu)\in \Rp\times \M_+$ with
\begin{align}
\partial_t\bar F_N(t,\mu)&=\frac{1}{2}\E \big(c+\Delta \langle \sigma_1\sigma_2\rangle\big)\log \big( c+\Delta \langle \sigma_1\sigma_2\rangle\big)-\frac{\Delta \m^2}{2}-\frac{c}{2}+\BigO(N^{-1}), \label{e.SBM.en.FE.der.t}\\
D_\mu \bar F_N(t,\mu,x)&=\E\big(c+\Delta \langle \sigma_1\rangle x\big)\log \big(c+\Delta \langle\sigma_1\rangle x\big)-c-\Delta \m x+\BigO(N^{-1}).\label{e.SBM.en.FE.der.mu}
\end{align}
We have used $\langle \cdot \rangle$ to denote the Gibbs average associated with the enriched Hamiltonian \eqref{e.intro.SBM.en.H}. This means that for any bounded and measurable function $f=f(\sigma^1,\ldots,\sigma^n)$ of finitely many replicas, 
\begin{equation}\label{e.SBM.Gibbs}
 \langle f(\sigma^1, \ldots, \sigma^n) \rangle := \langle f\rangle   :=\frac{\int_{\Sigma_N^n}f(\sigma^1,\ldots,\sigma^n)\prod_{\ell\leq n}\exp H_N^{t,\mu}(\sigma^\ell)\ud P_N^*(\sigma^\ell)}{\big(\int_{\Sigma_N} \exp H_N^{t,\mu}(\sigma)\ud P_N^*(\sigma)\big)^n}.
\end{equation}
The computations leading to \eqref{e.SBM.en.FE.der.t}-\eqref{e.SBM.en.FE.der.mu} are considerably simplified by the \emph{Nishimori identity}. This identity allows us to freely interchange one replica $\sigma^\ell$ by the signal $\sigma^*$ when taking an average with respect to all sources of randomness, thus avoiding a cascade of new replicas as we differentiate the free energy. More precisely, it states that, for every bounded and measurable function $f=f(\sigma^1,\ldots,\sigma^n,\D_N^{t,\mu})$ of finitely many replicas and the data, 
\begin{equation}\label{e.SBM.Nishimori}
\E\big\langle f\big(\sigma^1,\sigma^2,\ldots,\sigma^n, \D_N^{t,\mu}\big)\big\rangle=\E\big\langle f\big(\sigma^*,\sigma^2,\ldots,\sigma^n, \D_N^{t,\mu}\big)\big\rangle.
\end{equation}
This can be first verified for functions of product form using that the Gibbs average is the conditional law of the signal $\sigma^*$ given the data $\D_N^{t,\mu}$, and then extended to all bounded and measurable functions by a monotone class argument as in \cite[Proposition 4.1]{TD_JC_book}.

The proof of Proposition \ref{p.regularity.cont.of.derivative} will rely on a convexity property of the free energy that will be obtained through information-theoretic arguments taken from \cite{kireeva2023breakdown}.

\subsection{Uniform Lipschitz continuity of the free energy}
\label{subsec:regularity.Lipschitz}

The Lipschitz continuity of the free energy stated in Proposition \ref{p.regularity.Lipschitz} will be deduced from the Lipschitz continuity of an extension of the free energy. Given a probability measure $\mu \in \Pr[-1,1]$, consider a sequence $x=(x_{i,k})_{i,k\geq 1}$ of i.i.d.\@ random variables with law $\mu$. For each $s>0$ and $i\geq 1$, let $\Pi_{i,s}\sim \Poi(sN)$ be independent over $i\geq 1$, and introduce the Hamiltonian on $\Sigma_N$ defined by
\begin{equation}\label{e.SBM.H.s}
\widetilde{H}_N^{s,\mu}(\sigma):=\sum_{i\leq N}\sum_{k\leq \Pi_{i,s}}\log\bigg[\big(c+\Delta \sigma_ix_{i,k}\big)^{\tG_{i,k}^x}\Big(1-\frac{c+\Delta \sigma_ix_{i,k}}{N}\Big)^{1-\tG_{i,k}^x}\bigg],
\end{equation}
where the random variables $(\tG_{i,k}^x)_{i,k\geq 1}$ are independent with conditional distribution
\begin{equation}\label{eqn: SBM tG distribution}
\P\big\{\tG_{i,k}^x=1 \mid \sigma^*,x\big\}:=\frac{c+\Delta \sigma_i^*x_{i,k}}{N}.
\end{equation}
Recalling the definition of the time-dependent Hamiltonian in \eqref{e.intro.SBM.H.t}, we define the extended free energy functional
\begin{equation}\label{e.SBM.en.FE.s}
\widetilde{F}_N(t,s,\mu):=\frac{1}{N}\E\log \int_{\Sigma_N}\exp\big(H_N^t(\sigma)+\widetilde{H}_N^{s,\mu}(\sigma)\big)\ud P_N^*(\sigma).
\end{equation}
For any $\mu \in \M_+$, we have $\bar F_N(t,\mu)=\widetilde{F}_N\big(t,\mu[-1,1],\bar \mu\big)$, where $\bar \mu$ is the probability measure obtained by normalizing $\mu$ as in \eqref{e.def.barmu}. Indeed, this is how the enriched free energy \eqref{e.intro.SBM.en.FE} was defined in \cite{dominguez2024mutual}. The Lipschitz continuity of the extended free energy \eqref{e.SBM.en.FE.s} will be obtained by combining the mean value theorem with the derivative expressions \eqref{e.SBM.en.FE.der.t}-\eqref{e.SBM.en.FE.der.mu}. It will also be convenient to remember from \cite[Lemma 2.3]{dominguez2024mutual} that
\begin{equation}\label{e.SBM.extended.FE.der.s}
\partial_s\widetilde{F}_N(t,s,\mu)=\E\big(c+\Delta \langle \sigma_1\rangle x_1\big)\log \big(c+\Delta \langle\sigma_1\rangle x_1\big)-c-\Delta \m \E x_1+\BigO(N^{-1}),
\end{equation}
where the random variable $x_1$ has law $\mu$.

\begin{lemma}\label{l.SBM.extended.FE.Lip}
There exists $C<+\infty$ such that for any $(t,s,\mu), (t',s',\nu)\in \Rp\times \Rp\times \Pr[-1,1]$,
\begin{equation}
\lvert \widetilde F_N(t,s,\mu)-\widetilde F_N(t',s',\nu)\rvert \leq C\big(\abs{t-t'}+\abs{s-s'}+s'W( \mu,  \nu)\big).
\end{equation}
\end{lemma}

\begin{proof}
The derivative expressions \eqref{e.SBM.en.FE.der.t} and \eqref{e.SBM.extended.FE.der.s} reveal that $\partial_t\widetilde{F}_N(t,s,\mu)$ and $\partial_s\widetilde{F}_N(t,s,\mu)$ are bounded by some constant independent of the triple $(t,s,\mu)$. By the mean value theorem and the triangle inequality it therefore suffices to show that there exists $C<+\infty$ such that for all $t,s\in\Rpp$ and $\mu,\nu\in \Pr[-1,1]$,
\begin{equation}\label{e.SBM.extended.FE.Lip.goal}
\lvert \widetilde F_N(t,s,\mu)-\widetilde F_N(t,s,\nu)\rvert \leq
 CsW( \mu,  \nu).
\end{equation}
The fundamental theorem of calculus and the definition of the Gateaux derivative imply that
\begin{equation*}
\widetilde F_N(t,s,\mu)-\widetilde F_N(t,s,\nu)=\bar F_N(t,s\mu)-\bar F_N(t,s\nu)=s\int_0^1 D_\mu\bar{F}_N\big(t,s\nu+us(\mu-\nu); \mu-\nu\big)\ud u.
\end{equation*}
For each $u\in [0,1]$, let $f_{t,s,u}(x):=D_\mu\bar{F}_N\big(t,s\nu+us(\mu-\nu); x\big)$ in such a way that 
\begin{equation}\label{e.SBM.extended.FE.Lip.key}
\big\lvert \widetilde F_N(t,s,\mu)-\widetilde F_N(t,s,\nu) \big\rvert \leq s\int_0^1 \Big\lvert \int_{-1}^1 f_{t,s,u}(x)\ud (\mu-\nu)(x) \Big\rvert\ud u.
\end{equation}
The mean value theorem and the bound
\begin{equation}\label{e.SBM.extended.FE.Lip.der.bound}
\big\lvert \partial_x D_\mu \bar{F}_N(t,\mu,x)\big\rvert\leq c\big(1+\abs{\log(2c)}+\abs{\log(c-\abs{\Delta})}\big)
\end{equation}
established for any $(t,\mu)\in \Rp\times \M_+$ in \cite[Lemma 3.3]{dominguez2024mutual} imply that 
\begin{equation*}
\norm{f_{t,s,u}}_{\text{Lip}}\leq c\big(1+\abs{\log(2c)}+\abs{\log(c-\abs{\Delta})}\big).    
\end{equation*}
Together with the definition of the Wasserstein distance in \eqref{e.Wasserstein.def} and the bound \eqref{e.SBM.extended.FE.Lip.key}, this establishes \eqref{e.SBM.extended.FE.Lip.goal} and completes the proof.
\end{proof}

\begin{proof}[Proof of Proposition \ref{p.regularity.Lipschitz}]
The Lipschitz bound in Lemma \ref{l.SBM.extended.FE.Lip}, the representation \eqref{e.Wasserstein.1D} for the one-dimensional Wasserstein distance, and the Cauchy-Schwarz inequality give a constant $C<+\infty$ such that, for all $(t,s,\mu), (t',s',\nu)\in \Rp\times \Rp\times \Pr[-1,1]$,
\begin{equation*}
\lvert \widetilde F_N(t,s,\mu)-\widetilde F_N(t',s',\nu)\rvert \leq C\big(\abs{t-t'}+\abs{s-s'}+s'\abs{F_\mu-F_\nu}_{L^2}\big).
\end{equation*}
Remembering that for any $\mu \in \M_+$, we have $\bar F_N(t,\mu)=\widetilde{F}_N\big(t,\mu[-1,1],\bar \mu\big)$ completes the proof. 
\end{proof}

\subsection{Almost everywhere differentiability of the limit free energy}
\label{subsec:regularity.differentiability}

The Gateaux differentiability of the free energy stated in Proposition \ref{p.regularity.differentiability} will be deduced from the Gateaux differentiability of the free energy thought of as a function on the space of paths
\begin{align}
\D[0,1]:=\big\{q:[-1,1]\to [0,1] \mid q& \text{ is right-continuous}\notag\\
&\qquad\text{and non-decreasing with } q(1)=1\big\} \subset L^2[-1,1].
\end{align}
This space of paths is in one-to-one correspondence with the set of cumulative distribution functions via the mapping $\mathcal{F}:\Pr[-1,1]\to \D[0,1]$ defined by $\mathcal{F}(\mu):=F_\mu$. Through this mapping, any function $\smash{h:\Rp \times \M_+\to \R}$ may be identified with the function $\smash{\widetilde{h}':\Rp\times \Rp\times \D[0,1]\to \R}$ defined by
\begin{equation}\label{e.SBM.tilde.h.from.h}
\widetilde{h}'(t,s,q):=h\big(t,s\mathcal{F}^{-1}(q)\big).
\end{equation}
In particular, the free energy \eqref{e.intro.SBM.en.FE} may be identified with the function $\widetilde{F}_N':\Rp\times \Rp\times \D[0,1]\to \R$ defined by
\begin{equation}\label{e.SBM.en.FE.s.on.D01}
\widetilde{F}_N'(t,s,q):=\bar{F}_N\big(t,s\mathcal{F}^{-1}(q)\big).
\end{equation}
Notice that the notions of Gateaux differentiability for a function $\smash{h:\Rp\times \M_+\to \R}$ and its corresponding function $\smash{\widetilde{h}':\Rp\times \Rp \times \D[0,1]\to \R}$ differ. On the one hand, for each fixed $t\geq 0$, the Gateaux derivative of the function $\smash{h(t,\cdot)}$ at the measure $\mu \in \M_+$ may be identified with the continuous density $\smash{D_\mu h(t,\mu,\cdot)}$ defined according to \eqref{e.SBM.GD.M+.density}. On the other, for each fixed $t,s\geq 0$, the Gateaux derivative of the function $\smash{\widetilde{h}'(t,s,\cdot)}$ at the path $q\in \D[0,1]$ may be identified with the square-integrable function $\smash{\partial_q\widetilde{h}'(t,s,q,\cdot)\in L^2[-1,1]}$ having the property that for all $q'\in \mathrm{Adm}(\D[0,1],q)$,
\begin{equation}\label{e.SBM.GD.L2.density}
\widetilde{h}'(t,s,q+\epsilon q')=\widetilde{h}'(t,s,q)+\epsilon \int_{-1}^1 \partial_q\widetilde{h}'(t,s,q,u)q'(u)\ud u+o(\epsilon).
\end{equation}
However, it will be important to observe that whenever $\smash{\widetilde{h}'}$ is Gateaux differentiable at a point $\smash{(t,s,q)\in \Rp\times \Rp\times \D[0,1]}$, then $\smash{h}$ must be Gateaux differentiable at the corresponding point $\smash{(t,s\mathcal{F}^{-1}(q))\in \Rp\times \M_+}$.

\begin{lemma}\label{l.regularity.GD.h.tilde}
If the function $\smash{\widetilde{h}':\Rp\times \Rp\times \D[0,1]\to \R}$ is Gateaux differentiable jointly in its three variables at the point $\smash{(t,s,q)\in \Rpp\times \Rpp\times \D[0,1]}$, then the function $\smash{h:\Rp\times \M_+\to \R}$ used to define it according to \eqref{e.SBM.tilde.h.from.h} is Gateaux differentiable jointly in its two variables at the point $\smash{(t,s\mathcal{F}^{-1}(q))\in \Rpp\times \M_+}.$
\end{lemma}

\begin{proof}
For simplicity of notation, we will drop the time variable and prove that whenever the function $\widetilde{h}':\Rp\times \D[0,1]\to \R$ is Gateaux differentiable jointly in its two variables at the point $(s,q)\in \Rpp\times \D[0,1]$, then the function $h:\M_+\to \R$ used to define it according to \eqref{e.SBM.tilde.h.from.h} without the time variable is Gateaux differentiable at the point $s\mathcal{F}^{-1}(q)$. Generalizing the proof to incorporate the time variable requires only minor changes in notation. Let $\mu:=s\mathcal{F}^{-1}(q)$, and fix $\nu \in \mathrm{Adm}(\M_+,\mu)$ as well as $\epsilon>0$ small enough so that $\mu+\epsilon \nu \in \M_+$. We define $s':=\nu[-1,1]$, and observe that by Gateaux differentiability of $\widetilde{h}'$ at $(s,q)\in \Rpp\times \D[0,1]$,
\begin{align*}
h(\mu+\epsilon \nu)-h(\mu)&=\widetilde{h}'\bigg(s+\epsilon s', q+\epsilon \frac{s'F_{\bar \nu-\bar \mu}}{s+\epsilon s'}\bigg) - \widetilde{h}'(s,q)\\
&=\epsilon \bigg(s'\partial_s \widetilde{h}'(s,q)+\frac{s'}{s}\int_{-1}^1 \partial_q\widetilde{h}'(s,q,u) F_{\bar \nu - \bar \mu}(u)\ud u\bigg)+o(\epsilon).
\end{align*}
The Fubini-Tonelli theorem implies that the second term between parentheses is given by
\begin{align*}
\frac{1}{s}\int_{-1}^1 \int_{-1}^1 \partial_q\widetilde{h}'(s,q,u)\1_{[-1,u]}&(w)\ud \nu(w)\ud u-\frac{1}{s}\int_{-1}^1 \int_{-1}^1 \partial_q\widetilde{h}'(s,q,u) F_{\bar \mu}(u)\ud u \ud \nu(w)\\
&=\frac{1}{s}\int_{-1}^1 \bigg(\int_w^1 \partial_q\widetilde{h}'(s,q,u) \ud u -\int_{-1}^1 \partial_q\widetilde{h}'(s,q,u) F_{\bar \mu}(u)\ud u\bigg)\ud \nu(w).
\end{align*}
It follows that
\begin{equation*}
h(\mu+\epsilon \nu)-h(\mu)=\epsilon \int_{-1}^1 D_{\mu}(w)\ud \nu(w)+o(\epsilon)
\end{equation*}
for the continuous function $D_{\mu}:[-1,1]\to \R$ defined by
\begin{equation*}
D_{\mu}(w):=\partial_s\widetilde{h}'(s,q)+\frac{1}{s}\bigg(\int_w^1 \partial_q\widetilde{h}'(s,q,u) \ud u -\int_{-1}^1 \partial_q\widetilde{h}'(s,q,u) F_{\bar \mu}(u)\ud u\bigg).
\end{equation*}
We have used that $\partial_q\widetilde{h}'(s,q,\cdot)\in L^2[-1,1]$ to obtain the continuity of $D_{\mu}$. This establishes the Gateaux differentiability of $h$ at $\mu\in  \M_+$ and completes the proof.
\end{proof}

To account for the fact that we would like to establish Gateaux differentiability on a subset of $\Rpp\times \mathrm{Reg}(\M_+)$ that is dense in $\Rp\times \M_+$, we also introduce the space of paths
\begin{align}
\mathcal{R}[0,1]:=\big\{q \in \D[0,1] \mid q& \text{ is smooth with strictly positive derivative} \big\}.
\end{align}
This space of paths is in one-to-one correspondence with the probability measures in $\mathrm{Reg}(\M_+)$. To establish the Gateaux differentiability of a function defined on an infinite-dimensional Banach space, we will follow the arguments in \cite{chen2023free} and rely on the notion of a Gaussian null set. A Borel subset $B$ of a separable Banach space $X$ is said to be a \emph{Gaussian null set} if for every non-degenerate Gaussian measure $\mu$ on $X$, we have $\mu(B)=0$. Recall that a probability measure $\mu$ on $X$ is said to be a \emph{non-degenerate Gaussian measure} if for every non-zero $y \in X^*$, the measure $\mu \circ y^{-1}$ is a Gaussian measure with non-zero variance. Gaussian null sets appear when invoking the following two lemmas.

\begin{lemma}\label{l.SBM.not.Gaussian.crit}
Let $(w_n)_{n\geq 1}$ be a sequence in a separable Banach space $X$ that has dense linear span and satisfies $\lim_{n\to +\infty}\abs{w_n}=0$, and let $K$ be the closed convex hull of $\{0\}\cup\{w_n\mid n\geq 1\}$. For every $x\in X$, the set $x+K$ is not a Gaussian null set.
\end{lemma}

\begin{proof}
This is \cite[Lemma 2.5]{chen2023free} which itself is adapted from \cite[Lemma 3]{phelps1978}.
\end{proof}

\begin{lemma}\label{l.SBM.inf.dim.Rademacher}
Let $X$ be a separable Banach space. If a function $h:X\to \R$ is Lipschitz continuous, then it is Gateaux differentiable outside of a Gaussian null subset of $X$.
\end{lemma}

\begin{proof}
This is \cite[Theorem 2.6]{chen2023free} which itself is a special case of \cite[Theorem 6]{phelps1978}.
\end{proof}

\begin{lemma}\label{l.SBM.inf.dim.Rademacher.D01}
If $h:\Rp\times \Rp\times \D[0,1]\to \R$ is locally Lipschitz continuous in its first two variables and locally Lipschitz continuous with respect to the $L^2$-norm in its third variable, then it is Gateaux differentiable jointly in its three variables (in the sense of \eqref{e.SBM.GD.L2.density}) on a subset of $\Rpp\times \Rpp \times \mathcal{R}[0,1]$ that is dense in $\Rp\times \Rp \times \D[0,1]$.
\end{lemma}

\begin{proof}
For simplicity of notation, we will instead prove that a function $h:\D[0,1]\to \R$ that is locally Lipschitz continuous with respect to the $L^2$-norm is Gateaux differentiable on a subset of $\mathcal{R}[0,1]$ that is dense in $\D[0,1]$. Our proof generalizes with only modifications in notation to the setting of a function defined on $\Rp\times \Rp\times \D[0,1]$. Up to multiplying the function $h$ with a cutoff function that vanishes outside of a sufficiently large ball, we will also assume without loss of generality that $h:\D[0,1]\to \R$ is in fact uniformly Lipschitz continuous with respect to the $L^2$-norm.
 The proof proceeds in three steps. First, we construct a Lipschitz extension $\bar h:L^2[-1,1]\to \R$ of $h$, then we show that there exists a Gaussian null set $\mathcal{N}$ such that $h$ is Gateaux differentiable on $\mathcal{D}[0,1]\setminus \mathcal{N}$, and finally we prove that $\mathcal{R}[0,1]\setminus \mathcal{N}$ is dense in $\D[0,1]$.
\step{1: constructing the Lipschitz extension}
We denote by $P:L^2[-1,1]\to \D[0,1]$ the orthogonal projection onto the closed convex set $\D[0,1]$. We refer the reader to \cite[Exercise 2.16]{TD_JC_book} and its solution for a justification that this projection $P$ is well-defined based on the fact that $\D[0,1]$ is a closed convex subset of $L^2[-1,1]$. It is also shown there using the convexity of $\D[0,1]$ that, for every $q\in L^2[-1,1]$ and $p\in \D[0,1]$,
\begin{equation*}
\langle q-P(q), p- P(q)\rangle_{L^2}\leq 0.
\end{equation*}
It follows that for all $q,q'\in L^2[-1,1]$,
\begin{equation*}
\langle q-P(q), P(q')-P(q)\rangle_{L^2}\leq 0 \leq \langle q'-P(q'), P(q')-P(q)\rangle_{L^2}.
\end{equation*}
Rearranging and using the Cauchy-Schwarz inequality yields
\begin{equation*}
\abs{P(q)-P(q')}_{L^2}^2\leq \langle q'-q, P(q') -P(q)\rangle_{L^2}\leq \abs{q'-q}_{L^2}\abs{P(q)-P(q')}_{L^2}
\end{equation*}
from which it follows that
\begin{equation*}
\abs{P(q)-P(q')}_{L^2}\leq \abs{q'-q}_{L^2}.
\end{equation*}
This means that the extension $\bar h: L^2[-1,1]\to \R$ of $h$ defined by
\begin{equation*}
\bar h(q):=h(P(q))
\end{equation*}
is Lipschitz continuous.
\step{2: Gateaux differentiability outside a Gaussian null set}
By Lemma \ref{l.SBM.inf.dim.Rademacher}, there exists a Gaussian null set $\mathcal{N}\subset L^2[-1,1]$ such that $\bar h$ is Gateaux differentiable on $L^2[-1,1]\setminus \mathcal{N}$. We now fix $q\in \D[0,1]\setminus \mathcal{N}$, and prove that $h$ is Gateaux differentiable at $q$. Let $y\in (L^2[-1,1])^*\simeq L^2[-1,1]$ denote the Gateaux derivative of $\bar h$ at $q$. The limit $Dh(q;x)$ exists for every $x\in \mathrm{Adm}(\D[0,1],q)$, and it is equal to $\langle y,x\rangle_{L^2}$. Let us now assume that there exists $y'\in L^2[-1,1]$ with $\langle y',x\rangle_{L^2}=\langle y,x\rangle_{L^2}$ for all $x\in \mathrm{Adm}(\D[0,1],q)$, and prove that we must have $y'=y$. Observe that $\D[0,1]-q\subset \mathrm{Adm}(\D[0,1],q)$ by convexity of $\D[0,1]$. Moreover, the orthogonal complement of $\D[0,1]-q$ in $L^2[-1,1]$ is trivial. Indeed, suppose that $z\in L^2[-1,1]$ is such that $\langle z,q'\rangle_{L^2}=\langle z,q\rangle_{L^2}$ for all $q'\in \D[0,1]$. If $\mu'\in \Pr[-1,1]$ is such that $q'=F_{\mu'}$, then the Fubini-Tonelli theorem implies that
\begin{equation*}
\langle z,q'\rangle_{L^2}=\int_{-1}^1 z(u)\mu'[-1,u]\ud u=\int_{-1}^1 \int_{-1}^u z(u)\ud \mu'(v)\ud u=\int_{-1}^1 \int_v^1 z(u)\ud u \ud \mu'(v)=0.
\end{equation*}
Since $\D[0,1]$ is in one-to-one correspondence with $\Pr[-1,1]$, this means that for all $\mu'\in \Pr[-1,1]$,
\begin{equation*}
\int_{-1}^1 \int_v^1 z(u)\ud u \ud \mu'(v)=\langle z,q\rangle_{L^2}.
\end{equation*}
Choosing $\mu'=\delta_v$ for each $v\in [-1,1]$, and differentiating the resulting expression with respect to $v$ shows that $z=0$. This means that the closed linear span of $\D[0,1]-q$ is $L^2[-1,1]$ and that $\langle y',x\rangle_{L^2}=\langle y,x\rangle_{L^2}$ for all $x\in L^2[-1,1]$. It follows that $y'=y$, and therefore that $h$ is Gateaux differentiable on $\D[0,1]\setminus \mathcal{N}$.
\step{3: density of $\mathcal{R}[0,1]\setminus \mathcal{N}$ in $\D[0,1]$}
The set $\mathcal{R}[0,1]$ has dense linear span in $L^2[-1,1]$, so 
the separability of $L^2[-1,1]$ allows us to find a sequence $(w_n)_{n\geq 1}\subset \mathcal{R}[0,1]$ whose linear span is dense in $L^2[-1,1]$. Re-scaling each $w_n$, we can assume without loss of generality that $\abs{w_n}_{L^2}\leq 1$ and that $\lim_{n\to +\infty}\abs{w_n}_{L^2}=0$. Let $K$ be as in Lemma \ref{l.SBM.not.Gaussian.crit}, and observe that $K\subset \mathcal{R}[0,1]$ by convexity of $\mathcal{R}[0,1]$. At this point, fix $q\in \mathcal{R}[0,1]$ as well as $\epsilon >0$. Since $q+\epsilon K$ is not a Gaussian null set, we have that $q+\epsilon K\not\subset\mathcal{N}$. In particular, there exists $y\in q+\epsilon K\setminus \mathcal{N}$. Notice that for any $\kappa\in K$, we have $(q+\epsilon \kappa)'>0$ so in fact $y\in \mathcal{R}\setminus \mathcal{N}$ with $\abs{y-q}_{L^2}\leq \epsilon$. This establishes the density of $\mathcal{R}[0,1]\setminus \mathcal{N}$ in $\mathcal{R}[0,1]$. Together with the density of $\mathcal{R}[0,1]$ in $\D[0,1]$ which can be obtained through mollifiers, this completes the proof.
\end{proof}

\begin{proof}[Proof of Proposition \ref{p.regularity.differentiability}]
We denote by $\widetilde{f}':\Rp\times\Rp\times \D[0,1]\to\R$ the limit of the sequence of extended free energies $\smash{(\widetilde{F}_N')_{N\geq 1}}$ defined in \eqref{e.SBM.en.FE.s.on.D01} along the subsequence $\smash{(N_k)_{k\geq 0}}$. Invoking Lemma \ref{l.SBM.inf.dim.Rademacher.D01}, we can find a dense subset $\A$ of $\smash{\Rpp\times \Rpp\times \mathcal{R}[0,1]}$ on which $\smash{\widetilde{f}'}$ is Gateaux differentiable jointly in its three variables. By Lemma \ref{l.regularity.GD.h.tilde}, the limit free energy $f:\Rp\times \M_+\to \R$ is Gateaux differentiable jointly in its two variables on the set
\begin{equation}
\mathcal{B}:=\big\{\big(t,s\mathcal{F}^{-1}(q)\big)\mid (t,s,q)\in \mathcal{A}\big\}\subset \Rp\times \mathrm{Reg}(\M_+).
\end{equation}
We now show that this set is dense in $\Rp\times \M_+$. We fix $(t,\mu)\in \Rp\times \M_+$, and aim to show that there is a sequence $(t_n,\mu_n)_{n\geq 1}\subset \mathcal{B}$ such that $t_n\to t$ and $\mu_n\to \mu$ weakly. Let $s:=\mu[-1,1]$, and consider the point $(t,s,F_{\bar \mu})\in \Rp\times \Rp \times \D[0,1]$. There exists a sequence $(t_n,s_n,q_n)_{n\geq 1}\subset \mathcal{A}$ converging to $(t,s,F_{\bar\mu})$, and we define $\mu_n:=s_n\mathcal{F}^{-1}(q_n)$. By compactness, the sequence $\mu_n$ admits a subsequential limit. Since $F_{\mu_n}=s_nq_n\to sF_{\bar \mu}=F_\mu$ in $L^2$, the only possible subsequential limit is $\mu$ so indeed $\mu_n\to \mu$ weakly.
\end{proof}

\subsection{Convergence of the derivative free energy}
\label{subsec:regularity.convexity}

The convergence of the derivative free energy stated in Proposition \ref{p.regularity.cont.of.derivative} will be deduced from a convexity property of the free energy. Although the enriched free energy \eqref{e.intro.SBM.en.FE} is known not to be convex \cite{kireeva2023breakdown}, we now use the information-theoretic arguments in \cite[Proposition 4.1]{kireeva2023breakdown} to show that it is convex in certain directions. Together with a computation similar to that in \cite[Proposition~5.4]{chen2023free} we then establish Proposition \ref{p.regularity.cont.of.derivative}.

\begin{proposition}\label{p.regularity.en.FE.convex}
For all $(t, \mu), (s,\nu) \in \Rp\times \M_+$, the map $r\mapsto \bar F_N(t+ r s,\mu+r\nu)$ is convex.
\end{proposition}

\begin{proof}
The claim essentially follows from the general fact that if $S$ is a random variable and if $(X_n)_{n \ge 1}$ are i.i.d.\ conditionally on $S$, then the mutual information between $S$ and $(X_n)_{n \le \Poi(t)}$ is a concave function of $t$ (see e.g.\ the proof of \cite[Proposition 4.1]{kireeva2023breakdown}). In order to lighten the notation, we only prove the result for $s = 0$. (We note that for later purposes, all we really need is to apply Proposition~\ref{p.regularity.en.FE.convex} either with $s = 0$ or with $\nu = 0$; and the proof in the case $\nu = 0$ really boils down to the general observation above.)
Without loss of generality, we may assume that $\nu \in \Pr[-1,1]$. 
The superposition principle for Poisson point processes implies that, jointly in $\sigma\in \Sigma_N$,
\begin{equation*}
H_N^{\mu+r\nu}(\sigma) \stackrel{d}{=} H_N^\mu(\sigma)+H_N^{r\nu}(\sigma),
\end{equation*}
where we understand that the Hamiltonians $H_N^\mu(\sigma)$ and $H_N^{r\nu}(\sigma)$ are independent conditionally on the ground truth $\sigma^*$. Conditionally on the ground truth $\sigma^*$, the Poisson point processes $(\Lambda_i(N\mu))_{i\leq N}$, and the Poisson random variable $\Pi_t$, we define the ``prior'' measure on $\Sigma_N$ by
\begin{equation*}
P_{t,\mu}^*(\sigma):=\exp \big(H_N^t(\sigma)+H_N^{\mu}(\sigma)\big)P_N^*(\sigma)
\end{equation*}
in such a way that
\begin{equation*}
\bar F_N(t,\mu+r\nu)=\frac{1}{N}\E\log \int_{\Sigma_N}\exp H_N^{r\nu}(\sigma)\ud P_{t,\mu}^*(\sigma).
\end{equation*}
The measure $P_{t,\mu}^*$ is not necessarily a probability measure, but this will not have any significant effect.
The proof now proceeds in two steps. First, we relate the free energy $\bar F_N(t,\mu+r\nu)$ to a mutual information $I_N(r)$, and argue that $r\mapsto \bar F_N(t,\mu+r\nu)$ is convex in $r$ if and only if this mutual information is concave. Then, we compute the second derivative of the mutual information, and we show that it is non-positive.
\step{1: introducing a mutual information}
Conditionally on the ground truth $\sigma^*$, the Poisson point processes $(\Lambda_i(N\mu))_{i\leq N}$, and the Poisson random variable $\Pi_t$, we sample a signal $\sigma^\diamond$ from the normalized prior distribution $\bar P_{t,\mu}^*$ as well as a family $\Lambda^{r\nu}:=(\Lambda_i(Nr\nu))_{i\leq N}$ of Poisson point processes. Conditionally on the family of Poisson point processes $\Lambda^{r\nu}$, we sample the Bernoulli random variables $\widetilde{\mathbf{G}}_N:=\{\widetilde{G}_i^x \mid x\in \Lambda_i(Nr\nu),  i\leq N\}$ with conditional law
\begin{equation*}
\P\big\{\widetilde{G}_i^x=1\mid \sigma^\diamond, x\big\}:=\frac{c+\Delta \sigma^\diamond_ix}{N}.
\end{equation*}
The likelihood of the model is given by
\begin{equation*}
\P\big\{\widetilde{\mathbf{G}}_N=(\widetilde{G}_{i}^x)\mid \sigma^\diamond =\sigma\big\}=\prod_{i<N}\prod_{x\in \Lambda_i(Nr\nu)}\Big(\frac{c+\Delta \sigma_ix}{N}\Big)^{\widetilde{G}_{i}^x}\Big(1-\frac{c+\Delta \sigma_ix}{N}\Big)^{1-\widetilde{G}_{i}^x},
\end{equation*}
and Bayes' formula implies that the posterior of the model is the Gibbs measure
\begin{equation*}
\P\big\{\sigma^\diamond=\sigma \mid \widetilde{\mathbf{G}}_N=(\widetilde{G}_{i}^x)\big\}=\frac{\exp\big(H_N^\diamond(\sigma)\big) P_{t,\mu}^*(\sigma)}{\int_{\Sigma_N}\exp\big(H_N^\diamond(\tau)\big)\ud  P_{t,\mu}^*(\tau)}
\end{equation*}
associated with the Hamiltonian
\begin{equation}\label{e.regularity.en.FE.convex.H}
H_N^\diamond(\sigma):=\sum_{i\leq N}\sum_{x\in \Lambda_i(Nr\nu)}\log\bigg[\big(c+\Delta \sigma_ix\big)^{\widetilde{G}_{i}^x}\Big(1-\frac{c+\Delta \sigma_ix}{N}\Big)^{1-\widetilde{G}_{i}^x}\bigg].
\end{equation}
It follows that 
\begin{equation}\label{e.regularity.en.FE.convex.MI}
I_N(r):=\frac{1}{N}I\big(\widetilde{\mathbf{G}}_N;\sigma^\diamond\big)=\E\sum_{x\in \Lambda_1(Nr\nu)}\log\bigg[\big(c+\Delta \sigma_1^\diamond x\big)^{\widetilde{G}_{1}^x}\Big(1-\frac{c+\Delta \sigma_1^\diamond x}{N}\Big)^{1-\widetilde{G}_{1}^x}\bigg]-\overline{F}_N^\diamond(r)
\end{equation}
for the free energy defined on $\Sigma_N$ by
\begin{equation*}
\bar F_N^\diamond(r):= \frac{1}{N}\E\log \int_{\Sigma_N} \exp H_N^\diamond(\sigma)\ud P_{t,\mu}^*(\sigma).
\end{equation*}
Together with \cite[Proposition 5.12]{TD_JC_book} on Poisson averages, this implies that
\begin{equation*}
I_N(r)=Nr\E\int_{-1}^1 \log\bigg[\big(c+\Delta \sigma_1^\diamond x\big)^{\widetilde{G}_{1}^x}\Big(1-\frac{c+\Delta \sigma_1^\diamond x}{N}\Big)^{1-\widetilde{G}_{1}^x}\bigg]\ud \nu(x)-\overline{F}_N^\diamond(r),
\end{equation*}
so the free energy $\bar F_N^\diamond(r)$ is convex if and only if the mutual information $I_N(r)$ is concave. Averaging the free energy $\bar F_N^\diamond(r)$ over the randomness of the ground truth $\sigma^*$, the Poisson point processes $(\Lambda_i(N\mu))_{i\leq N}$, and the Poisson random variable $\Pi_t$ yields the free energy $\bar F_N(t,\mu+r\nu)$. It therefore suffices to prove that $r\mapsto I_N(r)$ is concave.
\step{2: showing the mutual information is concave}
The properties of Poisson point processes imply that the Hamiltonian \eqref{e.regularity.en.FE.convex.H} may be written as
\begin{equation*}
H_N^\diamond(\sigma)=\sum_{i\leq N}\sum_{k\leq \Pi_{i,r}}\log\bigg[\big(c+\Delta \sigma_ix_{i,k}\big)^{\widetilde{G}_{i,k}}\Big(1-\frac{c+\Delta \sigma_ix_{i,k}}{N}\Big)^{1-\widetilde{G}_{i,k}}\bigg],
\end{equation*}
where the random variables $(x_{i,k})_{i,k\geq 1}$ are i.i.d.\@ with law $\nu$,  the random variables $(\Pi_{i,r})_{i\leq N}$ are independent Poisson random variables with mean $Nr$, and $(\widetilde{G}_{i,k})_{i,k\geq 1}$ are Bernoulli random variables with conditional distribution
\begin{equation*}
\P\big\{\widetilde{G}_{i,k}=1\mid \sigma^\diamond, x_{i,k}\big\}:=\frac{c+\Delta \sigma^\diamond_ix_{i,k}}{N}.
\end{equation*}
This means that the mutual information \eqref{e.regularity.en.FE.convex.MI} depends on $r$ only through the Poisson random variable $\Pi_{i,r}$. With this in mind, given $i,j\leq N$ and $L_1,L_2\geq 0$, introduce the conditional mutual information functionals
\begin{equation*}
I_N^i(L_1):=\frac{1}{N}I\big(\widetilde{\mathbf{G}}_N;\sigma^\diamond\mid \Pi_{i,r}=L_1\big) \quad \text{and} \quad I_N^{i,j}(L_1,L_2):=I\big(\widetilde{\mathbf{G}}_N;\sigma^\diamond\mid \Pi_{i,r}=L_1, \Pi_{j,r}=L_2\big).
\end{equation*}
For every $\lambda>0$ and integer $m\geq 0$, we write
\begin{equation}\label{eqn: SBM Poisson density}
\pi(\lambda,m)=\frac{\lambda^m}{m!}\exp(-\lambda)
\end{equation}
to denote the probability for a Poisson random variable with parameter $\lambda$ to be equal to $m$.
A direct computation leveraging the product rule shows that
\begin{equation*}
\partial_rI_N(r)=\sum_{i\leq N}\sum_{L\geq 0}\pi(rN,L)\big(I_N^i(L+1)-I_N^i(L)\big).
\end{equation*}
Another direct computation leveraging the product rule reveals that 
\begin{align*}
\partial_r^2 I_N(r)=&N\sum_{i\leq N}\sum_{L_1\geq 0}\pi(rN,L_1)\big(I_N^i(L_1+2)-2I_N^i(L_1+1)+I_N^i(L_1)\big)\\
&+N\sum_{i\leq N}\sum_{j\neq i}\sum_{L_1,L_2\geq 0}\pi(rN,L_1)\pi(rN,L_2)\big(I_N^{i,j}(L_1+1,L_2+1)-I_N^{i,j}(L_1+1,L_2)\\
&\qquad\qquad\qquad\qquad\qquad\qquad\qquad\qquad\qquad\quad-I_N^{i,j}(L_1,L_2+1)+I_N^{i,j}(L_1,L_2)\big).
\end{align*}
It therefore suffices to show that for fixed $i\neq j\leq N$ and $L_1,L_2\geq 0$,
\begin{align}
I_N^i(L_1+2)-2I_N^i(L_1+1)+I_N^i(L_1)&\leq 0, \label{e.regularity.en.FE.convex.in.mass.goal.0}\\
I_N^{i,j}(L_1+1,L_2+1)-I_N^{i,j}(L_1+1,L_2)-I_N^{i,j}(L_1,L_2+1)+I_N^{i,j}(L_1,L_2)&\leq 0.\label{e.regularity.en.FE.convex.in.mass.goal.1}
\end{align}
We will only prove \eqref{e.regularity.en.FE.convex.in.mass.goal.1} as the proof of \eqref{e.regularity.en.FE.convex.in.mass.goal.0} is similar. 
To alleviate notation, conditionally on the events $\Pi_{i,r}=L_1$ and $\Pi_{j,r}=L_2$ introduce the random vectors
\begin{align*}
\widetilde{\mathbf{G}}_N'&:=\big\{\widetilde{G}_{m,k}\mid  m\neq i,\,\,\, m\neq j \text{ and } k\leq \Pi_{m,r}\big\}\quad \text{and} \quad Z:=\big((\widetilde{G}_{i,\ell})_{\ell\leq L_1}, (\widetilde{G}_{j,\ell})_{\ell\leq L_2}, \widetilde{\mathbf{G}}_N'\big).
\end{align*}
By the chain rule for mutual information,
\begin{align*}
I_N^{i,j}(L_1+1,L_2+1)&=I((\widetilde{G}_{i,L_1+1}, \widetilde{G}_{j,L_2+1}),Z;\sigma^\diamond)=I((\widetilde{G}_{i,L_1+1}, \widetilde{G}_{j,L_2+1});\sigma^\diamond\mid Z)+I(Z;\sigma^\diamond),\\
I_N^{i,j}(L_1+1,L_2)&=I(\widetilde{G}_{i,L_1+1},Z;\sigma^\diamond)=I(\widetilde{G}_{i,L_1+1};\sigma^\diamond\mid Z)+I(Z;\sigma^\diamond),\\
I_N^{i,j}(L_1,L_2+1)&=I(\widetilde{G}_{j,L_2+1},Z;\sigma^\diamond)=I(\widetilde{G}_{j,L_2+1};\sigma^\diamond\mid Z)+I(Z;\sigma^\diamond).
\end{align*}
It follows that \eqref{e.regularity.en.FE.convex.in.mass.goal.1} is equivalent to the inequality
\begin{equation}\label{e.regularity.en.FE.convex.in.mass.goal.2}
I((\widetilde{G}_{i,L_1+1}, \widetilde{G}_{j,L_2+1});\sigma^\diamond\mid Z)-I(\widetilde{G}_{i,L_1+1};\sigma^\diamond\mid Z)-I(\widetilde{G}_{j,L_2+1};\sigma^\diamond\mid Z)\leq 0.
\end{equation}
Another application of the chain rule for mutual information reveals that
\begin{equation*}
I((\widetilde{G}_{i,L_1+1}, \widetilde{G}_{j,L_2+1});\sigma^\diamond\mid Z)=I(\widetilde{G}_{i,L_1+1};\sigma^\diamond\mid Z)+I( \widetilde{G}_{j,L_2+1};\sigma^\diamond\mid Z, \widetilde{G}_{i,L_1+1})
\end{equation*}
and that the second term on the right side of this expression is given by
\begin{align*}
I( \widetilde{G}_{j,L_2+1};&(\sigma^\diamond, \widetilde{G}_{i,L_1+1})\mid Z)-I(\widetilde{G}_{j,L_2+1}; \widetilde{G}_{i,L_1+1}\mid Z)\\
&=I( \widetilde{G}_{j,L_2+1};\sigma^\diamond\mid Z)+I( \widetilde{G}_{j,L_2+1};\widetilde{G}_{i,L_1+1}\mid Z, \sigma^\diamond)-I(\widetilde{G}_{j,L_2+1}; \widetilde{G}_{i,L_1+1}\mid Z).
\end{align*}
It follows that the left side of \eqref{e.regularity.en.FE.convex.in.mass.goal.2} is
\begin{equation*}
I( \widetilde{G}_{j,L_2+1};\widetilde{G}_{i,L_1+1}\mid Z, \sigma^\diamond)-I(\widetilde{G}_{j,L_2+1}; \widetilde{G}_{i,L_1+1}\mid Z)=-I(\widetilde{G}_{j,L_2+1}; \widetilde{G}_{i,L_1+1}\mid Z)\leq 0,
\end{equation*}
where we have used that, conditionally on $\sigma^\diamond$, the random variables $\smash{\widetilde{G}_{j,L_2+1}}$, $\smash{\widetilde{G}_{i,L_1+1}}$, and $Z$ are independent to assert that $\smash{I( \widetilde{G}_{j,L_2+1};\widetilde{G}_{i,L_1+1}\mid Z, \sigma^\diamond)=0}$. This establishes \eqref{e.regularity.en.FE.convex.in.mass.goal.2} and completes the proof.
\end{proof}

\begin{proof}[Proof of Proposition \ref{p.regularity.cont.of.derivative}]
The claim concerning the convergence of the derivative with respect to the $t$ variable is a direct consequence of the convexity in $t$ ensured by Proposition~\ref{p.regularity.en.FE.convex}, see for instance \cite[Proposition~2.15]{TD_JC_book}. We now turn to the convergence of the derivative with respect to $\mu$. Fix $r,\eta>0$ and a probability measure $\nu \in \Pr[-1,1]\cap \mathrm{Reg}(\M_+)$. We define the measure $\nu_\eta\in \mathrm{Reg}(\M_+)$ by $\nu_\eta:=\eta\nu$. The convexity property of the enriched free energy in Proposition \ref{p.regularity.en.FE.convex} gives the upper bound
\begin{equation*}
\bar F_N(t,\mu+r\nu_\eta)=\bar F_N(t,\mu+(r\cdot 1+(1-r)\cdot 0)\nu_\eta)\leq r\bar F_N(t,\mu+\nu_\eta)+(1-r)\bar F_N(t,\mu).
\end{equation*}
Subtracting $\bar F_N(t,\mu)$ from both sides of this inequality, dividing by $r$, and letting $r$ tend to zero yields
\begin{equation*}
\eta D_\mu \bar F_N(t,\mu; \nu)\leq \bar F_N(t,\mu+\nu_\eta)-\bar F_N(t,\mu).
\end{equation*}
If $D_{t,\mu}$ denotes a weak subsequential limit of $D_\mu \bar F_{N_K}(t,\mu,\cdot)$, then letting $K$ tend to infinity, dividing by $\eta$, and letting $\eta$ tend to zero gives
\begin{equation*}
D_{t,\mu}(\nu)\leq D_\mu f(t,\mu; \nu).
\end{equation*}
Repeating the argument with $\eta$ replaced by $-\eta$, which is possible as $\mu,\nu\in \mathrm{Reg}(\M_+)$, shows that for any $\nu \in \mathrm{Reg}(\M_+)$, we have $D_{t,\mu}(\nu)=D_\mu f(t,\mu;\nu)$. By density of $\mathrm{Reg}(\M_+)$ in $\M_+$, this completes the proof.
\end{proof}

\section{Multioverlap concentration through perturbation}
\label{sec:concentration}

In the next section we will perform the cavity computation that gives rise to the functional \eqref{e.intro.HJ.functional} and leads to the critical point representation in Theorem \ref{t.intro.main}. This cavity computation will require an understanding of the asymptotic behavior of the array $(\sigma_i^\ell)_{i,\ell\geq 1}$ of i.i.d.\@ replicas sampled from the Gibbs measure \eqref{e.SBM.Gibbs}. In this section, we leverage the main result in \cite{barbier2022strong} to argue that, up to a small perturbation of the enriched Hamiltonian \eqref{e.intro.SBM.en.H} that does not affect the limit of the enriched free energy \eqref{e.intro.SBM.en.FE}, the asymptotic spin array $(\sigma_i^\ell)_{i,\ell\geq 1}$ is generated by some probability measure $\nu \in \M_p$ according to \eqref{gen.spin.array.1}-\eqref{gen.spin.array.2}. In particular, as we will see in Lemma \ref{l.multioverlap.asymptotic.average}, this means that the average of a product of finitely many replicas is asymptotically equal to a product of moments of~$\nu$. More precisely, given any finite set of $n$ replicas and collection $\{\mathcal{C}_\ell\}_{\ell\leq n}$ of finite indices,
\begin{equation}\label{e.multioverlap.asymptotic.average.0}
\lim_{N\to +\infty}\E\prod_{\ell\leq n}\Big\langle \prod_{i\in \mathcal{C}_\ell}\sigma_i^\ell \Big\rangle=\prod_{\ell\leq n}\E x_\ell^{\abs{\mathcal{C}_\ell}},
\end{equation}
where the $(x_\ell)_{\ell \ge 1}$ are i.i.d.\ with law $\nu$. 
This observation will be crucial in Lemma \ref{l.cavity.AN.limit}, where we study the asymptotic behavior of the cavity representation and first arrive at the functional \eqref{e.intro.HJ.functional}.

The perturbation Hamiltonians that will enforce the rigid asymptotic structure described by \eqref{gen.spin.array.1}-\eqref{gen.spin.array.2} depend on a perturbation parameter $\lambda:=(\lambda_k)_{k\geq 0}$ with $\lambda_k\in [2^{-k-1}, 2^{-k}]$ for $k\geq 0$, and they enforce the concentration of the multioverlaps
\begin{equation}\label{e.multioverlaps.def}
R_{1,\ldots,n}:=\frac{1}{N}\sum_{i\leq N}\sigma_i^{1}\cdots \sigma_i^{n}.
\end{equation}
As will be explained below, here $(\sigma_i^\ell)_{i,\ell\geq 1}$ denote i.i.d.\@ replicas from the Gibbs measure associated with the sum of the Hamiltonian \eqref{e.intro.SBM.en.H} and the two perturbation Hamiltonians that we now introduce. The first perturbation Hamiltonian will ensure the concentration of the overlap $R_{1,2}$. Given a sequence $(\epsilon_N)_{N\geq 1}$ with $\epsilon_N:=N^\gamma$ for some $-1/8<\gamma<0$ and a standard Gaussian vector $Z_0:=(Z_{0,1},\ldots,Z_{0,N})$ in $\R^N$, we introduce the Gaussian perturbation Hamiltonian
\begin{equation}\label{e.multioverlap.gaussian.pert}
H_N^{\mathrm{gauss}}(\sigma,\lambda_0):=\HH_0:=\sum_{i\leq N}\big(\lambda_0\epsilon_N \sigma^*_i\sigma_i+\sqrt{\lambda_0\epsilon_N}Z_{0,i}\sigma_i\big)
\end{equation}
associated with the task of recovering the signal $\sigma^*$ from the data
\begin{equation}
Y^{\mathrm{gauss}}:=\sqrt{\lambda_0\epsilon_N}\sigma^*+Z_0.
\end{equation}
Notice that
\begin{equation}\label{e.multioverlap.epsilon.sequence}
1\geq \epsilon_N \to 0 \quad \text{and} \quad N\epsilon_N\to +\infty.
\end{equation}
The second perturbation Hamiltonian will ensure the validity of the Franz-de Sanctis identities \cite{barbier2022strong, dominguez2024mutual} which together with the concentration of the overlap $R_{1,2}$ give the concentration of all the multioverlaps \eqref{e.multioverlaps.def}. We consider a sequence $(s_N)_{N\geq 1}$ with $s_N:=N^\eta$ for some $4/5<\eta<1$, so that in particular
\begin{equation}\label{e.multioverlap.s.sequence}
\frac{s_N}{N}\to 0\quad \text{and} \quad \frac{s_N}{\sqrt{N}}\to +\infty.
\end{equation}
We fix a sequence of i.i.d.\@ random variables $(\pi_k)_{k\geq 1}$ with $\Poi(s_N)$ distribution as well as a sequence $e:=(e_{jk})_{j,k\geq 1}$ of random variables with $\Exp(1)$ distribution and a family $(i_{jk})_{j,k\geq 1}$ of random indices uniformly sampled from the set $\{1,\ldots,N\}$. We define the exponential perturbation Hamiltonian by
\begin{equation}\label{e.multioverlap.exponential.pert}
H_N^{\mathrm{exp}}(\sigma):=\sum_{k\geq 1} \HH_k, \quad \text{where}\quad \HH_k:=\sum_{j\leq \pi_k}\bigg(\log(1+\lambda_k\sigma_{i_{jk}}\big)-\frac{\lambda_k e_{jk}\sigma_{i_{jk}}}{1+\lambda_k\sigma_{i_{jk}}^*}\bigg).
\end{equation}
This is the Hamiltonian associated with the task of recovering the signal $\sigma^*$ from the independently generated data
\begin{equation}
Y^{\mathrm{exp}}:=\bigg\{\frac{e_{jk}}{1+\lambda_k\sigma^*_{i_{jk}}} \mid  j\leq \pi_k \text{ and } k\geq 1\bigg\}.
\end{equation}
For each $(t,\mu)\in \Rp\times \M_+$, the perturbed enriched Hamiltonian is the sum of the enriched Hamiltonian \eqref{e.intro.SBM.en.H} and the perturbation Hamiltonians just defined,
\begin{equation}\label{e.multioverlap.pert.H}
H_N^{t,\mu}(\sigma,\lambda):=H_N^{t,\mu}(\sigma)+H_N^{\mathrm{gauss}}(\sigma,\lambda_0)+H_N^{\mathrm{exp}}(\sigma,\lambda).
\end{equation}
As usual, we denote by
\begin{equation}\label{e.multioverlap.pert.FE}
F_N(t,\mu,\lambda):=\frac{1}{N}\log\int_{\Sigma_N}\exp H_N^{t,\mu}(\sigma,\lambda)\ud P_N^*(\sigma) \quad \text{and} \quad \bar F_N(t,\mu,\lambda):=\E F_N(t,\mu,\lambda)
\end{equation}
its associated free energy functionals, and by $\langle \cdot \rangle$ its corresponding Gibbs average. This means that for any bounded and measurable function $f=f(\sigma^1,\ldots,\sigma^n)$ of finitely many replicas,
\begin{equation}\label{e.multioverlap.Gibbs.pert}
 \langle f(\sigma^1, \ldots, \sigma^n) \rangle := \langle f\rangle   :=\frac{\int_{\Sigma_N^n}f(\sigma^1,\ldots,\sigma^n)\prod_{\ell\leq n}\exp H_N^{t,\mu}(\sigma^\ell,\lambda)\ud P_N^*(\sigma^\ell)}{\big(\int_{\Sigma_N} \exp H_N^{t,\mu}(\sigma,\lambda)\ud P_N^*(\sigma)\big)^n}.
\end{equation}
Since the Gibbs measure associated with the perturbed Hamiltonian \eqref{e.multioverlap.pert.H} is a conditional expectation, it satisfies the Nishimori identity \eqref{e.SBM.Nishimori}.

\begin{lemma}\label{l.multioverlap.pert.no.effect.FE}
For all $(t,\mu)\in \Rp\times \M_+$ and perturbation parameter $\lambda$, the enriched free energy \eqref{e.intro.SBM.en.FE} and the perturbed free energy \eqref{e.multioverlap.pert.FE} are asymptotically equivalent,
\begin{equation}
\lim_{N\to +\infty}\big\lvert \overline{F}_N(t,\mu,\lambda)-\overline{F}_N(t,\mu)\big\rvert=0.
\end{equation}
\end{lemma}

\begin{proof}
This corresponds to \cite[Lemma 4.1]{dominguez2024mutual} but we provide a proof for completeness. A direct computation reveals that
\begin{equation*}
\big\lvert \overline{F}_N(t,\mu,\lambda)-\overline{F}_N(t,\mu)\big\rvert\leq \frac{1}{N}\E \max_{\sigma \in \Sigma_N}\big\lvert H_N^{\mathrm{gauss}}(\sigma,\lambda_0)\big\rvert+\frac{1}{N}\E \max_{\sigma \in \Sigma_N}\big\lvert H_N^{\mathrm{exp}}(\sigma,\lambda)\big\rvert.
\end{equation*}
For any spin configuration $\sigma\in \Sigma_N$,
\begin{equation*}
\big\lvert H_N^{\mathrm{gauss}}(\sigma,\lambda_0)\big\rvert\leq N\epsilon_N+\sqrt{\epsilon_N}\sum_{i\leq N}\abs{Z_{0,i}}
\end{equation*}
and
\begin{equation*}
\big\lvert H_N^{\mathrm{exp}}(\sigma,\lambda)\big\rvert\leq \sum_{1\leq k\leq K'}\sum_{j\leq \pi_k}\Big(\log(1+\lambda_k)+\frac{\lambda_ke_{jk}}{1-\lambda_k}\Big).  
\end{equation*}
Since these bounds are uniform in $\sigma$, it follows that
\begin{equation*}
\big\lvert \overline{F}_N(t,\mu,\lambda)-\overline{F}_N(t,\mu)\big\rvert\leq \epsilon_N+\sqrt{\epsilon_N}\E Z_{0,1}+\frac{s_N}{N}\sum_{k\geq 1}\Big(\log(1+\lambda_k)+\frac{\lambda_k}{1-\lambda_k}\Big).
\end{equation*}
The third term was obtained by taking the expectation with respect to the randomness of $e$, and then with respect to the randomness of $(\pi_k)_{k\geq 1}$. Leveraging \eqref{e.multioverlap.epsilon.sequence}-\eqref{e.multioverlap.s.sequence} to let $N$ tend to infinity completes the proof.
\end{proof}

\begin{lemma}\label{l.multioverlap.average.conc}
For every $n\geq 1$, the multioverlap $R_{1,\ldots,n}$ in \eqref{e.multioverlaps.def} associated with the perturbed Gibbs measure \eqref{e.multioverlap.Gibbs.pert} concentrates on average over the perturbation parameters,
\begin{equation}\label{e.multioverlap.average.conc}
\lim_{N\to+\infty}\E_\lambda\E\big\langle( R_{1,\ldots,n}-\E\langle R_{1,\ldots,n}\rangle)^2\big\rangle=0,
\end{equation}
where $\E_\lambda$ denotes the average with respect to the uniform random variables $\lambda_k$ on $[2^{-k-1},2^{-k}]$.
\end{lemma}

\begin{proof}
We introduce the concentration function
\begin{equation}
v_N:=\frac{1}{N}\sup\Big\{\E\big\lvert F_N(t,\mu,\lambda)-\bar F_N(t,\mu,\lambda)\big\rvert^2\mid \lambda_k\in [2^{-k-1},2^{-k}] \text{ for all } k\geq 0\Big\},
\end{equation}
and think of the perturbation parameters $\lambda:=(\lambda_k)_{k\geq 0}$ as being sampled independently with $\lambda_k$ uniform on $[2^{-k-1}, 2^{-k}]$ for $k\geq 0$. The concentration function $\nu_N$ is bounded from above by a constant independent of $N$ by \cite[Proposition B.4]{dominguez2024mutual}, so \eqref{e.multioverlap.average.conc} follows from \cite[Theorem 2.2]{barbier2022strong}. This completes the proof.
\end{proof}

Arguing as in \cite[Lemma 3.3]{PanSKB}, it is now possible to find a deterministic sequence $(\lambda^N)_{N\geq 1}$ along which, for all $n\geq 1$,
\begin{equation}\label{e.multioverlap.conc}
\lim_{N\to +\infty}\E\big\langle( R_{1,\ldots,n}-\E\langle R_{1,\ldots,n}\rangle_N)^2\big\rangle_N=0.
\end{equation}
Here, the Gibbs average $\langle \cdot \rangle_N$ for a system of size $N$ is associated with the perturbation parameter $\lambda^N$, and the multioverlap $R_{1,\ldots,n}$ is for i.i.d.\@ replicas $(\sigma_i^\ell)_{i,\ell\geq 1}$ sampled from this perturbed Gibbs measure. The multioverlap concentration \eqref{e.multioverlap.conc} ensures that the asymptotic spin array is generated by some probability measure $\nu \in \M_p$ according to \eqref{gen.spin.array.1}-\eqref{gen.spin.array.2}.

\begin{proposition}\label{p.multioverlap.asymptotic.spin}
Let $(\lambda^N)_{N\geq 1}$ be a deterministic sequence along which \eqref{e.multioverlap.conc} holds for all $n\geq 1$, and denote by $(\sigma_i^\ell)_{i,\ell\geq 1}$ i.i.d.\@ replicas sampled from the Gibbs measure $\langle\cdot\rangle_N$ in \eqref{e.multioverlap.Gibbs.pert} associated with the perturbation parameter $\lambda^N$. For any subsequence $(N_k)_{k\geq 1}$ along which the distribution of the spin array $(\sigma_i^\ell)_{i,\ell\geq 1}$ converges in the sense of finite-dimensional distributions, there exists a probability measure $\nu\in\M_p$ (which may depend on the subsequence) such that the asymptotic spin array is generated by $\nu$ according to \eqref{gen.spin.array.1}-\eqref{gen.spin.array.2}.
\end{proposition}

\begin{proof}
This is immediate from \cite[Corollary 2.3]{barbier2022strong}. We have implicitly used that $\E\langle R_1\rangle = \E \sigma_1^*=\m$ by the Nishimori identity to assert that $\nu\in \M_p$.
\end{proof}

The two-step process \eqref{gen.spin.array.1}-\eqref{gen.spin.array.2} to generate the asymptotic spin array from the measure $\nu \in \M_p$ allows us to write the asymptotic spin array $(\sigma_i^\ell)_{i,\ell\geq 1}$ explicitly in terms of the i.i.d.\@ random variables $(x_i)_{i\geq 1}$ sampled from $\nu$ and a family of i.i.d.\@ uniform random variables $(u_{i,\ell})_{i,\ell\geq 1}$ on $[0,1]$ also independent of $(x_i)_{i\geq 1}$,
\begin{equation}\label{e.multioverlap.asymptotic.spin}
\sigma_i^\ell:=2\1\bigg\{ u_{i,\ell}\leq \frac{1+x_i}{2}\bigg\}-1.
\end{equation}
From this representation of the asymptotic spins we can readily establish \eqref{e.multioverlap.asymptotic.average.0}.

\begin{lemma}\label{l.multioverlap.asymptotic.average}
Let $(\lambda^N)_{N\geq 1}$ be a deterministic sequence along which \eqref{e.multioverlap.conc} holds for all $n\geq 1$, denote by $(\sigma_i^\ell)_{i,\ell\geq 1}$ i.i.d.\@ replicas sampled from the Gibbs measure $\langle\cdot\rangle_N$ in \eqref{e.multioverlap.Gibbs.pert} associated with the perturbation parameter $\lambda^N$, and suppose that the distributional limit of the spin array $(\sigma_i^\ell)_{i,\ell\geq 1}$ is generated by a probability measure $\nu\in\M_p$ according to \eqref{gen.spin.array.1}-\eqref{gen.spin.array.2}. If $(x_\ell)_{\ell\geq 1}$ denote i.i.d.\@ samples from $\nu$, then for any finite set of $n$ replicas and collection $\{\mathcal{C}_\ell\}_{\ell\leq n}$ of finite indices,
\begin{equation}
\lim_{N\to +\infty}\E\prod_{\ell\leq n}\Big\langle \prod_{i\in \mathcal{C}_\ell}\sigma_i^\ell\Big\rangle_N=\prod_{\ell\leq n}\E x_\ell^{\abs{\mathcal{C}_\ell}}.
\end{equation}
\end{lemma}

\begin{proof}
Let $\mathcal{C}:=\{(i,\ell)\mid \ell\leq n \text{ and } i\in \mathcal{C}_\ell\}$, and observe that
\begin{equation*}
\E\prod_{\ell\leq n}\Big\langle \prod_{i\in \mathcal{C}_\ell}\sigma_i^\ell\Big\rangle_N=\E\Big\langle \prod_{(i,\ell)\in \mathcal{C}}\sigma_i^\ell\Big\rangle_N.
\end{equation*}
Recalling the representation \eqref{e.multioverlap.asymptotic.spin} of an asymptotic spin in terms of a family $(u_{i,\ell})_{i,\ell\geq 1}$ of i.i.d.\@ uniform random variables on $[0,1]$, and letting $N$ tend to infinity in the previous display reveals that
\begin{equation*}
\lim_{N\to +\infty}\E\prod_{\ell\leq n}\Big\langle \prod_{i\in \mathcal{C}_\ell}\sigma_i^\ell\Big\rangle_N=\E\prod_{(i,\ell)\in \mathcal{C}}\bigg(2\1\bigg\{ u_{i,\ell}\leq \frac{1+x_i}{2}\bigg\}-1\bigg)=\E\prod_{(i,\ell)\in \mathcal{C}}x_i,
\end{equation*}
where the independence of the $(u_{i,\ell})_{i,\ell\geq 1}$ has played its part in the second equality. Leveraging the independence of the $(x_i)_{i\geq 1}$ completes the proof.
\end{proof}

This formula to compute the asymptotic moments of the spin array will be applied to a slight modification of the Hamiltonian \eqref{e.intro.SBM.en.H} in Section \ref{sec:cavity}. The arguments to justify the formula in this modified setting are identical to those just given.

\begin{remark}
We discuss briefly how to extend the results of the present paper and cover the more general case when the measure $P^*$ is arbitrary with support in $[-1,1]$ in place of $\{-1,1\}$, with the understanding that we stick with the formula \eqref{e.intro.SSBM.edge.probs} to compute the link probabilities. In this case, the description of the limit law of $(\sigma^\ell_i)_{i,\ell \ge 1}$ must naturally be modified from \eqref{gen.spin.array.1}-\eqref{gen.spin.array.2}, since these explicitly refer to $\pm 1$ random variables. Instead, we would first need to sample i.i.d.\ random probability measures $(\mu_i)_{i \ge 1}$, each with law $\Xi \in \Pr(\Pr[-1,1])$, and then conditionally on $(\mu_i)_{i \ge 1}$, to sample $\sigma_i^\ell$ according to the law $\mu_i$ independently over $i, \ell \ge 1$. One key observation though is that Lemma~\ref{l.multioverlap.asymptotic.average} remains valid as stated, provided that we understand that $\nu$ is the law of the mean value of $\mu_1$, that is,
\begin{equation*}  
\nu = \text{Law} \Big( \int x \, \ud \mu_1(x) \Big).
\end{equation*}
Moreover, the cavity computations ultimately only require us to be able to evaluate moments of the form described by Lemma~\ref{l.multioverlap.asymptotic.average}. So even in this more general setting, we do not really need to keep track of the more complex object $\Xi \in \Pr(\Pr[-1,1])$, and can proceed as in the $\pm 1$ case by focusing on this simpler measure $\nu \in \M_p$. 
\end{remark}

\section{Cavity computation}
\label{sec:cavity}

In this section we perform the cavity computation that gives rise to the functional \eqref{e.intro.HJ.functional} and leads to the critical point representation in Theorem \ref{t.intro.main}. The idea behind the cavity computation is to determine the effect of perturbing the Hamiltonian \eqref{e.multioverlap.pert.H} for the system of size $N$ by a single additional spin. To emphasize this additional spin, we write $\rho:=(\epsilon,\sigma)\in \Sigma_{N+1}$ for $\sigma\in \Sigma_N$ and $\epsilon\in \{-1,+1\}$. Similarly, we write $\rho^*:=(\epsilon^*,\sigma^*)\in \Sigma_{N+1}$ for $\sigma^*\in \Sigma_N$ and $\epsilon^*\in \Sigma_1$ to denote the ground truth sampled from $P_{N+1}^*$. Given a pair $(t,\mu)\in \Rp\times \M_+$ and a perturbation parameter $\lambda=(\lambda_k)_{k\geq 0}$, we will compare the systems with Hamiltonians $H_{N+1}^{t,\mu}(\epsilon,\sigma,\lambda)$ and $H_N^{t,\mu}(\sigma,\lambda)$. This comparison will be done at the level of the free energies, and the re-scaled difference between $\bar F_{N+1}(t,\mu,\lambda)$ and $\bar F_N(t,\mu,\lambda)$ will be expressed as the difference between the modified Gibbs averages of two quantities that we may call the ``cavity fields''. We now describe the modified Hamiltonian associated with this modified Gibbs average as well as the two cavity fields.

For each $t\geq 0$, we introduce a random variable $\smash{\Pi_t\sim \Poi t\tbinom{N}{2}}$ as well as an independent family of i.i.d.\@ random matrices $\smash{(\td G^k)_{k \geq 1}}$ each having conditionally independent entries $\smash{(\td G_{i,j}^k)_{i,j\leq N}}$ taking values in $\{0,1\}$ with conditional distribution
\begin{equation}
\P\big\{\td G_{i,j}^k=1 \mid \sigma^*\big\}:=\frac{c+\Delta \sigma_{i}^*\sigma_{j}^*}{N+1}.
\end{equation}
Given a collection of random indices $\smash{(i_k,j_k)_{k\geq 1}}$ sampled uniformly at random from $\smash{\{1,\ldots,N\}^2}$, independently of the other random variables, we define the modified time-dependent Hamiltonian $\mathcal{H}_N^t$ on $\Sigma_N$ by
\begin{equation}\label{e.cavity.H.t}
\mathcal{H}_N^t(\sigma):=\sum_{k\leq \Pi_t}\log\bigg[\big(c+\Delta \sigma_{i_k}\sigma_{j_k}\big)^{\widetilde{G}_{i_k,j_k}^k}\Big(1-\frac{c+\Delta \sigma_{i_k}\sigma_{j_k}}{N+1}\Big)^{1-\widetilde{G}^k_{i_k,j_k}}\bigg].
\end{equation}
Similarly, for each $\mu \in \M_+$, consider a sequence $(\Lambda_i((N+1)\mu))_{i\leq N}$ of independent Poisson point processes with mean measure $(N+1)\mu$, and define the modified measure-dependent Hamiltonian $\mathcal{H}_N^\mu$ on $\Sigma_N$ by
\begin{equation}\label{e.cavity.H.mu}
\mathcal{H}_N^\mu(\sigma):=\sum_{i\leq N}\sum_{x\in \Lambda_i((N+1)\mu)}\log \bigg[\big(c+\Delta \sigma_{i}x\big)^{\td G^x_{i}}\Big(1-\frac{c+\Delta \sigma_{i}x}{N+1}\Big)^{1-\td G^x_{i}}\bigg],
\end{equation}
where $(\td G_i^x)_{i\leq N}$ are conditionally independent random variables taking values in $\{0,1\}$ with conditional distribution
\begin{equation}
\P\big\{\td G_i^x=1\mid \sigma^*, x\big\}:=\frac{c+\Delta \sigma_i^* x}{N+1}.
\end{equation}
For each pair $(t,\mu)\in \Rp\times \M_+$, the modified enriched Hamiltonian is the sum of the modified time-dependent and measure-dependent Hamiltonians,
\begin{equation}\label{e.cavity.en.H}
\mathcal{H}_N^{t,\mu}(\sigma):=\mathcal{H}_N^t(\sigma)+\mathcal{H}_N^\mu(\sigma),
\end{equation}
and, given a perturbation parameter $\lambda:=(\lambda_k)_{k\geq 0}$, its perturbed version is
\begin{equation}\label{e.cavity.pert.mod.H}
\mathcal{H}_N^{t,\mu}(\sigma,\lambda):=\mathcal{H}_N^{t,\mu}(\sigma)+H_N^{\mathrm{gauss}}(\sigma,\lambda_0)+H_N^{\mathrm{exp}}(\sigma,\lambda) 
\end{equation}
for the perturbation Hamiltonians \eqref{e.multioverlap.gaussian.pert} and \eqref{e.multioverlap.exponential.pert}. As usual, we write
\begin{equation}\label{e.cavity.mod.FE}
\bar F_N'(t,\mu):=\frac{1}{N}\E\log \int_{\Sigma_N}\exp \mathcal{H}_N^{t,\mu}(\sigma)\ud P_N^*(\sigma)
\end{equation}
for the modified free energy, and for each perturbation parameter $\lambda=(\lambda_k)_{k\geq 0}$, we denote by
\begin{equation}\label{e.cavity.mod.pert.FE}
\bar F_N'(t,\mu,\lambda):=\frac{1}{N}\E\log \int_{\Sigma_N}\exp \mathcal{H}_N^{t,\mu}(\sigma,\lambda) \ud P_N^*(\sigma)
\end{equation}
its perturbed version. As in \eqref{e.multioverlap.Gibbs.pert}, we write $\langle \cdot\rangle'$ for the Gibbs average with respect to the perturbed and modified Hamiltonian \eqref{e.cavity.pert.mod.H}. This means that for any bounded and measurable function $f=f(\sigma^1,\ldots,\sigma^n)$ of finitely many replicas,
\begin{equation}\label{e.cavity.gibbs}
 \langle f(\sigma^1, \ldots, \sigma^n) \rangle' := \langle f\rangle'   :=\frac{\int_{\Sigma_N^n}f(\sigma^1,\ldots,\sigma^n)\prod_{\ell\leq n}\exp \mathcal{H}_N^{t,\mu}(\sigma^\ell,\lambda)\ud P_N^*(\sigma^\ell)}{\big(\int_{\Sigma_N} \exp \mathcal{H}_N^{t,\mu}(\sigma,\lambda)\ud P_N^*(\sigma)\big)^n}.
\end{equation}
This is the modified Gibbs measure relative to which the cavity fields will be averaged. It is readily verified that this Gibbs measure is a conditional expectation, and therefore satisfies the Nishimori identity \eqref{e.SBM.Nishimori}.

The first cavity field will have a time-dependent and a measure-dependent component. For each $t\geq 0$, we introduce a random variable $\Pi_t'\sim \Poi(tN)$ as well as an independent family of i.i.d.\@ random vectors $\smash{({G'}^k)_{k\geq 1}}$ each having conditionally independent entries $\smash{({G'}^k_i)_{i\leq N}}$ taking values in $\{0,1\}$ with conditional distribution
\begin{equation}\label{e.cavity.G'.t.dist}
\P\big\{{G'}^k_i=1 \mid \epsilon^*, \sigma^*\big\}:=\frac{c+\Delta \epsilon^*\sigma_i^*}{N+1}.
\end{equation}
Given a collection of random indices $(\ell_k)_{k\geq 1}$ sampled uniformly at random from $\{1,\ldots,N\}$, independently of the other random variables, we define the time-dependent cavity field $z_N^t(\epsilon,\sigma)$ on $\Sigma_{N+1}$ by
\begin{equation}\label{e.cavity.z.t}
z_N^t(\epsilon,\sigma):=\sum_{k\leq \Pi_t'}\log\bigg[\big(c+\Delta\epsilon \sigma_{\ell_k}\big)^{{G'}^k_{\ell_k}}\Big(1-\frac{c+\Delta \epsilon \sigma_{\ell_k}}{N+1}\Big)^{1-{G'}^k_{\ell_k}}\bigg].
\end{equation}
Similarly, for each $\mu\in \M_+$, we define the measure-dependent cavity field $z_N^\mu(\epsilon)$ on $\Sigma_{1}$ by
\begin{equation}\label{e.cavity.z.mu}
z_N^\mu(\epsilon):=\sum_{x\in \Lambda_i(N\mu)}\log \bigg[\big(c+\Delta \epsilon x\big)^{{G'}^x}\Big(1-\frac{c+\Delta \epsilon x}{N+1}\Big)^{1-{G'}^x}\bigg],
\end{equation}
where $G'^x$ are conditionally independent random variables taking values in $\{0,1\}$ with conditional distribution
\begin{equation}
\P\big\{G'^x=1\mid \epsilon^*,x\big\}:=\frac{c+\Delta \epsilon^* x}{N+1}.
\end{equation}
For each pair $(t,\mu)\in \Rp\times \M_+$, the cavity field $z_N^{t,\mu}(\epsilon,\sigma)$ is the sum of the time-dependent and measure-dependent cavity fields,
\begin{equation}\label{e.cavity.z.field}
z_N^{t,\mu}(\epsilon,\sigma):=z_N^t(\epsilon,\sigma)+z_N^\mu(\epsilon).
\end{equation}
The second cavity field will only have a time-dependent component. For each $t\geq 0$, introduce an independent family of i.i.d.\@ random matrices $\smash{({G''}^k)_{k\geq 1}}$ each having conditionally independent entries $\smash{({G''}^k_{i,j})_{i,j\leq N}}$ taking values in $\{0,1\}$ with conditional distribution
\begin{equation}
\P\big\{{G''}^k_{i,j}=1 \mid \sigma^*\big\}:=\frac{c+\Delta \sigma^*_i \sigma^*_j}{N(N+1)}.
\end{equation}
Given a collection of random indices $\smash{(\ell_k,m_k)_{k\geq 1}}$ sampled uniformly at random from $\smash{\{1,\ldots,N\}^2}$, independently of the other random variables, we define the cavity field $y_N(\sigma)$ on $\Sigma_N$ by
\begin{equation}\label{e.cavity.y.field}
y_N^t(\sigma):=\sum_{k\leq \Pi_t}\log\bigg[\big(c+\Delta \sigma_{\ell_k}\sigma_{m_k}\big)^{{G''}^k_{\ell_k,m_k}}\Big(1-\frac{c+\Delta \sigma_{\ell_k}\sigma_{m_k}}{N(N+1)}\Big)^{1-{G''}^k_{\ell_k,m_k}}\bigg].
\end{equation}
The cavity computation reveals that the re-scaled difference between $\bar F_{N+1}(t,\mu,\lambda)$ and $\bar F_N(t,\mu,\lambda)$ is the difference between the cavity fields \eqref{e.cavity.z.field} and \eqref{e.cavity.y.field} averaged with respect to the Gibbs measure \eqref{e.cavity.gibbs}. More precisely, for each pair $(t,\mu)\in \Rp\times \M_+$, the asymptotic difference between these re-scaled free energies is given by
\begin{equation}\label{e.cavity.AN}
A_N(t,\mu,\lambda):=\E\log \int_{\Sigma_1}\langle \exp z_N^{t,\mu}(\epsilon,\sigma)\rangle' \ud P^*(\epsilon)-\E\log\big\langle \exp y_N^t(\sigma)\big\rangle'.
\end{equation}
To prove this, it will be convenient to introduce notation for approximate distributional identities. Given two Hamiltonians $H_N^1$ and $H_N^2$ on $\Sigma_N$, we write
\begin{equation}\label{e.cavity.approximate.distribution}
H_N^1(\sigma)\stackrel{d}{\approx} H_N^2(\sigma) \iff N \bar F_N^1=N\bar F_N^2+o_N(1),
\end{equation}
where $\bar F_N^1$ and $\bar F_N^2$ denote the free energies associated with the Hamiltonians $H_N^1$ and $H_N^2$, respectively. Oftentimes we will not prove rigorously that two Hamiltonians are approximately equal in distribution, contenting ourselves with showing that a given realization of them differs by a quantity that is stochastically of negative degree in $N$. However, all approximate distributional identities that we introduce may be justified rigorously through an interpolation argument similar to that in Lemmas \ref{l.multioverlap.pert.no.effect.FE} or \ref{l.cavity.mod.FE.equiv}.

\begin{lemma}\label{l.cavity.representation}
For all $(t,\mu)\in \Rp\times \M_+$, perturbation parameter $\lambda$, and $N\geq 1$, the re-scaled difference between the perturbed free energy \eqref{e.multioverlap.pert.FE} for a system of size $N+1$ and a system of size $N$ is given by
\begin{equation}\label{e.cavity.eqn}
(N+1)\bar F_{N+1}(t,\mu,\lambda)-N\bar F_N(t,\mu,\lambda)=A_N(t,\mu,\lambda)+o_N(1),
\end{equation}
where the error term is uniform over the perturbation parameter $\lambda=(\lambda_k)_{k\geq 1}$ with $\lambda_k\in [2^{-k-1},2^{-k}]$ for $k\geq 0$.
\end{lemma}

\begin{proof}
For simplicity, we will ignore the perturbation Hamiltonians \eqref{e.multioverlap.gaussian.pert} and \eqref{e.multioverlap.exponential.pert}; however, the choices \eqref{e.multioverlap.epsilon.sequence}-\eqref{e.multioverlap.s.sequence} of the sequences $(\epsilon_N)_{N\geq 1}$ and $(s_N)_{N\geq 1}$ can be used to show that the error incurred by replacing the perturbation Hamiltonian $H_{N+1}^{\mathrm{gauss}}(\sigma,\lambda_0)+H_{N+1}^{\mathrm{exp}}(\sigma,\lambda)$ for the system of size $N+1$ by that for the system of size $N$ can be absorbed into the error term on the right side of \eqref{e.cavity.eqn}. The proof proceeds in two steps. First we relate the time-dependent Hamiltonian \eqref{e.intro.SBM.H.t} for the system of size $N+1$ to that for the system of size $N$, and then we do the same for the measure-dependent Hamiltonian \eqref{e.intro.SBM.H.mu}.
\step{1: time-dependent cavity computation}
The probability that the tuple $\{i_k,j_k\}$ in the time-dependent Hamiltonian \eqref{e.intro.SBM.H.t} contains the index $N+1$ is 
\begin{equation*}
1-\Big(\frac{N}{N+1}\Big)^2=\frac{2N+1}{(N+1)^2}=\frac{2}{N+1}-\frac{1}{(N+1)^2}.
\end{equation*}
It follows by the Poisson coloring theorem that
\begin{equation}\label{e.intro.cavity.1}
H_{N+1}^t(\rho)\stackrel{d}{\approx}\mathcal{H}_N^{t}(\sigma)+z_N^t(\epsilon,\sigma)
\end{equation}
for the modified time-dependent Hamiltonian \eqref{e.cavity.H.t} and the time-dependent cavity field \eqref{e.cavity.z.t}. We have implicitly used that
\begin{equation*}
t\binom{N+1}{2}\Big(1-\frac{2}{N+1}\Big)=t\binom{N}{2} \quad \text{and} \quad t\binom{N+1}{2}\cdot \frac{2}{N+1}=tN.
\end{equation*}
Another application of the Poisson coloring theorem reveals that
\begin{equation}\label{e.intro.cavity.2}
H_N^t(\sigma)\stackrel{d}{\approx}\mathcal{H}_N^t(\sigma)+y_N^t(\sigma),
\end{equation}
for the modified time-dependent Hamiltonian \eqref{e.cavity.H.t} and the time-dependent cavity field \eqref{e.cavity.y.field}. We have implicitly used that
\begin{equation*}
\frac{1}{N}-\frac{1}{N+1}=\frac{1}{N(N+1)}.
\end{equation*}
\step{2: measure-dependent cavity computation}
The measure-dependent Hamiltonian \eqref{e.intro.SBM.H.mu} admits the approximate distributional decomposition
\begin{equation}\label{e.intro.cavity.3}
H_{N+1}^{\mu}(\rho)\stackrel{d}{\approx}\mathcal{H}_{N}^{\mu}(\sigma) + z_N^{\mu}(\epsilon,\sigma)
\end{equation}
for the modified measure-dependent Hamiltonian \eqref{e.cavity.H.mu} and the measure-dependent cavity field \eqref{e.cavity.z.mu}. Similarly, the Poisson coloring theorem implies that
\begin{equation}\label{e.intro.cavity.4}
H_{N}^{\mu}(\sigma)\stackrel{d}{\approx}\mathcal{H}_{N}^{\mu}(\sigma).
\end{equation}
Combining the cavity computations \eqref{e.intro.cavity.1}-\eqref{e.intro.cavity.4} with the definitions of the modified enriched Hamiltonian \eqref{e.cavity.en.H}, the Gibbs average \eqref{e.cavity.gibbs}, and the cavity fields \eqref{e.cavity.z.field} and \eqref{e.cavity.y.field} completes the proof.
\end{proof}

The cavity representation in Lemma \ref{l.cavity.representation} suggests that to understand the limit of the free energy, it is helpful to understand the asymptotic behavior of the quantity $A_N$ in \eqref{e.cavity.AN}. Indeed, if both limits existed, they would coincide. Unfortunately, it will not be possible to show that the limit of $A_N$ exists in general; however, we will now argue that any subsequential limit of $A_N$ along which the multioverlaps \eqref{e.multioverlaps.def} concentrate in the sense of \eqref{e.multioverlap.conc} is given by the functional \eqref{e.intro.HJ.functional} evaluated at a measure $\nu \in \M_p$ which generates the asymptotic spin array according to \eqref{gen.spin.array.1}-\eqref{gen.spin.array.2}.

\begin{lemma}\label{l.cavity.AN.limit}
Fix $(t,\mu)\in \Rp\times \M_+$, and suppose there is a sequence $(N_k,\lambda^{N_k})_{k\geq 1}$ such that $(N_k)_{k\geq 1}$ increases to infinity and the multioverlap concentration \eqref{e.multioverlap.conc} for the Hamiltonian \eqref{e.cavity.pert.mod.H} with perturbation parameters $(\lambda^{N_k})_{k\geq 1}$ holds for every $n\geq 1$. Then there are a subsequence $(N_k',\lambda^{N_k'})_{k\geq 1}$ and a measure $\nu\in \M_p$ such that the asymptotic spin array associated with the perturbation parameters $(\lambda^{N'_k})_{k\geq 1}$ is generated by $\nu$ according to \eqref{gen.spin.array.1}-\eqref{gen.spin.array.2} and 
\begin{equation}\label{e.cavity.AN.limit}
\lim_{k\to +\infty}A_{N_k'}\big(t,\mu,\lambda^{N_k'}\big)=\Par_{t,\mu}(\nu).
\end{equation}
\end{lemma}

\begin{proof}
Using the Prokhorov theorem, we can find a subsequence $(N_k')_{k\geq 1}$ along which the distribution of the spin array $(\sigma_i^\ell)_{i,\ell\geq 1}$ under the Gibbs measure $\E \langle \cdot \rangle'$ converges in the sense of finite-dimensional distributions. For simplicity of notation, we will denote this subsequence $N_k'$ simply by~$N$. By Proposition \ref{p.multioverlap.asymptotic.spin}, there exists a probability measure $\nu \in \M_p$ such that the asymptotic spin array is generated by $\nu$ according to \eqref{gen.spin.array.1}-\eqref{gen.spin.array.2}. The rest of the proof is devoted to establishing \eqref{e.cavity.AN.limit} and proceeds in two steps. Each step determines the asymptotic behavior of one of the quantities
\begin{equation*}
A_N^1(t,\mu,\lambda^{N}):=\E\log\int_{\Sigma_1}\langle \exp z_N^{t,\mu}(\epsilon,\sigma)\rangle'\ud P^*(\epsilon) \quad \text{and} \quad A_N^2(t,\mu,\lambda^{N}):=\E\log \langle \exp y_N^t(\sigma)\rangle'
\end{equation*}
which make up $A_N(t,\mu,\lambda^{N})$. The main difficulty will be that there are two limits to account for: the convergence of the spin array $(\sigma_i^\ell)_{i,\ell\geq 1}$ and the convergence of the Poisson sums. Nonetheless, it will be possible to show that
\begin{equation}\label{e.cavity.AN.limit.goal}
\lim_{N\to+\infty}A_N^1(t,\mu,\lambda^{N})=\psi(\mu+t\nu) \quad \text{and} \quad \lim_{N\to+\infty}A_N^2(t,\mu,\lambda^{N})=\frac{t}{2}\int_{-1}^1 G_\nu(y)\ud \nu(y)
\end{equation}
for the initial condition $\psi:\M_+\to \R$ in \eqref{e.SBM.IC} and the function $G_\nu:[-1,1]\to \R$ in \eqref{e.intro.SBM.G.mu}. To alleviate notation, the dependence of $A_N^1$, $A_N^2$, $z_N$ and $y_N$ on $t$, $\mu$ and $\lambda^N$ will be kept implicit, and we will write $s:=\mu[-1,1]$.
\step{1: limit of $A_N^1$}
The modified Hamiltonian \eqref{e.cavity.pert.mod.H} corresponds to a problem in statistical inference, so the Nishimori identity \eqref{e.SBM.Nishimori} implies that
\begin{equation*}
A_N^1=\E\log\int_{\Sigma_1}\langle \exp z_N(\epsilon,\sigma^1)\rangle'\ud P^*(\epsilon)
\end{equation*}
for the cavity field
\begin{equation*}
z_N(\epsilon,\sigma^1):=\sum_{k\leq \Pi_t'}\log\bigg[\big(c+\Delta\epsilon \sigma_{\ell_k}^1\big)^{{G'}^k_{\ell_k}}\Big(1-\frac{c+\Delta \epsilon \sigma_{\ell_k}^1}{N+1}\Big)^{1-{G'}^k_{\ell_k}}\bigg]+z_N^\mu(\epsilon),
\end{equation*}
where, through a slight abuse of notation, the family of i.i.d.\@ random matrices $\smash{({G'}^k)_{k\geq 1}}$ each has conditionally independent entries $\smash{({G'}^k_i)_{i\leq N}}$ taking values in $\{0,1\}$ with conditional distribution
\begin{equation*}
\P\big\{{G'}^k_i=1 \mid \epsilon^*, \sigma^2\big\}:=\frac{c+\Delta \epsilon^*\sigma_i^2}{N+1}.
\end{equation*}
The Poisson coloring theorem implies that the number of terms in the first sum defining $z_N(\epsilon,\sigma^1)$ that have $\smash{{G'}^k_{\ell_k}}$ equal to zero is Poisson with mean
\begin{equation*}
tN\cdot \bigg(1-\frac{c+\Delta \epsilon^*\sigma_{\ell_k}^2}{N+1}\bigg)=tN +\BigO(1).
\end{equation*}
Together with a Taylor expansion of the logarithm and the law of large numbers, this implies that the contribution of the terms in the first term defining $z_N(\epsilon,\sigma^1)$ that have $\smash{{G'}^k_{\ell_k}}$ equal to zero is
\begin{equation*}
-\sum_{k\leq tN} \frac{c+\Delta \epsilon \sigma^1_{\ell_k}}{N}+\BigO\big(N^{-1}\big)=-ct-\Delta t\epsilon \m +\BigO\big(N^{-1}\big).
\end{equation*}
A similar argument shows that the contribution of the terms in $z_N^\mu(\epsilon)$ that have $\smash{{G'}^x}$ equal to zero is $-cs - \Delta s \epsilon \E x_1+\BigO(N^{-1})$. If we define
\begin{equation*}
z_N'(\epsilon,\sigma^1):=\sum_{k\leq \Pi_t'}\log\big(c+\Delta\epsilon \sigma_{\ell_k}^1\big)^{{G'}^k_{\ell_k}}+\sum_{k\leq \Pi'_{s}}\log\big(c+\Delta\epsilon x_k\big)^{{G'}_{k}^x},
\end{equation*}
where $(x_k)_{k\geq 1}$ are i.i.d.\@ random variables with law $\mu$ and, through a slight abuse of notation, the random variables $\smash{( {G'}^x_k)_{k\geq 1}}$ are conditionally independent with conditional distribution
\begin{equation*}
\P\big\{{G'}_k^x=1\mid \epsilon^*,x_k\big\}:=\frac{c+\Delta \epsilon^* x_k}{N+1},
\end{equation*}
then this implies that
\begin{equation}\label{e.cavity.AN.1}
A_N^1=\E\log\int_{\Sigma_1}\exp(-\Delta\epsilon(t\m+s\E x_1))\langle \exp z_N'(\epsilon,\sigma^1)\rangle'\ud P^*(\epsilon)-c(t+s)+\BigO\big(N^{-1}\big).
\end{equation}
We now fix $\delta>0$ and focus on the first term in this expression which we denote by $A_N^{1,1}$. Averaging with respect to the randomness of $\smash{({G'}^k_{\ell_k})_{k\geq 1}}$ and $\smash{({G'}^x_k)_{k\geq 1}}$, and using symmetry between sites reveals that $\smash{A_N^{1,1}}$ is given by averaging the sum over $r\leq \Pi_t'$ and $q\leq \Pi_s'$ of 
\begin{align*}
\binom{\Pi_t'}{r}\binom{\Pi_s'}{q}\prod_{i=1}^r &\frac{c+\Delta \epsilon^*\sigma_i^2}{N+1}\prod_{i=r+1}^{\Pi_t'}\bigg(1-\frac{c+\Delta \epsilon^*\sigma_i^2}{N+1}\bigg) \prod_{j=1}^q \frac{c+\Delta \epsilon^*x_j}{N+1}\prod_{j=q+1}^{\Pi_s'}\bigg(1-\frac{c+\Delta \epsilon^*x_j}{N+1}\bigg)\\
&\qquad\log\bigg\langle \int_{\Sigma_1} \exp(-\Delta \epsilon(t\m+s\E x_1)) \prod_{i=1}^r (c+\Delta \epsilon \sigma^1_i)\prod_{j=1}^q (c+\Delta \epsilon x_j)\ud P^*(\epsilon)\bigg\rangle'
\end{align*}
with respect to the randomness of $\sigma^2$, $(x_k)_{k\geq 1}$, $\Pi_t'$, $\Pi_s'$, and $\epsilon^*$. Since the expression in the second line of this display grows at most linearly in $r$ and $q$, the concentration of $\Pi_t'$ and $\Pi_s'$ about $tN$ and $sN$, respectively, and the rapid decay of the binomial coefficients allow us to truncate both sums at a large enough integer $M$ by incurring an error of at most $\delta$ uniformly over $N$. Together with the approximations
\begin{align*}
\prod_{i=r+1}^{\Pi_t'}\bigg(1-\frac{c+\Delta \epsilon^* \sigma_i^1}{N+1}\bigg)&=\exp(-tc-\Delta t\epsilon^* \m) +o_N(1),\\
\prod_{j=q+1}^{\Pi_s'}\bigg(1-\frac{c+\Delta \epsilon^* x_j}{N+1}\bigg)&=\exp(-sc-\Delta s\epsilon^* \E x_1)+o_N(1),
\end{align*}
this implies that, up to an error vanishing in $N$ and uniform in $\delta$, the quantity $A_N^{1,1}$ is given by
\begin{align*}
\E\sum_{r\leq \Pi_t'\wedge M} \sum_{q\leq \Pi_s'\wedge M}\binom{\Pi_t'}{r}&\binom{\Pi_s'}{q}\frac{e^{-c(t+s)-\Delta \epsilon^*(t\m+s\E x_1)}}{(N+1)^{r+q}}\bigg\langle\prod_{i=1}^r (c+\Delta \epsilon^*\sigma_i^2)\prod_{j=1}^q (c+\Delta \epsilon^*x_j)\\
& \log\bigg\langle \int_{\Sigma_1} e^{-\Delta \epsilon(t\m+s\E x_1)} \prod_{i=1}^r (c+\Delta \epsilon \sigma^1_i)\prod_{j=1}^q (c+\Delta \epsilon x_j)\ud P^*(\epsilon)\bigg\rangle'\bigg\rangle'.
\end{align*}
Up to incurring a further error vanishing in $N$ and uniform in $\delta$, the Weierstrass approximation theorem can be used to replace the logarithm by a polynomial. This results in a linear combination of finitely many terms that are functions of finitely many spins at a time. Invoking Lemma \ref{l.multioverlap.asymptotic.average} to replace each of the averaged spins in this expression by a sample from the measure $\nu$, tracing back the manipulations we have done to $A_N^{1,1}$, and letting $\delta$ tend to zero reveals that
\begin{equation*}
A_N^{1,1}=\E\log\int_{\Sigma_1}\exp(-\Delta \epsilon(t\m+s\E x_1)) \exp z_N'(\epsilon)\ud P^*(\epsilon)+o_N(1),
\end{equation*}
for the cavity field
\begin{equation}\label{e.cavity.simplified.z.field}
z_N'(\epsilon):=\sum_{k\leq \Pi_t'}\log(c+\Delta \epsilon y_k)^{{G'}^y_k}+\sum_{k\leq \Pi_s'}\log(c+\Delta \epsilon x_k)^{{G'}^x_k},
\end{equation}
where $(y_k)_{k\geq 1}$ are independent samples from the measure $\nu\in \M_p$ which generates the asymptotic spin array according to \eqref{gen.spin.array.1}-\eqref{gen.spin.array.2}. In other words, it is possible to take the limit of the spin array while leaving the finite-volume Poisson sums unaffected. To simplify this expression further, suppose that the probability measures $\bar \mu$ and $\nu$ are discrete,
\begin{equation*}
\bar \mu:=\sum_{\ell\leq K} p_\ell \delta_{a_\ell} \quad \text{and} \quad \nu:=\sum_{\ell\leq K}q_\ell\delta_{a_\ell}
\end{equation*}
for some integer $K\geq1$, some weights $p_\ell,q_\ell \in [0,1]$, and some masses $a_\ell \in [-1,1]$. For each $\ell\leq K$, introduce the index sets
\begin{equation*}
\mathcal{I}_t^\ell:=\{k\leq \Pi_t'\mid {G'}_k^y=1 \text{ and } y_k=a_\ell\} \quad \text{and} \quad \mathcal{I}_s^\ell:=\{k\leq \Pi_s'\mid {G'}_k^x=1 \text{ and } x_k=a_\ell\}
\end{equation*}
in such a way that
\begin{align*}
A_N^{1,1}&=\E\log \int_{\Sigma_1}\exp(-\Delta \epsilon(t\m+s\E x_1)) \prod_{k\leq \Pi_t'}(c+\Delta \epsilon y_k)^{{G'}_k^y}\prod_{k\leq \Pi_s'} (c+\Delta \epsilon x_k)^{{G'}_k^x}\ud P^*(\epsilon)+o_N(1)\\
&=\E\log \int_{\Sigma_1}\exp(-\Delta \epsilon(t\m+s\E x_1)) \prod_{\ell\leq K}\prod_{k \in I_t^\ell \cup I_s^\ell} (c+\Delta \epsilon a_\ell)\ud P^*(\epsilon)+o_N(1).
\end{align*}
The Poisson coloring theorem implies that $\lvert I_t^\ell\cup I_s^\ell\rvert$ is Poisson with mean
\begin{align*}
\E\Pi_t' \P\{{G'}_1^y=1, y_1=a_\ell\}+\E\Pi_s' \P\{{G'}_1^x=1, x_1=a_\ell\}&=sp_\ell(c+\Delta \epsilon^* a_\ell)+tq_\ell(c+\Delta \epsilon^* a_\ell)+o_N(1)\\
&=(c+\Delta \epsilon^* a_\ell)(\mu+t\nu)(a_\ell)+o_N(1).
\end{align*}
It follows that
\begin{align*}
A_N^{1,1}&=\E\log \int_{\Sigma_1}\exp(-\Delta \epsilon(t\m+s\E x_1)) \prod_{x\in \Pi_*(\mu+t\nu)}(c+\Delta \epsilon x)\ud P^*(\epsilon)+o_N(1)\\
&=\psi(\mu+t\nu)+c(t+s)+o_N(1),
\end{align*}
where $\Pi_*(\mu+t\nu)$ is a Poisson point process with mean measure $(c+\Delta \epsilon^*x)\ud (\mu+t\nu)(x)$. Extending this result to the case when $\mu$ and $\nu$ are arbitrary can be done through a continuity argument leveraging Proposition \ref{p.regularity.Lipschitz}. Combining this with \eqref{e.cavity.AN.1} establishes the first equality in \eqref{e.cavity.AN.limit.goal}.
\step{2: limit of $A_N^2$}
The modified Hamiltonian \eqref{e.cavity.pert.mod.H} corresponds to a problem in statistical inference, so the Nishimori identity \eqref{e.SBM.Nishimori} implies that
\begin{equation*}
A_N^2=\E\log\langle \exp y_N(\sigma^1)\rangle'
\end{equation*}
for the cavity field
\begin{align*}
y_N(\sigma^1):=\sum_{k\leq \Pi_t}\log\bigg[\big(c+\Delta \sigma^1_{\ell_k}\sigma^1_{m_k}\big)^{{G''}^k_{\ell_k,m_k}}\Big(1-\frac{c+\Delta \sigma^1_{\ell_k}\sigma^1_{m_k}}{N(N+1)}\Big)^{1-{G''}^k_{\ell_k,m_k}}\bigg],
\end{align*}
where, through a slight abuse of notation, the family of i.i.d.\@ random matrices $\smash{({G''}^k)_{k\geq 1}}$ each has conditionally independent entries $\smash{({G''}^k_{i,j})_{i,j \leq N}}$ taking values in $\{0,1\}$ with conditional distribution
\begin{equation*}
\P\big\{{G''}^k_{i,j}=1 \mid \sigma^2\big\}:=\frac{c+ \sigma_{\ell_k}^2\sigma^2_{m_k}}{N(N+1)}.
\end{equation*}
The Poisson coloring theorem implies that the number of terms in the sum defining $y_N(\sigma^1)$ that have $\smash{{G''}^k_{\ell_k, m_k}}$ equal to zero is Poisson with mean
\begin{equation*}
t\binom{N}{2}\cdot \bigg(1-\frac{c+\Delta \sigma_{\ell_k}^2\sigma_{m_k}^2}{N(N+1)}\bigg)=\frac{t}{2}N^2 +\BigO(1).
\end{equation*}
Together with a Taylor expansion of the logarithm and the law of large numbers, this implies that the contribution of the terms in the sum defining $y_N(\sigma^1)$ that have $\smash{{G'}^k_{\ell_k}}$ equal to zero is
\begin{equation*}
-\sum_{k\leq tN^2/2} \frac{c+\Delta \sigma^1_{\ell_k}\sigma^1_{m_k}}{N(N+1)}+\BigO\big(N^{-1}\big)=-\frac{t}{2}\big(c+\Delta \m^2\big) +\BigO\big(N^{-1}\big).
\end{equation*}
If we define
\begin{equation*}
y_N'(\sigma^1):=\sum_{k\leq \Pi_t}\log(c+\Delta \sigma^1_{\ell_k}\sigma^1_{m_k})^{{G''}^k_{\ell_k,m_k}},
\end{equation*}
then this implies that
\begin{equation}\label{e.cavity.AN.2}
A_N^2=\E\log \langle \exp y_N'(\sigma^1)\rangle'-\frac{t}{2}\big(c-\Delta \m^2\big) +\BigO\big(N^{-1}\big).
\end{equation}
We now focus on the first term in this expression which we denote by $A_N^{2,1}$. A similar argument as that in Step 1 shows that it is possible to take the limit of the spin array while leaving the finite-volume Poisson sum unaffected. More precisely, it is possible to show that
\begin{equation*}
A_N^{2,1}=\E \sum_{k\leq \Pi_t}\log(c+\Delta y_{1,k} y_{2,k})^{{G''}^y_k}+o_N(1),
\end{equation*}
where $(y_{1,k})_{k\geq 1}$ and $(y_{2,k})_{k\geq 1}$ are independent samples from the measure $\nu\in \M_p$ which generates the asymptotic spin array according to \eqref{gen.spin.array.1}-\eqref{gen.spin.array.2}, and $({G''}^y_k)_{k\geq 1}$ has conditional distribution
\begin{equation*}
\P\big\{{G''}^y_k=1 \mid y_{1,k},y_{2,k}\big\}:=\frac{c+\Delta y_{1,k}y_{2,k}}{N(N+1)}.
\end{equation*}
Averaging with respect to the randomness of $({G''}^y_k)_{k\geq 1}$, and using the independence of $(y_{1,k})_{k\geq 1}$ and $(y_{2,k})_{k\geq 1}$ reveals that
\begin{align*}
A_N^{2,1}&=\frac{1}{N(N+1)}\E \sum_{k\leq \Pi_t} (c+\Delta y_{1,k}y_{2,k})\log(c+\Delta y_{1,k}y_{2,k})+o_N(1)\\
&=\frac{t}{2}\E(c+\Delta y_1y_2)\log(c+\Delta y_1 y_2)+o_N(1).
\end{align*}
Combining this with \eqref{e.cavity.AN.2} and recalling the definition of the function $G_\nu:[-1,1]\to \R$ in \eqref{e.intro.SBM.G.mu} gives
\begin{equation*}
\lim_{N\to +\infty}A_N^2=\frac{t}{2}\Big(\E(c+\Delta y_1y_2)\log(c+\Delta y_1 y_2)-\Delta \m^2-c\Big)=\frac{t}{2}\int_{-1}^1 G_\nu(y)\ud \nu(y).
\end{equation*}
This completes the proof.
\end{proof}

This result can now be combined with the cavity representation in Lemma \ref{l.cavity.representation} to bound the asymptotic free energy from above and from below by the functional \eqref{e.intro.HJ.functional} evaluated at two possibly different measures $\nu^\pm \in \M_p$ which generate the asymptotic spin array associated with two sets of perturbation parameters according to \eqref{gen.spin.array.1}-\eqref{gen.spin.array.2}. This result will be improved in the next section.

\begin{proposition}\label{p.cavity.limit.bounds}
For all $(t,\mu)\in \Rp\times \M_+$ there are sequences $\smash{\big(N_k^\pm,\lambda^{N_k^\pm}\big)_{k\geq 1}}$ and measures $\nu^\pm\in \M_p$ such that the following holds. The sequences $\smash{(N_k^\pm)_{k\geq 1}}$ increase to infinity, the multioverlap concentration \eqref{e.multioverlap.conc} for the Hamiltonian \eqref{e.cavity.pert.mod.H} with perturbation parameters $\smash{(\lambda^{N_k^\pm})_{k\geq 1}}$ holds for $n\geq 1$, the asymptotic spin array associated with the perturbation parameters $\smash{(\lambda^{N_k^\pm})_{k\geq 1}}$ is generated by $\nu^\pm$ according to \eqref{gen.spin.array.1}-\eqref{gen.spin.array.2}, and
\begin{equation}\label{e.cavity.limit.bounds}
\Par_{t,\mu}(\nu^-)\leq \liminf_{N\to+\infty}\bar F_N(t,\mu)\leq \limsup_{N\to +\infty}\bar F_N(t,\mu)\leq \Par_{t,\mu}(\nu^+).
\end{equation}
\end{proposition}

\begin{proof}
Combining Lemma \ref{l.multioverlap.average.conc} with the arguments in \cite[Lemma 3.3]{PanSKB}, it is possible to find a sequence of perturbation parameters $(\lambda^N)_{N\geq 1}$ such that the multioverlap concentration \eqref{e.multioverlap.conc} for the perturbed Hamiltonian \eqref{e.cavity.pert.mod.H} holds for all $n\geq 1$. Remembering that the perturbation parameters do not affect the asymptotic behavior of the free energy by Lemma \ref{l.multioverlap.pert.no.effect.FE} and that the $\liminf$ of a Cesàro sum is bounded from below by the $\liminf$ of its general term, we see that
\begin{equation*}
\liminf_{N\to +\infty} \bar F_N(t,\mu)\geq \liminf_{N\to +\infty}A_N(t,\mu,\lambda^N).
\end{equation*}
Invoking Lemma \ref{l.cavity.AN.limit} and passing to a further subsequence along which $(A_N(t,\mu,\lambda^N))_{N\geq 1}$ converges to its $\liminf$ if necessary, gives a sequence $\smash{(N_k^-, \lambda^{N_k^-})_{k\geq 1}}$ and a measure $\nu^-\in \M_p$ such that the asymptotic spin array associated with the perturbation parameters $\smash{(\lambda^{N_k^-})_{k\geq 1}}$ is generated by $\nu^-$ according to \eqref{gen.spin.array.1}-\eqref{gen.spin.array.2} and 
\begin{equation*}
\liminf_{N\to+\infty}A_N\big(t,\mu,\lambda^N\big)=\lim_{k\to +\infty}A_{N_k^-}\big(t,\mu,\lambda^{N_k^-}\big)=\Par_{t,\mu}\big(\nu^-\big).
\end{equation*}
Together with the previous display this gives the lower bound in \eqref{e.cavity.limit.bounds}.
The upper bound in this inequality is proved in an identical manner by instead using that the $\limsup$ of a Cesàro sum is bounded from above by the $\limsup$ of its general term. This completes the proof.
\end{proof}

We close this section by showing that the asymptotic behavior of the enriched free energy \eqref{e.intro.SBM.en.FE} and the modified enriched free energy \eqref{e.cavity.mod.FE} are equivalent. This will play its part in the next section when we show that whenever the free energy converges, we can turn \eqref{e.cavity.limit.bounds} into an equality.

\begin{lemma}\label{l.cavity.mod.FE.equiv}
For all $(t,\mu)\in \Rp\times \M_+$, the enriched free energy \eqref{e.intro.SBM.en.FE} and the modified enriched free energy \eqref{e.cavity.mod.FE} are asymptotically equivalent,
\begin{equation}\label{e.cavity.mod.FE.equiv}
\lim_{N\to +\infty}\big\lvert \bar F_N(t,\mu)-\bar F_N'(t,\mu)\big\rvert=0.
\end{equation}
\end{lemma}

\begin{proof}
Recall from \eqref{e.intro.cavity.2} and \eqref{e.intro.cavity.4} that the enriched Hamiltonian \eqref{e.intro.SBM.en.H} and the modified enriched Hamiltonian \eqref{e.cavity.en.H} differ by the cavity field $y_N^t$ in \eqref{e.cavity.y.field} in the approximate sense defined in \eqref{e.cavity.approximate.distribution},
\begin{equation}\label{e.cavity.approximate.distribution.2}
H_N^{t,\mu}(\sigma)\stackrel{d}{\approx}\mathcal{H}_N^{t,\mu}(\sigma)+y_N^t(\sigma).
\end{equation}
We now leverage an interpolation argument to establish the asymptotic equivalence \eqref{e.cavity.mod.FE.equiv}. For each $u\in [0,1]$, we define the interpolating Hamiltonian
\begin{equation*}
\mathcal{H}_N^{t,\mu}(u,\sigma):=\mathcal{H}_N^{t,\mu}(\sigma)+u y_N^t(\sigma),
\end{equation*}
and denote by $\bar F_N(u)$ and $\langle \cdot \rangle_u$ its associated free energy and Gibbs average, respectively. The fundamental theorem of calculus implies that
\begin{equation}\label{e.cavity.mod.FE.equiv.key}
\big\lvert \bar F_N(1)-\bar F_N(0)\big\rvert \leq \sup_{u\in [0,1]}\Big\lvert \frac{\ud}{\ud u}\bar F_N(u)\Big\rvert\leq \frac{1}{N}\sup_{u\in [0,1]}\E\big\langle \abs{y_N^t(\sigma)}\big\rangle_u.
\end{equation}
Averaging with respect to the randomness of $({G''}^k)_{k\geq 1}$ shows that for each $u\in [0,1]$,
\begin{equation*}
\E\big\langle \abs{y_N^t(\sigma)}\big\rangle_u\leq \E \Pi_t\bigg(\frac{(c+\abs{\Delta})(\log(c+\abs{\Delta})+1)}{N(N+1)}+\BigO\big(N^{-4}\big)\bigg)=\BigO(1).
\end{equation*}
Substituting this into \eqref{e.cavity.mod.FE.equiv.key}, recalling that $\bar F_N(1)=\bar F_N(t,\mu)+o_N(1)$ by the approximate distributional identity \eqref{e.cavity.approximate.distribution.2}, and observing that $\bar F_N(0)=\bar F_N'(t,\mu)$  completes the proof.
\end{proof}

\section{Critical point representation}
\label{sec:CP_rep}

In this section we finally prove Theorems \ref{t.intro.main} and \ref{t.intro.main.limit.array}. We begin by showing that whenever the limit free energy exists and is Gateaux differentiable, it may be identified with a measure $\nu \in \M_p$ through the function $G_\nu$ in \eqref{e.intro.SBM.G.mu}. We then show that the map $\eta\mapsto G_\eta$ is injective on $\M_p$, meaning that the identified measure $\nu\in \M_p$ is unique. We also prove that this measure generates some subsequential limit of the spin array associated with the modified and perturbed Hamiltonian~\eqref{e.cavity.pert.mod.H}.
We then argue that under these same assumptions, the measure $\nu$ is a valid choice for both the measure $\nu^+$ and the measure $\nu^-$ in Proposition \ref{p.cavity.limit.bounds}. Through arguments similar to those in Lemma~\ref{l.cavity.AN.limit}, at any point $(t,\mu)\in \Rp\times \M_+$ of differentiability of the limit free energy, we relate the Gateaux derivative of the limit free energy at $(t,\mu)$ to the Gateaux derivative of the asymptotic initial condition \eqref{e.SBM.IC} at $\mu+t\nu$. This leads to the equality $G_\nu=D_\mu\psi(\mu+t\nu,\cdot)$ which we show to be equivalent to the fixed point equation $\nu=\Gamma_{t,\mu}(\nu)$ associated with the operator \eqref{e.intro.SBM.FP.operator}. Finally, we combine all these insights with  Proposition \ref{p.regularity.differentiability} on the almost everywhere Gateaux differentiability of the limit free energy to establish Theorems \ref{t.intro.main} and \ref{t.intro.main.limit.array}.

To relate the Gateaux derivative of the limit free energy to that of the initial condition, it will be convenient to recall from \cite[Remark 3.2]{dominguez2024mutual} that for any $\mu \in \M_+$, the density of the Gateaux derivative $D_\mu \psi(\mu)$ is 
\begin{equation}\label{e.intro.IC.Gateaux.der}
D_\mu \psi(\mu,x) = \E \langle c+\Delta \epsilon x\rangle_* \log\langle c+\Delta \epsilon x\rangle_* -c-\Delta \m x,
\end{equation}
where, conditionally on the ground truth $\sigma^*$, the bracket $\langle \cdot \rangle_{*}$ denotes the Gibbs average associated with the Hamiltonian defined on $\Sigma_1$ by
\begin{equation}
H^\mu(\epsilon):=-\int_{-1}^1 (c+\Delta \epsilon x)\ud \mu(x)+ \sum_{x\in \Pi_*(\mu)} \log(c+\Delta \epsilon x).
\end{equation}
Here $\Pi_*(\mu)$ denotes a Poisson point process on $[-1,1]$ with mean measure $(c+\Delta \sigma^*x)\ud \mu(x)$. This means that for any bounded and measurable function $f=f(\epsilon^1,\ldots,\epsilon^n)$ of finitely many replicas,
\begin{equation}\label{e.intro.Gibbs.star}
\langle f(\epsilon^1,\ldots,\epsilon^n)\rangle_*:=\langle f\rangle_*:=\frac{\int_{\Sigma_1^n}f(\epsilon^1,\ldots,\epsilon^n)\prod_{\ell\leq n}\exp H^\mu(\epsilon^\ell)\ud P^*(\epsilon^\ell)}{\big(\int_{\Sigma_1} \exp H^\mu(\epsilon)\ud P^*(\epsilon)\big)^n}.
\end{equation}
It will also be convenient to remember from the proof of \cite[Lemma 2.4]{dominguez2024mutual} that the Gateaux derivative of the enriched free energy \eqref{e.intro.SBM.en.FE} at $(t,\mu)\in \Rp\times \M_+$ in the direction of a probability measure $\eta \in \Pr[-1,1]$ is given by
\begin{equation}\label{e.SBM.en.FE.der.mu.exact}
D_\mu \overline{F}_N(t,\mu;\eta)=\sum_{i\leq N}\E\log \bigg\langle \big(c+\Delta \sigma_iw_i\big)^{\tG_i^w}\Big(1-\frac{c+\Delta \sigma_iw_i}{N}\Big)^{1-\tG_i^w} \bigg\rangle,
\end{equation}
where the bracket $\langle \cdot \rangle$ denotes the Gibbs average \eqref{e.SBM.Gibbs}, the $(w_i)_{i\geq 1}$ are i.i.d.\@ random variables sampled from $\eta$, and $(\td G_i^w)_{i\leq N}$ are conditionally independent random variables taking values in $\{0,1\}$ with conditional distribution
\begin{equation}
\P\big\{\td G_i^w=1\mid \sigma^*, w_i\big\}:=\frac{c+\Delta \sigma_i^* w_i}{N}.
\end{equation}

\begin{lemma}\label{l.limit.GD.given.by.measure}
Suppose that the sequence of enriched free energies $(\bar F_N)_{N\geq 1}$ converges pointwise to some limit $f:\Rp\times \M_+\to \R$ along a subsequence $(N_k)_{k\geq 1}$, and fix $(t,\mu)\in \Rp\times \mathrm{Reg}(\M_+)$. If $f(t,\cdot)$ is Gateaux differentiable at $\mu \in \mathrm{Reg}(\M_+)$, then there exists $\nu \in \M_p$ with
\begin{equation}
G_\nu=D_\mu f(t,\mu,\cdot).
\end{equation}
\end{lemma}

\begin{proof}
For each integer $k\geq 1$, let $\nu_{N_k}:=\mathcal{L}\big(\langle \sigma_1\rangle\big)$ denote the law (under $\P$) of the Gibbs average of a spin variable, where the Gibbs measure is as in \eqref{e.SBM.Gibbs}. The Nishimori identity \eqref{e.SBM.Nishimori}, the derivative expression \eqref{e.SBM.en.FE.der.mu}, and the definition of the kernel $g$ in \eqref{e.intro.SBM.g} imply that
\begin{equation}\label{e.limit.GD.given.by.measure.key}
D_\mu \bar F_{N_k}(t,\mu,\cdot)=G_{\nu_{N_k}}(\cdot)+\BigO(N_k^{-1}).    
\end{equation}
By the Prokhorov theorem, the sequence $(\nu_{N_k})_{k\geq 1}\subset \M_p$ admits a subsequential limit $\nu \in \M_p$. Invoking Proposition \ref{p.regularity.cont.of.derivative} to let $k$ tend to infinity in \eqref{e.limit.GD.given.by.measure.key} completes the proof.
\end{proof}

\begin{lemma}\label{l.G.injective}
If $\nu,\nu'\in \M_p$ are such that $G_\nu=G_{\nu'}$, then $\nu=\nu'$.
\end{lemma}

\begin{proof}
The equality $G_\nu=G_{\nu'}$ means that for all $x\in [-1,1]$,
\begin{align*}
c\log(c)-c+\Delta \m x+&c\sum_{n\geq 2}\frac{(-\Delta/c)^n}{n(n-1)}\int_{-1}^1 y^n \ud \nu(y)x^n\\
&=c\log(c)-c+\Delta \m x+c\sum_{n\geq 2}\frac{(-\Delta/c)^n}{n(n-1)}\int_{-1}^1 y^n \ud \nu'(y)x^n.
\end{align*}
It follows that $\nu$ and $\nu'$ have all the same moments of order $n\geq 2$, and are therefore equal by the Stone-Weierstrass theorem (see e.g.\ \cite[Theorem~A.10]{TD_JC_book}). 
\end{proof}

\begin{lemma}\label{l.limit.GD.measure.is.generator}
Suppose that the sequence of enriched free energies $(\bar F_N)_{N\geq 1}$ converges pointwise to some limit $f:\Rp\times \M_+\to \R$ along a subsequence $(N_k)_{k\geq 1}$. If $(t,\mu)\in \Rp\times \mathrm{Reg}(\M_+)$ is such that $f(t,\cdot)$ is Gateaux differentiable at $\mu$ and $\nu \in \M_p$ is such that $G_\nu=D_\mu f(t,\mu,\cdot)$ (which exists by Lemma \ref{l.limit.GD.given.by.measure}), then there is a sequence $\smash{\big(N_k',\lambda^{N_k'}\big)_{k\geq 1}}$ such that the following holds. The sequence $\smash{(N_k')_{k\geq 1}}\subset (N_k)_{k\geq 1}$ increases to infinity, the multioverlap concentration \eqref{e.multioverlap.conc} for the Hamiltonian \eqref{e.cavity.pert.mod.H} with perturbation parameters $\smash{(\lambda^{N_k'})_{k\geq 1}}$ holds for every $n\geq 1$, and the asymptotic spin array associated with the perturbation parameters $\smash{(\lambda^{N_k'})_{k\geq 1}}$ is generated by $\nu$ according to \eqref{gen.spin.array.1}-\eqref{gen.spin.array.2}.
\end{lemma}

\begin{proof}
Combining Lemma \ref{l.multioverlap.average.conc} with the arguments in \cite[Lemma 3.3]{PanSKB}, it is possible to find a sequence of perturbation parameters $(\lambda^N)_{N\geq 1}$ such that the multioverlap concentration \eqref{e.multioverlap.conc} for the perturbed Hamiltonian \eqref{e.cavity.pert.mod.H} holds for all $n\geq 1$. Using the Prokhorov theorem, we can find a subsequence $(N_k')_{k\geq 1}$ of $(N_k)_{k\geq 1}$ along which the distribution of the spin array $(\sigma_i^\ell)_{i,\ell\geq 1}$ under the Gibbs measure $\E \langle \cdot \rangle'_{N_k}$ with perturbation parameter $\smash{\lambda^{N_k}}$ converges in the sense of finite-dimensional distributions. By Proposition \ref{p.multioverlap.asymptotic.spin}, there exists a probability measure $\nu^* \in \M_p$ such that the asymptotic spin array is generated by $\nu^*$ according to \eqref{gen.spin.array.1}-\eqref{gen.spin.array.2}. We now show that $\nu^*=\nu$. An identical computation to that leading to the derivative expression \eqref{e.SBM.en.FE.der.mu} reveals that the Gateaux derivative of the perturbed modified free energy \eqref{e.cavity.mod.pert.FE} is given by
\begin{equation*}
D_\mu \bar F'_{N_k'}(t,\mu, \lambda^{N_k'},x)=\big(c+\Delta \m x\big)\log(c)+c\sum_{n\geq 2}\frac{(-\Delta/c)^n}{n(n-1)}\E\langle R_{1,\ldots,n}\rangle_{N_k'}' x^n-c+o_{N_k'}(1).
\end{equation*}
Remembering Lemmas \ref{l.cavity.mod.FE.equiv} and \ref{l.multioverlap.pert.no.effect.FE}, and arguing as in the proof of Proposition \ref{p.regularity.cont.of.derivative} to let $k$ tend to infinity reveals that $D_\mu f(t,\mu,\cdot)=G_{\nu^*}$. Leveraging the assumption that $G_{\nu}=D_\mu f(t,\mu,\cdot)$ implies that $G_{\nu}=G_{\nu^*}$. Invoking Lemma \ref{l.G.injective} completes the proof.
\end{proof}

\begin{lemma}\label{l.main.assuming.differentiability}
Suppose that the sequence of enriched free energies $(\bar F_N)_{N\geq 1}$ converges pointwise to some limit $f:\Rp\times \M_+\to \R$, and fix $(t,\mu)\in \Rp\times \mathrm{Reg}(\M_+)$. If $f(t,\cdot)$ is Gateaux differentiable at $\mu \in \mathrm{Reg}(\M_+)$ and $\nu \in \M_p$ is such that $G_\nu=D_\mu f(t,\mu,\cdot)$ (which exists by Lemma \ref{l.limit.GD.given.by.measure}), then
\begin{equation}
f(t,\mu)=\Par_{t,\mu}(\nu).
\end{equation}
\end{lemma}

\begin{proof}
Let $(N_k,\lambda^{N_k^\pm})_{k\geq 1}$ and $\nu^\pm\in \M_p$ be as in the proof of Proposition \ref{p.cavity.limit.bounds}. By inequality \eqref{e.cavity.limit.bounds}, it suffices to show that $\nu^-=\nu^+=\nu$. Since $\nu^\pm$ generate the asymptotic spin array associated with the perturbation parameters $\smash{(\lambda^{N_k^\pm})_{k\geq1}}$ according to \eqref{gen.spin.array.1}-\eqref{gen.spin.array.2}, arguing as in the proof of Lemma \ref{l.limit.GD.measure.is.generator} with $\nu^*$ replaced by $\nu^\pm$, it is possible to show that $G_{\nu^\pm}=D_\mu f(t,\mu,\cdot)=G_\nu$. Invoking Lemma \ref{l.G.injective} gives $\nu=\nu^\pm$ and completes the proof.
\end{proof}

\begin{lemma}\label{l.GD.of.limit.FE.as.GD.of.IC}
Suppose that the sequence of enriched free energies $(\bar F_N)_{N\geq 1}$ converges pointwise to some limit $f:\Rp\times \M_+\to \R$ along a subsequence $(N_k)_{k\geq 1}$, and fix $(t,\mu)\in \Rp\times \mathrm{Reg}(\M_+)$. If $f(t,\cdot)$ is Gateaux differentiable at $\mu \in \mathrm{Reg}(\M_+)$ and $\nu \in \M_p$ is such that $G_\nu=D_\mu f(t,\mu,\cdot)$ (which exists by Lemma \ref{l.limit.GD.given.by.measure}), then
\begin{equation}
D_\mu f(t,\mu,\cdot)=D_\mu \psi(\mu+t\nu,\cdot).
\end{equation}
\end{lemma}

\begin{proof}
Let $(N_k',\lambda^{N_k'})_{k\geq 1}$ be as in the statement of Lemma \ref{l.limit.GD.measure.is.generator}. To alleviate the exposition, we abuse notation and denote this sequence by $(N,\lambda^N)_{N\geq 1}$. An identical computation to that leading to the derivative expression \eqref{e.SBM.en.FE.der.mu.exact} reveals that the Gateaux derivative of the perturbed free energy \eqref{e.multioverlap.pert.FE} in the direction of a probability measure $\eta \in \Pr[-1,1]$ is given by
\begin{equation*}
D_\mu \overline{F}_{N+1}(t,\mu,\lambda^N;\eta)=\sum_{i\leq N+1}\E\log \bigg\langle \big(c+\Delta \sigma_iw_i\big)^{\tG_i^w}\Big(1-\frac{c+\Delta \sigma_iw_i}{N+1}\Big)^{1-\tG_i^w} \bigg\rangle_{N+1},
\end{equation*}
where the bracket $\langle \cdot \rangle_{N+1}$ denotes the Gibbs average \eqref{e.multioverlap.Gibbs.pert} with perturbation parameter $\lambda^N$, the $(w_i)_{i\geq 1}$ are i.i.d.\@ random variables sampled from $\eta$, and $(\td G_i^w)_{i\leq N+1}$ are conditionally independent random variables taking values in $\{0,1\}$ with conditional distribution
\begin{equation*}
\P\big\{\td G_i^w=1\mid \sigma^*, w_i\big\}:=\frac{c+\Delta \sigma_i^* w_i}{N+1}.
\end{equation*}
Taylor-expanding the logarithm and using the law of large numbers shows that the contribution of the terms in this sum that have $\td G_i^w$ equal to zero is
\begin{equation*}
-c -\frac{\Delta \E w_1}{N+1}\sum_{i\leq N+1} \sigma_i+\BigO(N^{-1})=-c- \Delta \m \E w_1+\BigO(N^{-1}).
\end{equation*}
Together with the symmetry between sites, this implies that
\begin{equation}\label{e.GD.of.limit.FE.as.GD.of.IC.key}
D_\mu \bar F_{N+1}(t,\mu,\lambda^N;\eta)=(N+1)\E\log \big\langle (c+\Delta \epsilon w_1)^{\tG_1^w}\big\rangle_{N+1}-c-\Delta \m \E w_1.
\end{equation}
The distributional identities \eqref{e.intro.cavity.1} and \eqref{e.intro.cavity.3} reveal that the first term in this expression is given by
\begin{align*}
D^1
&:=\E\log \int_{\Sigma_1}\big\langle (c+\Delta \epsilon w_{1})^{\tG_{1}^w}\exp z_N^{t,\mu}(\epsilon,\sigma)\big\rangle'-\E\log \int_{\Sigma_1}\langle \exp z_N^{t,\mu}(\epsilon,\sigma)\rangle'\ud P^*(\epsilon)
\end{align*}
for the cavity field $z_N^{t,\mu}$ in \eqref{e.cavity.z.field} and the modified Gibbs average \eqref{e.cavity.gibbs}. By Lemma~\ref{l.limit.GD.measure.is.generator}, the choice of sequence $(N,\lambda^N)_{N\geq 1}$ ensures that the asymptotic spin array $(\sigma_i^\ell)_{i,\ell\geq 1}$ is generated by $\nu$ according to \eqref{gen.spin.array.1}-\eqref{gen.spin.array.2}, so arguing as in Step 1 of the proof of Lemma \ref{l.cavity.AN.limit} it is possible to show that
\begin{align*}
D^1=\E\log \int_{\Sigma_1}e^{-c(t+s)-\Delta\epsilon(t\m+s\E x_1)} &(c+\Delta \epsilon w_1)^{{G'}^w}\exp z_N'(\epsilon)\ud P^*(\epsilon)\\
&-\E\log \int_{\Sigma_1}e^{-c(t+s)-\Delta\epsilon(t\m+s\E x_1)} \exp z_N'(\epsilon)\ud P^*(\epsilon)+o_N(1),
\end{align*}
where $s:=\mu[-1,1]$, the cavity field $z_N'$ is defined in \eqref{e.cavity.simplified.z.field}, and ${G'}^w$ has conditional distribution
\begin{equation*}
\P\big\{{G'}^w=1 \mid \epsilon^*, w_1\big\}:=\frac{c+\Delta \epsilon^* w_1}{N+1}.
\end{equation*}
Continuing to argue as in Step 1 of the proof of Lemma \ref{l.cavity.AN.limit} gives
\begin{equation*}
D^1=\E\log \big\langle (c+\Delta \epsilon w_1)^{{G'}^w}\big\rangle_*+o_N(1)=\frac{1}{N+1}\E\langle c+\Delta \epsilon w_1 \rangle_*\log\langle c+\Delta \epsilon w_1\rangle_*+o_N(1).
\end{equation*}
Substituting this into \eqref{e.GD.of.limit.FE.as.GD.of.IC.key} and leveraging Lemma \ref{l.multioverlap.pert.no.effect.FE} and Proposition \ref{p.regularity.cont.of.derivative} to let $N$ tend to infinity reveals that
\begin{equation*}
D_\mu f(t,\mu;\eta)=\E\langle c+\Delta \epsilon w_1 \rangle_*\log\langle c+\Delta \epsilon w_1\rangle_*-c-\Delta \m \E w_1.
\end{equation*}
Remembering the expression \eqref{e.intro.IC.Gateaux.der} for the Gateaux derivative of the initial condition completes the proof.
\end{proof}

\begin{lemma}\label{l.equivalence.of.FP.eqns}
For any $(t,\mu)\in \Rp\times \M_+$, a measure $\nu \in \M_p$ satisfies $G_\nu=D_\mu\psi(\mu+t\nu,\cdot)$ if and only if it is a fixed point of the map $\nu\mapsto \Gamma_{t,\mu}(\nu)$.
\end{lemma}

\begin{proof}
We denote by $y_1$ a random variable sampled from $\nu$. The definition of the function $g$ in \eqref{e.intro.SBM.g} implies that for every $x\in [-1,1]$,
\begin{equation}
G_{\nu}(x)=\int_{-1}^1 g(xy)\ud \nu^*(y)=(c+\Delta \E y_1 x)\log(c)+c\sum_{n\geq 2}\frac{(-\Delta/c)^n}{n(n-1)}\E y_1^nx^n-c.
\end{equation}
On the other hand, the derivative expression \eqref{e.intro.IC.Gateaux.der} implies that
\begin{equation}
D_\mu \psi(\mu+t\nu,x) = \E g(\langle \epsilon\rangle_*x)=(c+\Delta \E\langle \epsilon\rangle_* x)\log(c)+c\sum_{n\geq 2}\frac{(-\Delta/c)^n}{n(n-1)}\E\langle \epsilon\rangle_*^nx^n-c.
\end{equation}
The equation $G_\nu=D_\mu\psi(\mu+t\nu,\cdot)$ is therefore equivalent to the equality between all moments of $y_1$ and $\langle \epsilon\rangle_*$ of order $2$ and higher. By the Stone-Weierstrass theorem, this is in turn equivalent to the equality between the laws of $y_1$ and $\langle \epsilon\rangle_*$. These laws are equal if and only if the fixed point equation $\nu=\Gamma_{t,\mu}(\nu)$ is satisfied. This completes the proof.
\end{proof}

\begin{proof}[Proof of Theorem \ref{t.intro.main}]
Proposition \ref{p.regularity.differentiability} gives a sequence $(\mu_n)_{n\geq 1}\subset \mathrm{Reg}(\M_+)$ converging weakly to $\mu$ such that $f(t,\cdot)$ is Gateaux differentiable at every $\mu_n\in \mathrm{Reg}(\M_+)$. By Lemmas \ref{l.main.assuming.differentiability} and  \ref{l.GD.of.limit.FE.as.GD.of.IC}, for every $n\geq 1$, there is $\nu_n\in \M_p$ with
\begin{equation*}
G_{\nu_n}=D_\mu f(t,\mu_n,\cdot)=D_\mu \psi(\mu+t\nu_n,\cdot) \quad \text{and} \quad f(t,\mu_n)=\Par_{t,\mu_n}(\nu_n).
\end{equation*}
By the Prokhorov theorem, up to passing to a subsequence, we can assume that $(\nu_n)_{n \geq 1}$ converges weakly to some limit $\nu^*\in \M_p$. The continuity of the map $\nu\mapsto G_\nu$ and of the Gateaux derivative map $\nu\mapsto D_\mu\psi(\nu,\cdot)$ in \eqref{e.intro.IC.Gateaux.der} imply that $\nu^*$ satisfies the equation
\begin{equation}\label{e.intro.main.FP.prelim}
 G_{\nu^*}=D_\mu \psi(\mu + t\nu^*,\cdot).
\end{equation}
Similarly, the Lipschitz continuity of the limit free energy and of the asymptotic initial condition established in Proposition \ref{p.regularity.Lipschitz} imply that
\begin{equation*}
f(t,\mu)=\lim_{n\to +\infty}f(t,\mu_n)=\lim_{n\to +\infty}\Par_{t,\mu_n}(\nu_n)=\Par_{t,\mu}(\nu^*).
\end{equation*}
Remembering that equation \eqref{e.intro.main.FP.prelim} is equivalent to the fixed point equation $\nu^*=\Gamma_{t,\mu}(\nu^*)$ by Lemma \ref{l.equivalence.of.FP.eqns} completes the proof.
\end{proof}

\begin{proof}[Proof of Theorem \ref{t.intro.main.limit.array}]
Lemmas \ref{l.GD.of.limit.FE.as.GD.of.IC} and \ref{l.equivalence.of.FP.eqns} imply that $\nu^*$ is a fixed point of the map $\nu\mapsto \Gamma_{t,\mu}(\nu)$, and Lemma \ref{l.main.assuming.differentiability} ensures that $f(t,\mu)=\Par_{t,\mu}(\nu^*)$. The second part of the statement on the asymptotic spin array follows from Lemma \ref{l.limit.GD.measure.is.generator}. This completes the proof.
\end{proof}

\section{The disassortative and small-time settings}
\label{sec:special.cases}

Two limitations of Theorems \ref{t.intro.main} and \ref{t.intro.main.limit.array} are that they assume the existence of the limit free energy, and that they do not provide a simple characterization of the fixed point for which \eqref{e.intro.main} holds. In this section we focus on two settings in which these limitations can be addressed: the disassortative setting when $\Delta\leq 0$ and the small-time setting when $t\in \Rp$ is close to zero.

In the disassortative stochastic block model with two communities, it can be shown that the limit free energy exists and is given by evaluating the functional $\Par_{t,\mu}$ defined in \eqref{e.intro.HJ.functional} at the fixed point of the operator $\Gamma_{t,\mu}$ which maximizes $\Par_{t,\mu}$. This is the content of Proposition~\ref{p.intro.main.disassortative} which we now prove.

\begin{proof}[Proof of Proposition \ref{p.intro.main.disassortative}]
We fix $\nu \in \M_p$, and for each $u\in [0,1]$, we define the interpolating free energy
\begin{equation*}
\p(u):=\bar F_N(tu,\mu+t(1-u)\nu)
\end{equation*}
in such a way that $\p(1)=\bar F_N(t,\mu)$ and $\p(0)=\psi(\mu+t\nu)$. We denote by $x_1$ a random variable sampled from $\nu$. Recall from the derivative computations in \cite[Corollaries 2.2 and 2.5]{dominguez2024mutual} that
\begin{align}
\partial_t \overline{F}_N(t,\mu)&=\frac{1}{2}\big(c+\Delta\m^2\big)\log (c)+\frac{c}{2}\sum_{n\geq 2}\frac{(-\Delta/c)^n}{n(n-1)}\E\big\langle R_{[n]}^2\big\rangle-\frac{c}{2}+\BigO(N^{-1}), \label{e.en.FE.t.der}\\
D_\mu \overline{F}_N(t,\mu,x)&=\big(c+\Delta \m x\big)\log(c)+c\sum_{n\geq 2}\frac{(-\Delta/c)^n}{n(n-1)}\E\langle R_{[n]}\rangle x^n-c+\BigO(N^{-1}).\label{e.en.FE.mu.der}
\end{align}
Together with the assumption that $\Delta\leq 0$ these imply that
\begin{align*}
\p'(u)&=t\partial_t \bar F_N(tu, \mu+t(1-u)\nu)-t\int_{-1}^1 D_\mu\bar F_N(tu,\mu+t(1-u)\nu)\ud \nu\\
&=t\bigg(\frac{c}{2}-\frac{1}{2}\big(c+\Delta \m^2\big)\log(c)+\frac{c}{2}\sum_{n\geq 2}\frac{(-\Delta/c)^n}{n(n-1)}\E\big\langle R_{[n]}^2\big\rangle- c\sum_{n\geq 2}\frac{(-\Delta/c)^n}{n(n-1)}\E\langle R_{[n]}\rangle \E x_1^n\bigg)\\
&=-\frac{t}{2}\bigg(\big(c+\Delta \m^2\big)\log(c)+ c\sum_{n\geq 2}\frac{(-\Delta/c)^n}{n(n-1)}\big(\E x_1^n\big)^2 -c\bigg)+\frac{tc}{2}\sum_{n\geq 2}\frac{(-\Delta/c)^n}{n(n-1)}\E\big\langle (R_{[n]}- \E x_1^n)^2\big\rangle\\
&\geq -\frac{t}{2}\int_{-1}^1 G_\nu(y)\ud \nu(y).
\end{align*}
Leveraging the fundamental theorem of calculus and letting $N$ tend to infinity gives the lower bound
\begin{equation*}
\liminf_{N\to +\infty}\bar F_N(t,\mu)\geq \sup_{\nu\in \M_p}\Par_{t,\mu}(\nu).
\end{equation*}
The matching upper bound is immediate from Proposition~\ref{p.cavity.limit.bounds}.
\end{proof}

In the small-time setting, a Lipschitz continuity bound on the fixed point operator $\Gamma_{t,\mu}$ defined in \eqref{e.intro.SBM.FP.operator} can be combined with the Banach fixed point theorem to show that $\Gamma_{t,\mu}$ admits a unique fixed point $\nu_{t,\mu}\in \M_p$. Together with a Lipschitz continuity bound on the map $(t,\mu)\mapsto \nu_{t,\mu}$ and an argument similar to that in Lemma \ref{l.main.assuming.differentiability}, this uniqueness can be used to establish Proposition~\ref{p.intro.main.small.t}.

\begin{lemma}\label{l.t.small.FP.operator.props}
There exists a constant $C<+\infty$ such that the fixed point operator $\Gamma_{t,\mu}:\M_p\to\M_p$ defined in \eqref{e.intro.SBM.FP.operator} satisfies the following Lipschitz continuity bounds.
\begin{enumerate}
    \item For all $(t,\mu)\in \Rp\times \M_+$, and $\nu,\eta\in \M_p$,
    \begin{equation}\label{e.t.small.FP.operator.contraction}
    W\big(\Gamma_{t,\mu}(\nu),\Gamma_{t,\mu}(\eta)\big)\leq C t W(\nu,\eta).
    \end{equation}
    \item For all $(t,\mu), (t',\mu')\in \Rp\times \M_+$, and $\nu\in \M_p$,
    \begin{equation}\label{e.t.small.FP.operator.Lip}
    W\big(\Gamma_{t,\mu}(\nu), \Gamma_{t',\mu'}(\nu)\big)\leq C\big(\abs{t'-t}+\mu[-1,1]W\big(\bar \mu, \bar \mu'\big)+\abs{\mu[-1,1]-\mu'[-1,1]}\big).
    \end{equation}
\end{enumerate}
\end{lemma}

\begin{proof}
It will be convenient to remember the dual representation \eqref{e.Wasserstein.def.inf} of the Wasserstein distance. With this representation in mind, we prove each bound separately. 
\begin{enumerate}
\item Fix $(t,\mu)\in \Rp\times \M_+$ as well as $\nu,\eta\in \M_p$. 
Conditionally on the ground truth signal $\epsilon^*$ sampled from $P^*$, introduce the measures
\begin{equation}\label{e.t.small.tilde.nu}
\ud \tilde{\nu}(x):=(c+\Delta \epsilon^*x)\ud \nu(x) \quad \text{and} \quad \ud \tilde{\eta}(x):=(c+\Delta \epsilon^*x)\ud \eta(x).
\end{equation}
We denote by $(X_k)_{k\geq 1}$ and $(Y_k)_{k \geq 1}$ i.i.d.\@ samples from the probability measures $\nu'$ and $\eta'$ induced by $\tilde{\nu}$ and $\tilde{\eta}$, and let $\Pi_t$ be a Poisson random variable with mean $t(c+\Delta \epsilon^*\m)$. For each $u\in [0,1]$, we define the interpolating Hamiltonian $H_u$ on $\Sigma_1$ by
\begin{equation*}
H_u(\epsilon):=-\epsilon \m +\sum_{x\in \Pi_{0,\mu}(0)}\log(c+\Delta \epsilon x)+\sum_{k\leq \Pi_t} \log\big(c+\Delta \epsilon (X_k+u(Y_k-X_k))\big)
\end{equation*}
as well as the random variable
\begin{equation*}
Z_u:=\frac{\int_{\Sigma_1} \epsilon \exp H_u(\epsilon)\ud P^*(\epsilon)}{\int_{\Sigma_1} \exp H_u(\epsilon)\ud P^*(\epsilon)}.
\end{equation*}
The superposition principle ensures that $\smash{\Gamma_{t,\mu}(\nu)\stackrel{d}{=}Z_0}$ and $\smash{\Gamma_{t,\mu}(\eta)\stackrel{d}{=}Z_1}$. If we write $\langle \cdot\rangle_u$ for the Gibbs measure associated with the Hamiltonian $H_u$, then
\begin{align*}
\frac{\ud}{\ud u} Z_u 
&= \sum_{k\leq \Pi_t}(Y_k-X_k)\bigg(\bigg\langle \frac{\epsilon}{c+\Delta \epsilon (X_k+u(Y_k-X_k))}\bigg\rangle_u - \langle \epsilon\rangle_u \bigg\langle \frac{1}{c+\Delta \epsilon (X_k+u(Y_k-X_k))}\bigg\rangle_u\bigg).
\end{align*}
It follows by the mean value theorem that for some constant $C<+\infty$ that depends only on $c$ and $\Delta$ whose value may not be the same at each occurrence,
\begin{equation*}
\E\abs{Z_1-Z_0}\leq C\E\sum_{k\leq \Pi_t}\abs{Y_k-X_k}\leq tC\E \abs{X_1-Y_1}.
\end{equation*}
Taking the infimum over all couplings between $(X_k)_{k\geq 1}$ and $(Y_k)_{k\geq 1}$ yields
\begin{equation}\label{e.t.small.unique.FP.bound.1}
W\big(\Gamma_{t,0}(\nu),\Gamma_{t,0}(\eta)\big)=W(Z_0,Z_1)\leq tCW\big(\nu',\eta'\big),
\end{equation}
where we recall that $\nu'$ and $\eta'$ denote the probability measures induced by $\tilde{\nu}$ and $\tilde{\eta}$. We now control the distance between the measures $\nu'$ and $\eta'$ by that between the measures $\nu$ and $\eta$. The explicit representation \eqref{e.Wasserstein.1D} of the one-dimensional Wasserstein distance reveals that
\begin{equation*}
W\big(\nu', \eta'\big)=\int_{-1}^1 \abs{F_{\nu'}(x)-F_{\eta'}(x)}\ud x,
\end{equation*}
and integrating by parts shows that for $x\in [-1,1]$,
\begin{equation*}
F_{\nu'}(x)
=\frac{1}{c+\Delta \epsilon^*\m}\bigg(cF_\nu(x)+\Delta \epsilon^* \Big(xF_\nu(x)-\int_{-1}^x F_\nu(z)\ud z\Big)\bigg).
\end{equation*}
It follows that for $x\in [-1,1]$,
\begin{equation*}
\abs{F_{\nu'}(x)-F_{\eta'}(x)}\ud x\leq C\big(\abs{F_\nu(x)-F_\eta(x)}+W(\nu,\eta)\big),
\end{equation*}
and therefore
\begin{equation*}
W\big(\nu',\eta'\big)\leq CW(\nu,\eta).
\end{equation*}
Together with \eqref{e.t.small.unique.FP.bound.1} this establishes the Lipschitz bound \eqref{e.t.small.FP.operator.contraction}.

\item We fix $(t,\mu), (t',\mu')\in \Rp\times \M_+$ as well as $\nu \in \M_p$. To simplify notation, we denote by $\tilde{\mu}$, $\tilde{\mu}'$ and $\tilde{\nu}$ the measures defined as in \eqref{e.t.small.tilde.nu}, and let $\tilde{s}:=\tilde{\mu}[-1,1]$ and $\tilde{s}':=\tilde{\mu}'[-1,1]$. Up to interchanging the roles of $t$ and $t'$ or $\mu$ and $\mu'$, assume without loss of generality that $t\leq t'$ and $\tilde{s}\leq \tilde{s}'$. For any measure $\eta\in \M_+$ write $\Lambda(\eta)$ for the Poisson point process with intensity measure $\eta$. Through a slight abuse of notation, we denote by $(X_k)_{k\geq 1}$ and $(Y_k)_{k\geq 1}$ i.i.d.\@ samples from the probability measures induced by $\tilde{\mu}$ and $\tilde{\mu}'$, and let $\Pi_{\tilde{s}}$ and $\Pi_{\tilde{s}'-\tilde{s}}$ be Poisson random variables with means $\tilde{s}$ and $\tilde{s}'$, respectively. Through another abuse of notation, for each $u\in [0,1]$, we define the interpolating Hamiltonian $H_u$ on $\Sigma_1$ by
\begin{align*}
H_u(\epsilon):=-\epsilon \m &+ \sum_{x\in \Lambda(t\tilde{\nu})} \log(c+\Delta \epsilon x)+u\sum_{x\in \Lambda((t'-t)\tilde{\nu})}\log(c+\Delta \epsilon x)\\
&+\sum_{k\leq \Pi_{\tilde{s}}}\log\big(c+\Delta \epsilon (X_k+u(Y_k-X_k))\big)+ u\sum_{k\leq \Pi_{\tilde{s}'-\tilde{s}}}\log(c+\Delta \epsilon Y_k)
\end{align*}
as well as the random variable
\begin{equation*}
Z_u:=\frac{\int_{\Sigma_1} \epsilon \exp H_u(\epsilon)\ud P^*(\epsilon)}{\int_{\Sigma_1} \exp H_u(\epsilon)\ud P^*(\epsilon)}.
\end{equation*}
We have implicitly used that $t\leq t'$ and $\tilde{s}\leq \tilde{s}'$ to define $H_u$. The superposition principle ensures that $\smash{\Gamma_{t,\mu}(\nu)\stackrel{d}{=}Z_0}$ and $\smash{\Gamma_{t',\mu'}(\nu)\stackrel{d}{=}Z_1}$. If we write $\langle \cdot\rangle_u$ for the Gibbs measure associated with the Hamiltonian $H_u$, then
\begin{align*}
\frac{\ud}{\ud u} Z_u 
=& \sum_{k\leq \Pi_{\tilde{s}}} (Y_k-X_k) \bigg(\bigg\langle \frac{\epsilon}{c+\Delta \epsilon (X_k+u(Y_k-X_k))}\bigg\rangle_u - \langle \epsilon\rangle_u \bigg\langle \frac{1}{c+\Delta \epsilon (X_k+u(Y_k-X_k))}\bigg\rangle_u\bigg)\\
&\qquad+\sum_{x\in \Lambda((t'-t)\tilde{\nu})} \big(\langle \epsilon \log(c+\Delta \epsilon x)\rangle_u-\langle \epsilon \rangle_u\langle \log(c+\Delta \epsilon x)\rangle_u\big)\\
&\qquad\qquad+\sum_{k\leq \Pi_{\tilde{s}'-\tilde{s}}} \big(\langle \epsilon \log(c+\Delta \epsilon Y_k)\rangle_u-\langle \epsilon \rangle_u\langle \log(c+\Delta \epsilon Y_k)\rangle_u\big).
\end{align*}
It follows by the mean value theorem that for some constant $C<+\infty$ that depends only on $c$ and $\Delta$ whose value may not be the same at each occurrence,
\begin{align*}
\E\abs{Z_1-Z_0}&\leq C\bigg(\E\sum_{k\leq \Pi_{\tilde{s}}}\abs{Y_k-X_k} + \E\abs{\Lambda((t'-t)\tilde{\nu})} + \E \Pi_{\tilde{s}'-\tilde{s}}\bigg)\\
&\leq C\big(\mu[-1,1] \E\abs{Y_1-X_1}+\abs{t'-t}+\abs{\mu'[-1,1]-\mu[-1,1]}\big).
\end{align*}
Taking the infimum over all couplings between $(X_k)_{k\geq 1}$ and $(Y_k)_{k\geq 1}$ completes the proof. \qedhere
\end{enumerate}
\end{proof}

\begin{lemma}\label{l.t.small.unique.FP}
There exists a constant $C<+\infty$ such that for all $t< C^{-1}$ and $\mu \in \M_+$, the fixed point operator $\Gamma_{t,\mu}:\M_p\to \M_p$ defined in \eqref{e.intro.SBM.FP.operator} admits a unique fixed point $\nu_{t,\mu}\in \M_p$. Moreover, for all $(t,\mu), (t',\mu') \in \Rp\times \M_+$ with $t,t'< C^{-1}$, the fixed point map satisfies the Lipschitz bound
\begin{equation}
W(\nu_{t,\mu}, \nu_{t',\mu'})\leq C\big(\abs{t'-t}+\mu[-1,1]W\big(\bar \mu, \bar \mu'\big)+\abs{\mu[-1,1]-\mu'[-1,1]}\big).
\end{equation}
\end{lemma}

\begin{proof}
Let $C<+\infty$ be as in the statement of Lemma \ref{l.t.small.FP.operator.props}, and fix $L\in (0,1)$. For $t<LC^{-1}$, the Lipschitz bound \eqref{e.t.small.FP.operator.contraction} and the Banach fixed point theorem ensure that $\Gamma_{t,\mu}$ admits a unique fixed point $\nu_{t,\mu}\in \M_p$. Moreover, for any initial measure $\nu^0\in \M_p$, the Banach fixed point theorem implies that the sequence $(\nu_{t,\mu}^n)_{n\geq 1}$ defined iteratively by
\begin{equation*}
\nu_{t,\mu}^0 := \nu^0 \qquad \text{and} \qquad \nu_{t,\mu}^{n+1} :=\Gamma_{t,\mu}(\nu_{t,\mu}^n) \text{ for }n\geq 0
\end{equation*}
converges weakly to $\nu_{t,\mu}\in \M_p$, and satisfies the rate bound
\begin{equation*}
W\big(\nu_{t,\mu},\nu^0\big)\leq \frac{1}{1-L}W\big(\nu_{t,\mu}^1,\nu^0\big)=\frac{1}{1-L}W\big(\Gamma_{t,\mu}(\nu^0),\nu^0\big).
\end{equation*}
Applying this with $\nu^0:=\nu_{t',\mu'}$ and recalling that this is the fixed point of $\Gamma_{t',\mu'}$ reveals that
\begin{equation*}
W(\nu_{t,\mu}, \nu_{t',\mu'})\leq \frac{1}{1-L}W\big(\Gamma_{t,\mu}(\nu_{t',\mu'}),\nu_{t',\mu'}\big)=\frac{1}{1-L}W\big(\Gamma_{t,\mu}(\nu_{t',\mu'}),\Gamma_{t',\mu'}(\nu_{t',\mu'})\big)
\end{equation*}
Invoking the Lipschitz bound \eqref{e.t.small.FP.operator.Lip} and redefining the constant $C<+\infty$ completes the proof.
\end{proof}

\begin{proof}[Proof of Proposition \ref{p.intro.main.small.t}]
The Arzelà-Ascoli theorem gives a subsequence $(N_k)_{k\geq 1}$ along which the free energy converges to some limit $f:\Rp\times \M_+\to \R$. With the notation $\nu_{t,\mu}$ as in Lemma~\ref{l.t.small.unique.FP}, we first observe that if $f$ is Gateaux differentiable jointly in $(t,\mu)$ at some $(t,\mu) \in (0,C^{-1}) \times \mathrm{Reg}(\M_+)$, then one must have
\begin{equation}  
\label{e.Dmu.fixedpoint}
D_\mu f(t,\mu,\cdot) = G_{\nu_{t,\mu}}
\end{equation}
and
\begin{equation}  
\label{e.drt.fixedpoint}
\partial_t f(t,\mu) = \frac 1 2 \int_{-1}^1 G_{\nu_{t,\mu}} \, \ud \nu_{t,\mu}. 
\end{equation}
Indeed, the relation \eqref{e.Dmu.fixedpoint} follows from Lemmas~\ref{l.GD.of.limit.FE.as.GD.of.IC}, \ref{l.equivalence.of.FP.eqns}, and the uniqueness of the fixed point from Lemma~\ref{l.t.small.unique.FP}, while \eqref{e.drt.fixedpoint} follows from the explicit representation of $\partial_t \bar F_{N}$ in \eqref{e.SBM.en.FE.der.t}, Lemmas~\ref{l.limit.GD.measure.is.generator}, \ref{l.GD.of.limit.FE.as.GD.of.IC}, and \ref{l.equivalence.of.FP.eqns} which identify the limit law of the overlap array, and Proposition~\ref{p.regularity.cont.of.derivative} which ensures the convergence of $\partial_t \bar F_{N_k}(t,\mu)$ to $\partial_t f(t,\mu)$. 

Let us now fix $(t,\mu) \in (0,C^{-1})\times \M_+$, and for each $s \in \R$ sufficiently small, let $(t_s,\mu_s) \in \Rpp \times \mathrm{reg}(\M_+)$ be a point of joint Gateaux differentiability of $f$ such that 
\begin{equation*}  
|t - t_s| + |F_{\mu} - F_{\mu_s}|_{L^2} \le s^2.
\end{equation*}
Such $(t_s, \mu_s)$ exists since the set of points of differentiability of $f$ is dense. By Proposition~\ref{p.regularity.Lipschitz}, we have that as $s > 0$ tends to zero,
\begin{align*}  
s^{-1}(f(t+s,\mu) - f(t,\mu)) 
& = s^{-1} (f(t_s+ s,\mu_s) - f(t_s, \mu_s))  + O(s) \\
& \ge \frac 1 2 \int_{-1}^1 G_{\nu_{t_s,\mu_s}} \, \ud \nu_{t_s,\mu_s} + O(s),
\end{align*}
where we used the convexity of $f$ in the $t$ variable from Proposition~\ref{p.regularity.en.FE.convex} and \eqref{e.drt.fixedpoint} to obtain the last line. The continuity of the mapping $(t,\mu) \mapsto \nu_{t,\mu}$ from Lemma~\ref{l.t.small.unique.FP} therefore implies that 
\begin{equation*}  
\liminf_{s \downarrow 0} s^{-1}(f(t+s,\mu) - f(t,\mu)) \ge \frac 1 2 \int_{-1}^1 G_{\nu_{t,\mu}} \, \ud \nu_{t,\mu}.
\end{equation*}
In order to obtain the converse bound, we instead select $(t_s', \mu_s') \in \Rpp \times \mathrm{Reg}(\M_+)$ such that $(t_s' + s, \mu_s')$ is a point of differentiability of $f$ and 
\begin{equation*}  
|t - t'_s| + |F_{\mu} - F_{\mu'_s}|_{L^2} \le s^2.
\end{equation*}
Using the Lipschitz continuity and the convexity of $f$ once more, we can then write
\begin{align*}  
s^{-1}(f(t+s,\mu) - f(t,\mu)) 
& = s^{-1} (f(t'_s+ s,\mu'_s) - f(t'_s, \mu'_s))  + O(s) \\
& \le \frac 1 2 \int_{-1}^1 G_{\nu_{t_s'+s,\mu'_s}} \, \ud \nu_{t_s'+s,\mu'_s} + O(s).
\end{align*}
Using again the continuity of the map $(t,\mu) \mapsto \nu_{t,\mu}$, we therefore conclude that 
\begin{equation*}  
\lim_{s \downarrow 0} s^{-1}(f(t+s,\mu) - f(t,\mu)) = \frac 1 2 \int_{-1}^1 G_{\nu_{t,\mu}} \, \ud \nu_{t,\mu}.
\end{equation*}
The limit as $s$ tends to zero from the left can be handled similarly. We have thus shown that $f$ is differentiable in time everywhere on $(0,C^{-1}) \times \M_+$, with derivative given by \eqref{e.drt.fixedpoint}. Recalling also from \cite[Lemma~3.1]{dominguez2024mutual} that $f(0,\cdot) = \Psi$, we conclude that the function $f$ is determined uniquely as 
\begin{equation*}  
f(t,\mu) = \Psi(\mu) + \frac 1 2 \int_0^t \int_{-1}^1 G_{\nu_{s,\mu}} \, \ud \nu_{s,\mu} \, \ud s.
\end{equation*}
In particular, the identity of $f$ does not depend on the subsequence, and we can thus conclude that $\bar F_N$ in fact converges along the full sequence. Proposition~\ref{p.intro.main.small.t} then follows from Theorem~\ref{t.intro.main} and the uniqueness of the fixed point on $(0,C^{-1}) \times \M_+$. 
\end{proof}

We close this section by showing that the limit free energy is differentiable about the origin. The generalization of this result to the bipartite version of the stochastic block model will be used in Proposition \ref{p.bipartite.no.var.formula} to argue that the generalization of the variational formula \eqref{e.intro.main.disassortative} does not hold in this setting. 

\begin{corollary}\label{c.SBM.limit.FE.diff.at.0}
If $C<+\infty$ is as in the statement of Proposition~\ref{p.intro.main.small.t}, and $f:[0,C^{-1}]\times \M_+\to \R$ denotes the limit of the sequence of enriched free energies $(\bar F_N)_{N\geq 1}$ on $[0,C^{-1}]\times \M_+$, then the map $t\mapsto f(t,0)$ is differentiable at $t=0$.
\end{corollary}

\begin{proof}
We fix $t<C^{-1}$, and denote by $\nu_t\in \M_p$ the unique fixed point of the operator $\Gamma_{t,0}$. Observe that 
\begin{equation}\label{e.limit.FE.diff.0.key}
f(t,0)-f(0,0)=\Par_{t,0}(\nu_0)-\Par_{0,0}(\nu_0)+\Par_{t,0}(\nu_t)-\Par_{t,0}(\nu_0).
\end{equation}
On the one hand,
\begin{equation}\label{e.limit.FE.diff.0.t1}
\Par_{t,0}(\nu_0)-\Par_{0,0}(\nu_0)=\psi(t\nu_0)-\psi(0)-\frac{t}{2}\int_{-1}^1 G_{\nu_0}(y)\ud \nu_0(y).
\end{equation}
On the other hand, if $\nu_{u,t}:=(1-u)\nu_0+u\nu_t$, then the fundamental theorem of calculus implies that
\begin{align*}
\Par_{t,0}(\nu_t)-\Par_{t,0}(\nu_0)&=\int_0^1 D_\mu\Par_{t,0}(\nu_{u,t}; \nu_t-\nu_0)\ud u=\int_0^1 \int_{-1}^1 D_\mu\Par_{t,0}(\nu_{u,t},x)\ud (\nu_t-\nu_0)(x)\ud u.
\end{align*}
A direct computation reveals that for all $u\in [0,1]$ and $x\in [-1,1]$,
\begin{equation*}
D_\mu\Par_{t,0}(\nu_{u,t},x)=t\big(D_\mu\psi(t\nu_{u,t},x)-G_{\nu_{u,t}}(x)\big).
\end{equation*}
Together with the mean value theorem and the derivative bound on the density of the initial condition's Gateaux derivative implied by \eqref{e.SBM.extended.FE.Lip.der.bound}, this means that there exists a constant $L>0$ such that the map $f_{u,t}(x):= D_\mu\Par_{t,0}(\nu_{u,t},x)$ is Lipschitz continuous with $\norm{f_{u,t}}_{\mathrm{Lip}}\leq Lt$. It follows by definition of the Wasserstein distance \eqref{e.Wasserstein.def} and the Lipschitz continuity of the map $t\mapsto \nu_t$ established in Lemma \ref{l.t.small.unique.FP} that, up to redefining the constant $L$,
\begin{equation*}
\abs{\Par_{t,0}(\nu_t)-\Par_{t,0}(\nu_0)}\leq Lt^2.
\end{equation*}
Together with \eqref{e.limit.FE.diff.0.key}-\eqref{e.limit.FE.diff.0.t1}, this implies that the map $t\mapsto f(t,0)$ is differentiable at the origin with
\begin{equation*}
\partial_tf(0,0)=D_\mu\psi(0;\nu_0)-\frac{1}{2}\int_{-1}^1 G_{\nu_0}(y)\ud \nu_0(y).
\end{equation*}
This completes the proof.
\end{proof}

\section{Critical point selection}
\label{sec:CP_sel}

In this section, we argue that we do not expect the variational formula
\begin{equation}\label{e.CP_sel.sup.candidate}
\lim_{N\to +\infty}\bar F_N(t,\mu)=\sup_{\nu \in \M_p}\Par_{t,\mu}(\nu)
\end{equation}
to hold for the general sparse stochastic block model. More precisely, we will argue that it does not generalize to a bipartite version of the stochastic block model which we now describe. 

To motivate the definition of the bipartite version of the stochastic block model, consider a group of $N$ students indexed by the elements $\{1,\ldots,N\}$ and $N$ tutors also indexed by the elements $\{1,\ldots,N\}$. Each individual is assigned to an in-person group or a virtual group, depending on whether they prefer in-person or virtual classes. The idea is to match students preferring in-person instruction with tutors willing to deliver it, and similarly for students preferring virtual instruction. An allocation of the $2N$ students and tutors into the two groups can be identified with a vector
\begin{equation}
(x,y):=((x_1,x_2,\ldots,x_N), (y_1,y_2,\ldots,y_N))\in \Sigma_N^2,
\end{equation}
where $x \in \Sigma_N$ encodes the allocation of students, $y\in \Sigma_N$ denotes the allocation of tutors, the label $+1$ represents the group preferring in-person instruction, and the label $-1$ represents the group preferring virtual instruction. We will assume that the labels $X_i\sim P^X$ and $Y_i\sim P^Y$ are taken to be i.i.d.\@ from two probability measures $P^X$ and $P^Y$ on $\Sigma_1$. This means that the vector $(X, Y)$ is sampled from the product distribution
\begin{equation}
P_N^X\otimes P_N^Y:=(P^X)^{\otimes N}\otimes (P^Y)^{\otimes N}.
\end{equation}
Using the assignment vector $(X,Y)$, a random undirected bipartite graph $\mathbf{G}_N:=(G_{ij})_{i,j\leq N}$ with vertex sets $\{1,\ldots,N\}$ and $\{1,\ldots,N\}$ is constructed by stipulating that an edge between nodes $i,j\in \{1,\ldots,N\}$ in different groups is present with conditional probability
\begin{equation}\label{e.CP_sel.bip.G.prob.aN.bN}
\P\{G_{i,j}=1 \mid X, Y\}:=
\frac{c+\Delta X_iY_j}{N},
\end{equation}
where $c > 0$ and $\Delta \in (-c,c) \setminus \{0\}$ are fixed parameters. 
%
%
The likelihood of the model is given by
\begin{equation}
\P\big\{\mathbf{G}_N=(G_{ij})_{i,j\leq N} \mid X=x, Y=y\big\}=\prod_{i,j=1}^N \Big(\frac{c+\Delta X_iY_j}{N}\Big)^{G_{ij}}\Big(1-\frac{c+\Delta X_iY_j}{N}\Big)^{1-G_{ij}},
\end{equation}
so Bayes' formula implies that the posterior of the model is the Gibbs measure
\begin{equation}
\P\big\{X=x,Y=y \mid \mathbf{G}_N=(G_{ij})_{i,j\leq N}\big\}=\frac{\exp H_N(x,y)}{\int_{\Sigma_N^2} \exp H_N(x',y')\ud P_N^X(x')\ud P_N^Y(y')}
\end{equation}
associated with the Hamiltonian on $\Sigma_N^2$ defined by
\begin{equation}\label{e.CP_sel.bip.H}
H_N^\circ(x,y):=\sum_{i,j=1}^N \log\bigg[(c+{\Delta x_iy_j})^{G_{ij}}\Big(1-\frac{c +\Delta x_iy_j}{N}\Big)^{1-G_{ij}}\bigg].
\end{equation}
The free energy corresponding to this Hamiltonian is
\begin{equation}\label{e.CP_sel.bip.FE0}
\bar F_N^\circ:=\frac{1}{N}\E\log \int_{\Sigma_N^2}\exp  H_N^\circ(x,y) \ud P_N^X(x)\ud P_N^Y(y).
\end{equation}
As in Section \ref{sec:intro}, it will be convenient to modify the free energy \eqref{e.CP_sel.bip.FE0} without changing its limiting value, and then enrich it so that it depends on a time parameter $t\geq 0$ and measure parameters $\mu_1,\mu_2\in \M_+$. 
For each $t\geq 0$, we introduce a random variable $\smash{\Pi_t\sim \Poi (tN^2)}$ as well as an independent family of i.i.d.\@ random matrices $\smash{(G^k)_{k \geq 1}}$ each having conditionally independent entries $\smash{(G_{i,j}^k)_{i,j\leq N}}$ taking values in $\{0,1\}$ with conditional distribution
\begin{equation}\label{e.CP_sel.bip.G.distribution}
\P\big\{G_{i,j}^k=1 \mid X,Y\big\}:=\frac{c+\Delta X_{i}Y_j}{N}.
\end{equation}
Given a collection of random indices $\smash{(i_k,j_k)_{k\geq 1}}$ sampled uniformly at random from $\smash{\{1,\ldots,N\}^2}$, independently of the other random variables, we define the time-dependent Hamiltonian $H_N^t$ on $\Sigma_N^2$ by
\begin{equation}\label{e.CP_sel.bip.H.t}
H_N^t(x,y):=\sum_{k\leq \Pi_t}\log \bigg[\big(c+\Delta x_{i_k}y_{j_k}\big)^{G^k_{i_k,j_k}}\Big(1-\frac{c+\Delta x_{i_k}y_{j_k}}{N}\Big)^{1-G^k_{i_k,j_k}}\bigg],
\end{equation}
and denote by 
\begin{equation}
\bar F_N(t):= \frac{1}{N} \E \log \int_{\Sigma_N}\exp H_N^t(x,y)\ud P_N^X(x)\ud P_N^Y(y)
\end{equation}
its corresponding free energy. 
This is the Hamiltonian associated with the task of inferring the signal $(X,Y)$ from the data 
\begin{equation}
\D_N^t:=\big(\Pi_t, (i_k,j_k)_{k\leq \Pi_t},(G^k_{i_k,j_k})_{k\leq \Pi_t}\big).
\end{equation}
Given $\mu_1,\mu_2\in \M_+$, we consider sequences $(\Lambda_i(N\mu_1))_{i\leq N}$ and $(\Lambda_i(N\mu_2))_{i\leq N}$ of independent Poisson point processes with intensity measures $N\mu_1$ and $N\mu_2$, and we define the measure-dependent Hamiltonian $H_N^{\mu_1,\mu_2}$ on $\Sigma_N^2$ by
\begin{align}\label{e.CP_sel.bip.H.mu}
H_N^{\mu_1,\mu_2}(x,y):=\sum_{i\leq N}\sum_{z\in \Lambda_i(N\mu_1)}\log &\bigg[\big(c+\Delta x_{i}z\big)^{G^{z,X}_{i}}\Big(1-\frac{c+\Delta x_{i}z}{N}\Big)^{1-G^{z,X}_{i}}\bigg] \notag\\
&+\sum_{i\leq N}\sum_{z\in \Lambda_i(N\mu_2)}\log \bigg[\big(c+\Delta y_{i}z\big)^{G^{z,Y}_{i}}\Big(1-\frac{c+\Delta y_{i}z}{N}\Big)^{1-G^{z,Y}_{i}}\bigg],
\end{align}
where $(G_i^{z,X})_{i\leq N}$ and $(G_i^{z,Y})_{i\leq N}$ are conditionally independent random variables taking values in $\{0,1\}$ with conditional distributions
\begin{equation}
\P\{G_i^{z,X}=1\mid X, z\}:=\frac{c+\Delta X_i z}{N} \quad \text{and} \quad \P\{G_i^{z,Y}=1\mid Y, z\}:=\frac{c+\Delta Y_i z}{N}.
\end{equation}
This measure-dependent Hamiltonian is the Hamiltonian associated with the task of inferring the signal $(X,Y)$ from the data
\begin{equation}
{\D}_N^{\mu_1,\mu_2}:=\big(\Lambda_i(N\mu_1), (G_{i}^{z,X})_{z\in \Lambda_i(N\mu_1)}, \Lambda_i(N\mu_2), (G_{i}^{z,Y})_{z\in \Lambda_i(N\mu_2)}\big)_{i\leq N}.
\end{equation}
Finally, for each $(t,\mu_1,\mu_2)\in \Rp\times \M_+^2$, we introduce the enriched Hamiltonian
\begin{equation}\label{e.CP_sel.bip.en.H}
H_N^{t,\mu_1,\mu_2}(x,y):=H_N^t(x,y)+H_N^{\mu_1,\mu_2}(x,y)
\end{equation}
as well as its associated free energy
\begin{equation}\label{e.CP_sel.bip.en.FE}
\bar F_N(t,\mu_1,\mu_2):=\frac{1}{N}\E\log \int_{\Sigma_N^2}\exp H_N^{t,\mu_1,\mu_2}(x,y)\ud P_N^X(x)\ud P_N^Y(y).
\end{equation}
The Hamiltonian \eqref{e.CP_sel.bip.en.H} is associated with the inference of the signal $(X,Y)$ from the data
\begin{equation}\label{e.CP_sel.bip.en.data}
{\D}_N^{t,\mu_1,\mu_2}:=(\D_N^t, \D_N^{\mu_1,\mu_2}),
\end{equation}
where the randomness in these two datasets is taken to be independent conditionally on $(X,Y)$. As usual, we write $\langle \cdot \rangle$ for the Gibbs average associated with the Hamiltonian \eqref{e.CP_sel.bip.en.H}. This means that for any bounded and measurable function $f=f(x^1,y^1\ldots,x^n,y^n)$ of finitely many replicas, 
\begin{equation}\label{e.CP_sel.bip.Gibbs}
 \langle f\rangle   :=\frac{\int_{\Sigma_N^{2n}}f(x^1,y^1,\ldots,x^n,y^n)\prod_{\ell\leq n}\exp H_N^{t,\mu_1,\mu_2}(x^\ell,y^\ell)\ud P_N^X(x^\ell)\ud P_N^Y(y^\ell)}{\big(\int_{\Sigma_N^2} \exp H_N^{t,\mu_1,\mu_2}(x,y)\ud P_N^X(x)\ud P_N^Y(y)\big)^n}.
\end{equation}
The enriched free energy in \eqref{e.CP_sel.bip.en.FE} is the main object of study in the bipartite stochastic block model much in the same way as the enriched free energy \eqref{e.intro.SBM.en.FE} was the main object of study in the two-community stochastic block model.

To show that the variational formula \eqref{e.CP_sel.sup.candidate} does not generalize to the setting of the bipartite stochastic block model, we focus on the case with
\begin{equation}
P^X:=\Ber(1/2) \quad \text{and} \quad P^Y:=\Ber(p)
\end{equation}
for some $p\neq 1/2$. This means that the labels $X_i$ are centered while the labels $Y_i$ each have non-zero mean
\begin{equation}
\m:=2p-1\neq 0.
\end{equation}
Arguing as in Theorem \ref{t.intro.main}, it can be shown that, if the sequence of enriched free energies $(\bar F_N)_{N\geq 1}$ converges pointwise to some limit $f:\Rp\times \M_+^2\to \R$, then this limit free energy is given by evaluating the functional $\Par_{t,\mu_1,\mu_2}:\M_+^2\to \R$ defined by
\begin{equation}
\Par_{t,\mu_1,\mu_2}(\nu_1,\nu_2):=\psi_X(\mu_1+t\nu_1)+\psi_Y(\mu_2+t\nu_2)-t\int_{-1}^1\int_{-1}^1 g(xy)\ud \nu_1\ud \nu_2
\end{equation}
at some critical point $(\nu_1,\nu_2)\in \M_{1/2}\times \M_p$ of this functional. The initial conditions $\psi_X:\M_+\to \R$ and $\psi_Y:\M_+\to \R$ are given by \eqref{e.SBM.IC} for the priors $P^X$ and $P^Y$, respectively. The generalization of the variational formula \eqref{e.CP_sel.sup.candidate} would then be that for all $(t,\mu_1,\mu_2)\in \Rp\times \M_+^2$, we have
\begin{equation}\label{e.CP_sel.bip.sup.candidate}
\lim_{N\to +\infty}\bar F_N(t,\mu_1,\mu_2)=\sup_{\nu_1\in \M_{1/2}}\sup_{\nu_2\in \M_p}\Par_{t,\mu_1,\mu_2}(\nu_1,\nu_2).
\end{equation}
We now argue that this formula cannot hold for small $t > 0$.

\begin{proposition}\label{p.bipartite.no.var.formula}
If $\Delta>0$, $c>e$ and $\m$ is sufficiently small, then the variational formula \eqref{e.CP_sel.bip.sup.candidate} for $\mu_1=\mu_2=0$ does not hold in a neighborhood of $t=0$.
\end{proposition}

\begin{proof}
The strategy will be to show that the time derivative of the limit free energy at the origin, which can be shown to exist through similar arguments to those in the proof of Corollary \ref{c.SBM.limit.FE.diff.at.0}, does not agree with the time derivative implied by the variational formula \eqref{e.CP_sel.bip.sup.candidate}. Determining the time derivative of the enriched free energy \eqref{e.CP_sel.bip.en.FE} can be done through a similar computation to that in \cite[Lemma 2.1]{dominguez2024mutual}, and reveals that for all $(t,\mu_1,\mu_2)\in \Rp\times \M_+^2$,
\begin{equation*}
\partial_t\overline{F}_N(t,\mu_1,\mu_2)=\E \big(c+\Delta \langle x_iy_j\rangle\big)\log \big( c+\Delta \langle x_iy_j\rangle\big)-c+\BigO(N^{-1}).
\end{equation*}
The assumption that $X$ is centered and independent of $Y$ has played its part. Together with the fact that the Gibbs measure associated with the parameters $(t,\mu_1,\mu_2)=(0,0,0)$ is simply the average with respect to the prior $P^X\otimes P^Y$ and an analogue of Proposition \ref{p.regularity.cont.of.derivative}, this implies that the time derivative of the limit free energy $f:\Rp\times \M_+^2\to \R$ at the origin is given by
\begin{equation}\label{e.CP_sel.bip.lim.FE.der.t}
\partial_tf(0,0,0)=\E(c+\Delta \E (X_1Y_1))\log(c+\Delta \E (X_2Y_2))-c=c\log(c)-c.
\end{equation}
On the other hand, the validity of the variational formula \eqref{e.CP_sel.bip.sup.candidate} in a neighborhood of $t=0$ implies that for any $t>0$ in such a neighborhood,
\begin{equation*}
\frac{f(t,0,0)-f(0,0,0)}{t}\geq \frac{\Par_{t,0,0}(\delta_0, \delta_{\m})}{t}=\frac{\psi_X(t\delta_0)+\psi_Y(t\delta_{\m})-tg(0)}{t}=\frac{\psi_Y(t\delta_{\m})}{t},
\end{equation*}
where we have used that $f(0,0,0)=0$ and $\psi_X(t\delta_0)=tg(0)=t(c\log(c)-c)$. Letting $t$ tend to zero and recalling the expression \eqref{e.intro.IC.Gateaux.der} for the directional derivative of the asymptotic initial condition gives the lower bound
\begin{equation*}
\partial_tf(0,0,0)\geq D_\mu\psi_Y(0;\delta_{\m})=\E \langle c+\Delta \epsilon \m \rangle_* \log\langle c+\Delta \epsilon \m\rangle_* -c-\Delta \m^2.
\end{equation*}
Noticing that the Gibbs average $\langle \cdot\rangle_*$ is simply the average with respect to the prior $P^Y$ and Taylor-expanding the logarithm 
implies that
\begin{align*}
\partial_tf(0,0,0)&\geq (c+\Delta \m^2)\log(c+\Delta\m^2)-c-\Delta \m^2= c\log(c)-c+\Delta \m^2\log(c) +\BigO(\m^4).
\end{align*}
Whenever $\m$ is small enough, $c>e$ and $\Delta>0$, we therefore have
\begin{equation*}
\partial_tf(0,0,0)>c\log(c)-c.
\end{equation*}
This contradicts \eqref{e.CP_sel.bip.lim.FE.der.t} and completes the proof.
\end{proof}

\bibliographystyle{plain}
\bibliography{crit}
\end{document}